\documentclass[a4paper,11pt,reqno]{amsart}
\usepackage{amsmath, enumerate}
\usepackage{amsthm}
\usepackage{graphicx}
\usepackage{amssymb}
\usepackage{appendix}
\usepackage[utf8x]{inputenc}
\usepackage{fancyhdr}
\usepackage{calc}
\usepackage{url}

\usepackage[ top=3cm, bottom=3cm, left=2.1cm, right=2.1cm, centering]{geometry}

\usepackage[italian,english]{babel}
\usepackage[T1]{fontenc}
\usepackage[usenames,dvipsnames]{color}
\usepackage[overload]{empheq}

\makeatletter
\def\cleardoublepage{\clearpage\if@twoside \ifodd\c@page\else%
	\hbox{}%
	\thispagestyle{empty}
	\newpage%
	\if@twocolumn\hbox{}\newpage\fi\fi\fi}
\makeatother

\DeclareMathOperator{\sign}{sign}

\let\cleardoublepage\clearpage

\addto\captionsitalian{}

\newcommand{\R}{\mathbb{R}}
\newcommand{\norm}[1]{\left\lVert#1\right\rVert}

\newtheorem{thm}{Theorem}[section]
\newtheorem{cor}[thm]{Corollary}
\newtheorem{lem}[thm]{Lemma}
\newtheorem{pro}[thm]{Proposition}
\newtheorem{den}[thm]{Definition}
\newtheorem{oss}[thm]{Remark}

\newtheorem{rem}[thm]{Remark}
\numberwithin{equation}{section}

\begin{document}
	
	\title[An inhomogeneous PME with non-integrable data: asymptotics]{An inhomogeneous porous medium equation \\ with non-integrable data: asymptotics}
	
	\author{Matteo Muratori, Troy Petitt, and Fernando Quir\'os}
	
	\address{Matteo Muratori and Troy Petitt: Dipartimento di Matematica, Politecnico di Milano, Piazza Leonardo da Vinci 32, 20133 Milano (Italy)}
	\email{matteo.muratori@polimi.it}
	\email{troy.petitt@polimi.it}
	

	\address{Fernando Quir\'os: Departamento de Matem\'aticas, Universidad Aut\'onoma de Madrid, and  Instituto de Ciencias Matem\'aticas ICMAT (CSIC--UAM--UCM--UC3M), Campus de Cantoblanco, 28049 Madrid (Spain)}
	\email{fernando.quiros@uam.es}

	\makeatletter
	\@namedef{subjclassname@2020}{%
		\textup{2020} Mathematics Subject Classification}
	\makeatother
	

	
	\begin{abstract}
    We investigate the asymptotic behavior as $t\to+\infty$ of solutions to a weighted porous medium equation in $ \mathbb{R}^N $, whose weight {$\rho(x)$} behaves at spatial infinity like $ |x|^{-\gamma} $ with subcritical power, namely  $ \gamma \in [0,2) $. Inspired by some results \cite{AR,KU} from the 1980s on the unweighted problem, we focus on solutions whose initial data $u_0(x)$ are not globally integrable with respect to the weight and behave at infinity like $ |x|^{-\alpha} $, for $\alpha\in(0,N-\gamma)$. {In the special case $ \rho(x)=|x|^{-\gamma} $ and $ u_0(x)=|x|^{-\alpha} $ we show that} self-similar solutions of Barenblatt type, i.e.~reminiscent of the usual source-type solutions, still exist, although they are no longer compactly supported. Moreover, they exhibit a transition phenomenon which is new even for the unweighted equation. We prove that such self-similar {solutions} are attractors for {the original problem}, and convergence takes place globally in suitable weighted $ L^p $ spaces for $p\in[1,\infty)$ and even globally in $L^\infty$ under some mild additional regularity assumptions on the weight. Among the fundamental tools that we exploit, it is worth mentioning a global smoothing effect for non-integrable data.
	\end{abstract}
	
\keywords{Porous medium equation; asymptotic behavior; slowly decaying data; weighted Lebesgue spaces; smoothing effects; uniform convergence.}

	\maketitle
	
	
	\section{Introduction}
	We study intermediate asymptotics of non-integrable solutions to the Cauchy problem for the following \emph{weighted porous medium equation}
	\begin{equation}\label{wpme}
		\begin{cases}
			\rho\, u_t = \Delta (u^m) & \text{in } \mathbb{R}^N\times (0,+\infty) \, , \\
			u  =  u_0 & \text{on } \mathbb{R}^N\times \{0\} \, ,
		\end{cases}
	\end{equation}
	for $m>1$, $N\geq 3$, and a measurable weight $\rho\equiv\rho(x)$ behaving like the power $|x|^{-\gamma}$ for some $\gamma\in [0,2)$, in a sense to be made precise below. {Note that \eqref{wpme} models nonlinear diffusion phenomena taking place in an {inhomogeneous} medium, so it can also be referred to as the \emph{inhomogeneous porous medium equation}.} Since we will be able to treat \emph{sign-changing solutions}, we implicitly set $u^m:=|u|^{m-1}u$ as is standard practice. We will consider $u_0(x)$ behaving roughly like the power $|x|^{-\alpha}$ for $\alpha\in \left(0,N-\gamma\right)$, which is \emph{not globally integrable} against the weight.
	
For the sake of readability, we will always work in dimension $ N \ge 3 $, where it is possible to employ a strategy that covers all cases. However, most of our main results also hold in lower dimensions (up to some unavoidable technical restrictions on the parameters), except that one needs to suitably adjust a few passages in the proofs.	 For more details, we refer to Remark \ref{N=2} below.
	
	In order to study asymptotic properties of such solutions, it is first necessary to settle basic well-posedness issues beyond the global weighted $ L^1 $ setting, the latter having been treated in \cite{RV1}. To this end, existence and uniqueness of solutions to \eqref{wpme} for non-integrable (and even growing) initial data in the spirit of B\'enilan, Crandall, and Pierre's work \cite{BCP} have recently been established by the first two named authors in \cite{MP}, where it was shown that initial data of order at most
	\begin{equation}\label{O}
		O\!\left(|x|^{\frac{2-\gamma}{m-1}}\right)\qquad\textrm{as}\;\;|x| \to +\infty
	\end{equation}
	give rise to \emph{local-in-time} solutions that may blow up in finite time. Nonetheless, in the same paper, it was shown that solutions taking initial data {that grow} strictly slower than \eqref{O} actually exist \emph{globally} in time, which is our case. {For a more precise description of such results and the relation with the present framework, we refer to Subsection \ref{setup}}.
	
In the classical unweighted case (i.e.~$ \rho \equiv 1 $), the asymptotic behavior of $L^1$ solutions was first established in Friedman and Kamin's fundamental work \cite{FK} and later completed in \cite{Vaz-euc}. Likewise, the asymptotics of solutions to the inhomogeneous problem \eqref{wpme} has been completely settled when $u_0$ is nonnegative and globally integrable with respect to the weight $\rho$, stemming from the pioneering one-dimensional study \cite{KR}. In the subcritical range $\gamma\in(0,2)$, Reyes and V\'{a}zquez in \cite{RV1, RV2} obtained corresponding convergence results to a suitable perturbation of the standard Barenblatt solution. Such solutions have the explicit form
	\begin{equation}\label{int-baren}
		U_B(x,t) =t^{-\lambda}\left(k_1-k_2\,t^{-\theta\lambda}\,|x|^{2-\gamma}\right)_+^{\tfrac{1}{m-1}},
	\end{equation}
	for computable constants $k_i>0$, where
	\begin{equation}\label{def-lambda-theta}
		\lambda =\frac{N-\gamma}{(N-\gamma)(m-1)+2-\gamma} \, , \qquad \theta =\frac{2-\gamma}{N-\gamma} \, ,
	\end{equation}
	are special self-similarity exponents that are also highly relevant to the works \cite{BCP,MP}.

For supercritical exponents $\gamma>2$ on the other hand, Kamin, Reyes, and V\'{a}zquez in \cite{KRV} showed that solutions behave asymptotically like a separate-variable profile where the spatial part satisfies a suitable semilinear elliptic equation. The critical case $\gamma=2$, for a globally bounded weight, has been addressed by Nieto and Reyes in \cite{NR}, where they prove convergence to a (logarithmic-type) Barenblatt solution with a singularity at the origin. When $\rho $ equals the pure power $ |x|^{-2}$, Iagar and S\'{a}nchez in \cite{IS1} proved some further interesting convergence results. In these works and others it has been established that the asymptotic behavior of solutions changes dramatically at $\gamma=2$ for all dimensions (see \cite{IS3} for $N=1,2$) and all $m\geq1$ (for the weighted heat equation see \cite{IS2}).
	
	A common approach in these results is to rescale both the weight and the initial datum, and as long as these objects suitably converge to positive multiples of pure powers along rescaling sequences -- see for example \eqref{weight-cond2} and \eqref{conv-resc-datum} below -- one may hypothesize that solutions of \eqref{wpme} converge as $ t \to +\infty $ to the {unique} solution of the following singular weighted problem:
	\begin{equation}\label{singularWPME}
		\begin{cases}
			c  \left|x\right|^{-\gamma} \partial_t \, \mathcal{U}_\alpha  = \Delta (\mathcal{U}_\alpha^m) & \text{in } \mathbb R^N\times (0,+\infty) \, , \\
		\mathcal{U}_\alpha =  b \left|x\right|^{-\alpha} & \text{on } \mathbb R^N\times \{ 0 \} \, ,
		\end{cases}
	\end{equation}
 where $ b,c>0 $. Indeed, it is true, although perhaps less well-known, that asymptotic results can also be obtained for non-integrable solutions using a similar scaling method. {In this direction, we quote~\cite{AR} by Alikakos and Rostamian, and second, \cite{KU} by Kamin and Ughi, both of which study the $\rho\equiv1$ framework. We also mention \cite{KP-0,KP}, where similar techniques were applied to the heat equation and the porous medium equation with \emph{absorption}.}

 Since \cite{AR} is our closest precedent, it is worth discussing their results in more detail. The authors prove that every nonnegative initial datum $u_0\in L^1_{\mathrm{loc}}\!\left(\R^N\right)$ behaving like $|x|^{-\alpha}$ as $ |x| \to +\infty $, for some $\alpha \in \left(-{2}/{(m-1)},N\right)$, gives rise to a solution $u$ of \eqref{wpme} {with $\rho \equiv 1$} that converges to the solution $\mathcal{U}_\alpha$ of \eqref{singularWPME} in the sense
	\begin{equation*}
		\lim_{t\to + \infty} t^{\lambda_\alpha} \left| u(x,t)-\mathcal{U}_\alpha(x,t) \right|=0 \, ,
	\end{equation*}
	uniformly on expanding sets of the type $ \left\{ |x|\leq C\, t^{\lambda_\alpha/\alpha} \right\}$, where $C>0$ is arbitrary and $\lambda_\alpha$ is explicitly defined in \eqref{lambdaa} (with $\gamma=0$ for this exposition).
By use of the celebrated Aronson-Caffarelli estimate \cite{AC} for continuous \emph{nonnegative} solutions, the authors also prove that such a condition on $u_0$ is optimal for their convergence results. We stress that this delicate estimate is not currently known in our weighted framework, and only applies to nonnegative solutions. We finally mention \cite{KU} where, among several results, the critical case $ \alpha = N $ was treated showing suitable convergence to the classical Barenblatt solution under a nonstandard scaling transformation. 

In the present paper, we manage to extend many of the above results to our weighted framework and actually improve them in a few different directions. The presence of a nontrivial weight $\rho$ implies several technical difficulties that we are able to overcome. In particular, the Aronson-Caffarelli estimate, the Aronson-B\'enilan inequality and \emph{a priori} H\"{o}lder regularity estimates (down to the origin) are in general not available. Furthermore, in contrast to \cite{AR} and \cite{KU}, we \emph{do not} need our solutions to be everywhere \emph{nonnegative}, although the assumptions we make on the initial data certainly imply that they are ``mostly'' positive at infinity. We further require that the initial datum $u_0$ and its rescalings are uniformly bounded in a new normed space defined through \eqref{norm-trasl}. This norm seems to be the natural weighted counterpart of the translation-invariant object introduced in \cite[Proposition 1.3]{BCP}, guaranteeing via a \emph{smoothing effect} that solutions are \emph{globally bounded} away from $t=0$; to this end, we refer to Proposition \ref{p-transl}, which may be of independent interest. Under these assumptions we are able to prove  asymptotic convergence to the solution of \eqref{singularWPME} which, as in the well-established case of~\eqref{int-baren}, has a self-similar structure, but this time of the type
$$
\mathcal{U}_\alpha(x,t)=t^{-\lambda_\alpha}g_\alpha\!\left( t^{-\frac{\lambda_\alpha}{\alpha}}|x|\right) .
$$
The smooth radial profile $ r \mapsto g_\alpha(r)$ turns out to satisfy a nonlinear
ODE that we study independently in Section \ref{ODE} because it had never been addressed before with the level of detail we need here. Such solutions $ \mathcal{U}_\alpha $ exhibit a few fundamental differences with respect to the standard Barenblatt solutions, besides no longer being explicit. First, the profile is strictly positive in the whole $\R^N $, and moreover in an explicit range of $\alpha$ close to $0$ it is \emph{everywhere decreasing} in time. {This is} in sharp contrast with the wave-like propagation property {described in Remark \ref{dich-rem}, which is characteristic} of compactly-supported Barenblatt profiles. The latter appears to be a completely new phenomenon.

\medskip

We prove two main asymptotic results. First, in Theorem \ref{main thm} we establish convergence in suitable \emph{global weighted $L^p$ spaces} for $p\in[1,\infty)$; then, in Theorem \ref{thm-unif-conv} we prove \emph{global $L^\infty$ convergence} under some additional (mild) regularity requirements on the weight $\rho$. We stress that both of these main theorems are new even in the unweighted case, since the results of \cite{AR} were purely local (in rescaled variables). Moreover, they are the first asymptotics results for non-globally-integrable (and possibly sign-changing) solutions of a weighted porous medium equation.
	
	
\section{Preliminary material and statements of the main results}\label{prelim}
	
In the following, we make some fundamental assumptions on the data and introduce basic functional quantities and definitions that will be used throughout. Finally, at the end of the section, we state the main results of the paper.

\subsection{Conditions on the weight}
\label{cond-1}
As in \cite{MP}, we require that the weight (or density) $\rho$ be measurable and satisfy the pointwise bounds
\begin{equation}\label{weight-cond}
\underline{C} \left( 1 + |x|  \right)^{-\gamma}  \le \rho(x) \le \overline{C} \left| x \right|^{-\gamma} \qquad \text{for a.e. } x \in \mathbb{R}^N \, ,
\end{equation}
for some $\gamma\in [0,2)$ and ordered constants $ \underline{C} , \overline{C} >0$. Additionally, we require that at infinity it behaves \emph{precisely} like the power $ |x|^{-\gamma} $, in the sense that there exists $ c > 0 $ such that, for every sequence $ \xi_k \to +\infty $, it holds
	\begin{equation}\label{weight-cond2}
		\lim_{k \to \infty} \left\| \rho_k - c \, |x|^{-\gamma}  \right\|_{L^1_{\mathrm{loc}} \left( \mathbb{R}^N \right)} = 0 \, ,
	\end{equation}
	where we introduce the rescaled density
\begin{equation}\label{rescaled-weight}
\rho_k(x) :=  \xi^\gamma_k \, \rho\!\left(\xi_k x\right) .
\end{equation}	
Assumption \eqref{weight-cond2} may be equivalently written as
	\begin{equation*}\label{weight-cond2-bis}
		\lim_{\xi \to +\infty} \xi^{\gamma-N} \int_{B_{\xi R}} \left| \rho(y) - c \, |y|^{-\gamma} \right| dy = 0 \qquad \forall R>  0 \, ,
	\end{equation*}
where $ B_r $ is the ball of radius $r>0$ centered at the origin. It is not difficult to check that, under~\eqref{weight-cond}, condition \eqref{weight-cond2} is implied (for instance) by
	\begin{equation}\label{weight-cond3}
	\underset{|x| \to +\infty}{\operatorname{ess}\lim} \, |x|^{\gamma} \, \rho(x) = c \, .
	\end{equation}		
Note that the rescaled density $ \rho_k $, for all $ k $ so large that $ \xi_k \ge 1 $, still satisfies \eqref{weight-cond} with the same constants:
\begin{equation}\label{weight-cond-scaled}
\underline{C} \left( 1 + |x|  \right)^{-\gamma}  \le \rho_k(x) \le \overline{C} \left| x \right|^{-\gamma} \qquad \text{for a.e. } x \in \mathbb{R}^N \, .
\end{equation}

	\subsection{A key functional space}\label{kse}
	For a given  $ \rho $ satisfying \eqref{weight-cond} and any $ f \in L^1_{\mathrm{loc}}\!\left( \mathbb{R}^N , \rho \right) $, let us introduce the following norm:
	\begin{equation}\label{norm-trasl}
		\left\| f \right\|_{0,\rho} := \sup_{\substack{ R \ge 1 \\[0.4mm] z_R \in \partial B_{R} }} R^{-\frac{\gamma(N-2)}{2}} \int_{B_{R^{{\gamma}/{2}}}(z_R)} \left|f(x)\right| \rho(x) \, dx \, ,
	\end{equation}
where $ B_r(z) $ is the ball of radius $ r>0 $ centered at $ z \in \mathbb{R}^N $, and we omit the argument when $ z \equiv 0 $. Note that the set of all functions $ f \in L^1_{\mathrm{loc}}\!\left( \mathbb{R}^N , \rho \right) $ such that $ \left\| f \right\|_{0,\rho} < +\infty $ is a Banach space in which both $ L^1\!\left( \mathbb{R}^N , \rho \right) $ and $ L^\infty\!\left( \mathbb{R}^N \right) $ are strictly contained.

In the unweighted case, {i.e.}~$ \gamma=0 $, we observe that $ \| f \|_{0,\rho} $ is equivalent to the \emph{translation-invariant} norm
	\begin{equation}\label{norm-trasl-unw}
	\sup_{z \in \mathbb{R}^N} \int_{B_{1}(z)} \left|f(x)\right| dx \, ,
	\end{equation}
	which was defined in \cite[Proposition 1.3]{BCP}. The latter proved to be the right tool in order to ensure \emph{global boundedness} of the solutions to \eqref{wpme} (when $ \rho \equiv 1 $) for a class of initial data larger than $ L^1 $. However, because the introduction of $ \rho $ breaks translation invariance, it is not obvious {\emph{a priori}} what the analogue of \eqref{norm-trasl-unw} should be. As we will see in Subsection \ref{smooth}, the norm \eqref{norm-trasl} represents the correct answer; a crucial point to this lies in the fact that a usual \emph{cut-off} function supported in $ B_{R^{{\gamma}/{2}}}(z_R) $ has a Laplacian that behaves like $ R^{-\gamma} $ for $ R $ large, exactly the same decay as $ \rho $.

\subsection{Assumptions on the initial data}
First of all we require that $ u_0 \in L^1_{\mathrm{loc}}\!\left(\mathbb{R}^N ,\rho \right)$. Besides that, the kind of initial data we have in mind are those that essentially behave like $ |x|^{-\alpha} $ as $ |x| \to +\infty $, in the sense that
	\begin{equation}\label{datum-cond3}
	\underset{|x| \to +\infty}{\operatorname{ess}\lim} \, |x|^{\alpha} \, u_0(x) = b \, .
	\end{equation}
In order to study convergence in $L^p$ spaces when $ p \in [1,\infty) $, we need to introduce a suitable weight $ \Phi : \mathbb{R}^N \to \mathbb{R}^+ $  so that the solution $u$ belongs to the weighted space $C\big((0,+\infty);L^p\big(\mathbb{R}^N,\Phi|x|^{-\gamma}\big) \big)$. In this way, we can give a rigorous meaning to the statements of our main asymptotic results. However, not all weights are adequate for this purpose.  Given $ \alpha \in (0,N-\gamma) $, we say that a weight is \emph{admissible} if it satisfies the following conditions:
\begin{gather}
\Phi \in C^2\!\left( \mathbb{R}^N \right) \cap L^\infty\!\left(\mathbb{R}^N\right) , \label{A1}  \\  \left| \nabla \Phi(x) \right| \le \frac{K}{\left(1+|x|\right)^{\gamma-1}} \quad \text{and}  \quad  \left| \Delta \Phi (x) \right| \le K \, \frac{\Phi(x)}{\left(1+|x|\right)^\gamma} \qquad \forall  x \in \mathbb{R}^N \, , \label{A2} \\
 \int_{\mathbb{R}^N} \frac{\Phi(x)}{|x|^{\alpha + \gamma}} \, dx < + \infty  \, , \label{A3}
\end{gather}
for some constant $K>0$. We emphasize that \eqref{A3} is unavoidable as long as we deal with decay rates of the type $ |x|^{-\alpha} $, whereas \eqref{A2} is required for purely technical reasons. A significant example of an admissible weight is
\begin{equation}\label{spec-choice}
\Phi(x) = \left( 1 + |x|^2 \right)^{-\frac{N+\varepsilon-\alpha-\gamma}{2}} ,
\end{equation}
for any $ \varepsilon>0 $.

Once we have introduced the weights $\Phi$, we allow the initial data to satisfy \eqref{datum-cond3} only in a related integral sense. Thus, instead of~\eqref{datum-cond3} we require that there exist $ b>0 $ and an admissible weight $ \Phi $ such that, for every sequence $ \xi_k \to +\infty $, it holds
\begin{equation}\label{conv-resc-datum}
\lim_{k \to \infty} \int_{\mathbb{R}^N} \left| \xi^\alpha_k \, u_0\!\left(\xi_k x\right) - b \left| x \right|^{-\alpha} \right| \Phi(x) \, \rho_k(x) \, dx = 0 \, ,
\end{equation}
or equivalently
\begin{equation}\label{conv-resc-datum-bis}
\lim_{\xi  \to +\infty} \xi^{\alpha+\gamma-N} \int_{\mathbb{R}^N} \left| u_0\!\left( y \right) - b \left| y \right|^{-\alpha} \right| \Phi\!\left( \tfrac{y}{\xi} \right) \rho(y) \, dy = 0 \, .
\end{equation}
In other words, we are assuming that the \emph{rescaled} initial datum
	\begin{equation}\label{rescaled-datum}
u_{0k}(x) := \xi^\alpha_k \, u_0\!\left(\xi_k x\right)
\end{equation}
approximates the \emph{singular} datum $ |x|^{-\alpha} $ (up to constants) as $ k \to \infty $, and the accuracy of such an approximation is measured via the weight $ \Phi $. Clearly, if \eqref{datum-cond3} holds then \eqref{conv-resc-datum} is always satisfied, for any admissible weight (see Proposition \ref{datum} for more details). However, it is apparent that \eqref{conv-resc-datum} permits $u_0$ to deviate much more from $ b \, |x|^{-\alpha} $ than \eqref{datum-cond3} does.

In addition to the asymptotic assumption \eqref{conv-resc-datum}, we will also require that
\begin{equation}\label{second-req-u0}
\sup_{k \in \mathbb{N}} \left\| u_{0k} \right\|_{0,\rho_k} < +\infty \, ,
\end{equation}
for every sequence $ \xi_k \to +\infty $. This turns out to be crucial in order to ensure \emph{global boundedness} of the solutions  to \eqref{wpme}. After some standard manipulations, it is not difficult to show that \eqref{second-req-u0} is equivalent to
\begin{equation}\label{second-req-u0-bis}
\limsup_{\xi \to +\infty} \, \xi^{\alpha-\frac{N(2-\gamma)}{2}} \sup_{\substack{ R \ge \xi \\[0.4mm] z_R \in \partial B_{R} }} R^{-\frac{\gamma(N-2)}{2}} \int_{B_{\xi^{{(2-\gamma)}/{2}} R^{ \gamma / 2} }(z_R)} \left| u_0(y) \right| \rho(y) \, dy  < + \infty \, .
\end{equation}
Although \eqref{second-req-u0-bis} may appear complicated, it is nothing but a way to  control quantitatively the behavior of the (weighted) integral of $ \left| u_0(y) \right| $ on balls of growing radius, uniformly with respect to their center. The finiteness of the $ \limsup $ in \eqref{second-req-u0-bis} amounts to requiring that such a behavior is not worse than the one achieved by $ |y|^{-\alpha} $: indeed, a simple computation, which is carried out in the proof of Proposition~\ref{datum}, yields
\begin{equation*}\label{asymp-pow}
\sup_{\substack{ R \ge \xi \\[0.4mm] z_R \in \partial B_{R} }} R^{-\frac{\gamma(N-2)}{2}} \int_{B_{\xi^{{(2-\gamma)}/{2}} R^{ \gamma / 2} }(z_R)} |y|^{-\alpha-\gamma} \,  dy \sim \xi^{\frac{N(2-\gamma)}{2} - \alpha} \qquad \text{as } \xi \to + \infty \, .
\end{equation*}
Finally, for our \emph{uniform convergence} results, it will be crucial to reinforce \eqref{second-req-u0} by requiring
\begin{equation}\label{second-req-u0-re}
\lim_{k \to \infty}\left\| u_{0k} - b \, |x|^{-\alpha} \right\|_{0,\rho_k} =0
\end{equation}
for some $ \alpha \in (0,N-\gamma) $ and $ b>0 $ or, equivalently,
\begin{equation*}\label{second-req-u0-bis-re}
\lim_{\xi \to +\infty} \, \xi^{\alpha-\frac{N(2-\gamma)}{2}} \sup_{\substack{ R \ge \xi \\[0.4mm] z_R \in \partial B_{R} }} R^{-\frac{\gamma(N-2)}{2}} \int_{B_{\xi^{{(2-\gamma)}/{2}} R^{ \gamma / 2} }(z_R)} \left| u_0(y) - b|y|^{-\alpha} \right| \rho(y) \, dy  = 0  \, .
\end{equation*}

\begin{rem}\rm
	It is natural to ask if the requirements  \eqref{conv-resc-datum} and \eqref{second-req-u0} on the initial datum  are to some extent redundant. This is indeed the case  for certain admissible weights $\Phi$ and for certain $\alpha\in(0,N-\gamma)$. First of all, let us assume that $ \Phi $ is of the type \eqref{spec-choice} for some $\varepsilon>0$ to be chosen later. Using
	\begin{equation*}
		\left|u_{0k}(x)\right| \leq \left|u_{0k}(x)-b\, |x|^{-\alpha} \right|+b \, |x|^{-\alpha}
	\end{equation*}
	and noticing that $ \norm{|x|^{-\alpha}}_{0,\rho_k} $ is plainly bounded, in order to prove that $\{u_{0k}\}$ satisfies \eqref{second-req-u0} under \eqref{conv-resc-datum}, it is enough to concentrate on the first term on the right-hand side. To this end, we also point out the simple estimate
	\begin{equation*}\label{rem lowerbound}
		C\,R^{-N-\varepsilon+\alpha+\gamma}\leq \Phi(x) \qquad \text{in } B_{R^{\gamma / 2}}\!\left(z_R\right) ,
	\end{equation*}
	for some $C>0$ independent of $R\geq1$. Then
	\begin{align*}
		\int_{B_{R^{{\gamma}/{2}}}(z_R)} \left|u_{0k} -b \, |x|^{-\alpha}\right| \rho_k \, dx & \leq \frac{R^{N+\varepsilon-\alpha-\gamma}}{C} \int_{B_{R^{{\gamma}/{2}}}(z_R)} \left|u_{0k} -b\,|x|^{-\alpha}\right| \Phi \, \rho_k \, dx \\
		&\leq \frac{R^{N+\varepsilon-\alpha-\gamma}}{C} \int_{\R^N} \left|u_{0k} -b\,|x|^{-\alpha}\right| \Phi \, \rho_k \, dx \, .
	\end{align*}
	Therefore, recalling  \eqref{norm-trasl}, we can infer:
	\begin{align*}
		\sup_{k \in \mathbb{N}} \left\| u_{0k}  -  b \, |x|^{-\alpha} \right\|_{0,\rho_k} \leq C^{-1} \sup_{k \in \mathbb{N}} \sup_{ R \ge 1 } R^{N+\varepsilon-\alpha-\gamma-\frac{\gamma(N-2)}{2}}\int_{\R^N} \left|u_{0k}-b \, |x|^{-\alpha}\right| \Phi \, \rho_k \, dx \, ,
	\end{align*}
	so that, thanks to \eqref{conv-resc-datum}, the right-hand side  is finite provided
	$$
 \tfrac{2-\gamma}{2} \, N  <  \alpha < N-\gamma 	\qquad \text{and} \qquad 0 < \varepsilon \le  \alpha -  \tfrac{2-\gamma}{2} \, N  \, .
	$$
On the other hand, it is not difficult to check that for any $ \alpha \in \big( 0 , \tfrac{2-\gamma}{2} N \big) $ the function
$$
u_0(x) = b \left| x \right|^{-\alpha} + \sum_{j=1}^\infty \chi_{B_{r_j}(z_j)}(x) \quad \text{with } |z_j|= e^j \text{ and } r_j = j e^{\frac{\gamma}{2}j}
$$	
satisfies \eqref{conv-resc-datum} for all weights of the type \eqref{spec-choice} but does not comply with \eqref{second-req-u0}, so in general the implication \eqref{conv-resc-datum} $\Rightarrow$ \eqref{second-req-u0} need not be true. As for the opposite implication, it is enough to take $ u_0(x)=b' \, |x|^{-\alpha} $ with $ b' \neq b $ or even a compactly supported datum to see that  \eqref{second-req-u0}  holds but \eqref{conv-resc-datum} fails.
\end{rem}
			
\subsection{The non-integrable self-similar solution}\label{self-sim}		
First of all, let us define an important exponent that will be used frequently in the sequel: for any given $ \alpha \in (0,N-\gamma) $,
	\begin{equation}\label{lambdaa}
		\lambda_\alpha=\frac{\alpha}{\alpha(m-1)+2-\gamma} \, .
	\end{equation}
Our main goal is to prove that, if the density $ \rho(x) $ asymptotically behaves like $ |x|^{-\gamma} $ and the initial datum $ u_0(x) $ asymptotically behaves like $ |x|^{-\alpha} $, then the corresponding solution of \eqref{wpme} suitably converges to the solution of the singular problem \eqref{singularWPME}. In the next theorem, which will be proved in Section \ref{ODE}, we see that such a solution has a special \emph{self-similar} shape, reminiscent of the classical \emph{Barenblatt solutions} \eqref{int-baren} (but with some substantial differences). Its construction relies entirely on the resolution of the following ODE problem:
	\begin{equation}\label{ode1}
		\begin{cases}
			(g_\alpha^m)^{''}(r)+\frac{N-1}{r} \, (g_\alpha^m)^{'}(r)+c \, r^{-\gamma}\left[\frac{\lambda_\alpha}{\alpha} \, r \, g_\alpha'(r)+\lambda_\alpha \, g_\alpha(r)\right]=0 & \text{for } r>0 \, , \\
			g_\alpha'(r) = o\!\left(r^{-\frac \gamma 2}\right) & \text{as } r \to 0^+ , \\
			g(r)>0 & \text{for } r \ge 0 \, , \\
			\lim_{r \to +\infty}r^{\alpha} \, g_\alpha(r)=b \, .
		\end{cases}
	\end{equation}
The detailed well-posedness study of \eqref{ode1} will also be carried out in Section \ref{ODE}.

\begin{thm}[Non-integrable self-similar solutions]\label{selfsim-sol}
Let $N\geq3$, $m>1$, $ \gamma \in [0,2) $, $ \alpha \in (0,N-\gamma) $ and $b,c>0$. Let $  g_\alpha \in C^2((0,+\infty)) \cap C([0,+\infty)) $ be the solution of \eqref{ode1}. Then the unique solution of problem \eqref{singularWPME} is
	\begin{equation}\label{Barenblatt}
		\mathcal{U}_\alpha(x,t)=t^{-\lambda_\alpha} g_\alpha\!\left(t^{-\frac{\lambda_\alpha}{\alpha}} |x|\right) ,
	\end{equation}
which satisfies
	\begin{equation}\label{u-alpha-1}
\mathcal{U}_\alpha \in AC_{\mathrm{loc}}\!\left([0,+\infty);L^1_{\mathrm{loc}}\!\left(\R^N,|x|^{-\gamma}\right)\right)
\end{equation}
and the two-sided estimate
\begin{equation}\label{u-alpha-2}
\frac{c_1 \, b }{ b^{(m-1)\lambda_\alpha} \, t^{\lambda_\alpha} +  |x|^\alpha}\le \mathcal{U}_\alpha(x,t) \le \frac{c_2 \, b }{ b^{(m-1)\lambda_\alpha} \, t^{\lambda_\alpha} +  |x|^\alpha} \, ,
\end{equation}
for some constants $ c_1,c_2>0 $ depending only on $ N,m,\gamma,\alpha,c $. Moreover, the following dichotomy holds:
		\begin{itemize}
		
\item if $ \alpha \in \left( 0 , \tfrac{N-2}{m} \right]  $ then $ \partial_t\,\mathcal{U}_\alpha < 0 $ in $ \mathbb{R}^N \times (0,+\infty) $;

\item if $ \alpha \in \left( \tfrac{N-2}{m} , N-\gamma \right)  $ then there exists $ r^\ast \in (0,+\infty) $ such that
$$
\partial_t \, \mathcal{U}_\alpha < 0 \quad \text{in } \left\{ |x| < r^\ast  \, t^{\frac{\lambda_\alpha}{\alpha}} \right\} \qquad  \text{and} \qquad \partial_t \, \mathcal{U}_\alpha > 0 \quad \text{in } \left\{ |x| > r^\ast  \, t^{\frac{\lambda_\alpha}{\alpha}} \right\} .
$$
		\end{itemize}
\end{thm}

The concept of solution to \eqref{singularWPME}, or the more general \eqref{wpme}, is understood in the sense of \cite{MP}, and will be made clear in the next subsection. The uniqueness of $ \mathcal{U}_\alpha $, claimed in the above theorem, holds in a wide class of solutions as will be summarized in the next subsection. Also, let us observe that formula \eqref{Barenblatt} implies, in particular, that $ \mathcal{U}_\alpha $ enjoys the \emph{self-similar} property
	\begin{equation}\label{scale invar}
		\mathcal{U}_\alpha(x,t)=\xi^\alpha  \, \mathcal{U}_\alpha\!\left( \xi x, \xi^{\frac{\alpha}{\lambda_\alpha}} \, t \right)
	\end{equation}
	for all $ \xi > 0 $, which will be of crucial importance to our main purposes.

	\begin{rem}[A noteworthy transition effect]\rm \label{dich-rem}
		The dichotomy stated in Theorem \ref{selfsim-sol} is to our knowledge new even for the unweighted equation, and is an example of a \emph{transition phenomenon} in the following sense. For $\gamma=0$, according to \cite{FK, Vaz-euc}, when the initial datum is globally integrable, the asymptotic attractor for the problem is the  Barenblatt solution \eqref{int-baren}. The same is true for $ \gamma \in (0,2) $ as shown in \cite{RV2}. It can be calculated that the time derivative of such a solution has the same qualitative behavior as $\partial_t \, \mathcal{U}_\alpha$ for $\alpha\in\left( \tfrac{N-2}{m} , N-\gamma \right)$. That is, as $\alpha$ approaches the  threshold value $N-\gamma$ from below, it is still non-integrable but its behavior resembles more and more that of the integrable Barenblatt solution.
		
Next, notice that the critical value $\alpha=\frac{N-2}{m}$ does not depend on the weight. This value seems to be strictly related to (diffusive) properties of the Laplacian in $\R^N$; as a preliminary observation, we point out that the stationary function $ \hat{u}(x,t)=|x|^{-\alpha}$ is a weak \emph{supersolution} of the differential equation in \eqref{wpme}
		if and only if $\alpha\in \left(0,\tfrac{N-2}{m}\right]$. {In this regard}, see also Remark \ref{purepowerode}.  Clearly, here $\rho$ can be any positive function, so it plays no role. Furthermore, still in the critical case $ \alpha = \tfrac{N-2}{m} $ we have that $ \hat{u}^m $ is (up to constants) exactly the \emph{Green's function} associated with $ -\Delta $.
		
Finally, as a heuristic interpretation, if $ \alpha > \tfrac{N-2}{m} $ the solution is sufficiently small at infinity to support a wave-like behavior: that is, fixing a point away from the origin, it is increasing there up to a certain time, after which it becomes monotone decreasing (the wave passes by the point). If $\alpha\in\big(0,\tfrac{N-2}{m}\big]$ instead, the solution decays so slowly at infinity that this short-time increase is blocked.
	\end{rem}

	\subsection{Well-posedness and setup for the asymptotic analysis}\label{setup}
	Before stating our main results, we introduce the following notion of solution to \eqref{wpme}, which is adapted from \cite[Definition 2.1]{MP}. For notational purposes, when writing $ u(t) $ we refer to the whole solution at time $t$ as a spatial function.
	
\begin{den}\label{wpme sol}
	Let  $N\geq3$, $m>1$ and $\rho$ be a measurable function satisfying \eqref{weight-cond} with respect to some $\gamma\in[0,2)$ and $\underline{C},\overline{C}>0$. Let $u_0\in L^1_{\mathrm{loc}}\!\left(\mathbb{R}^N,\rho\right)$. Then we say that a function $ u $ is a solution of problem \eqref{wpme} if
	$$
	u \in C\!\left([0,+\infty);L^1_{\mathrm{loc}}\!\left(\mathbb{R}^N,\rho\right)\right) , \quad \nabla u^m \in L^2_{\mathrm{loc}}\!\left( \mathbb{R}^N \times (0,+\infty) \right) , \quad u(0) = u_0 \, ,
	$$
	and
	\begin{equation*}\label{q50w}
		\int_0^{+\infty} \int_{\mathbb{R}^N} u \, \phi_t \, \rho   \, dx dt = \int_0^{+\infty} \int_{\mathbb{R}^N}  \nabla u^m \cdot \nabla \phi\, dx dt
	\end{equation*}
	for all $\phi\in C^\infty_c\!\left(\mathbb{R}^N\times (0, +\infty)\right)$.
\end{den}

The above definition is in fact slightly stronger than the one given in \cite{MP}, since we make the extra requirement that solutions have a finite \emph{local weak energy} (i.e.~the local $ L^2 $ norm of the gradient of $ u^m $). This will be relevant to be able to apply \emph{local H\"{o}lder estimates} away from the origin. We will prove in Lemma \ref{constr-weak} that the kind of solutions we deal with, namely the ones constructed in \cite{MP}, are actually local weak energy solutions -- see also Remark \ref{loc-weak}.

Note that the lifetime of our solutions is always $ T = +\infty $, which is consistent with the fact that we will treat initial data that \emph{do not} have a critical growth. More precisely, it was shown in \cite[Theorem 2.2]{MP} that the condition
\begin{equation}\label{growth-norm}
	\left\| u_0 \right\|_{1, r}  := \sup_{R\geq r} {R^{-\frac{2-\gamma}{m-1} -N+\gamma }} \int_{B_R} \left| u_0(x) \right| \rho(x) \, dx < + \infty
\end{equation}
is sufficient (and to some extent necessary) to have existence of a solution to \eqref{wpme}, and the latter is ensured to exist at least up to
$$
T \sim \frac{1}{\left\| u_0 \right\|_{1, r}^{m-1}} \, .
$$
Therefore, if $ \lim_{r \to + \infty} \left\| u_0 \right\|_{1, r}  = 0 $ we obtain an infinite lifetime, and it is not difficult to check that this is always the case for initial data with $ \left\| u_0 \right\|_{0,\rho} < +\infty $. Such a behavior is not surprising, as the norm $\left\| u_0 \right\|_{1, r} $ is meant to capture \emph{large data} with a growth rate of order $ |x|^{(2-\gamma)/(m-1)} $ as $ |x| \to +\infty $, whereas  $\left\| u_0 \right\|_{0, \rho} $ accounts for data having (in particular) a \emph{bounded average}. In fact, we will prove that the finiteness of $ \left\| u_0 \right\|_{0,\rho}  $ yields \emph{global boundedness} of constructed solutions for positive times (see Proposition \ref{p-transl}).

From now on, we will simply call a \emph{solution} of \eqref{wpme} the one constructed in \cite[Theorem 2.2]{MP} which, as explained above and established in Appendix \ref{technical}, turns out to also meet the requirements of Definition \ref{wpme sol}. Furthermore, \emph{all} the constructed solutions considered in \cite{MP} satisfy the pointwise bound
$$
 \left|  u(x,t) \right|  \le C_\varepsilon \left( 1 +|x| \right)^{\frac{2}{m-1}} \qquad \text{for a.e. } x \in \mathbb{R}^N \, , \ \forall t \in (\varepsilon,T-\varepsilon) \, ,
$$
for every $ \varepsilon \in (0,T/2) $. According to \cite[Theorem 2.3]{MP}, this additional bound gives rise to a \emph{uniqueness class}, so no confusion occurs.

In order to study asymptotics, we now introduce \emph{rescaled solutions} as in \cite{AR, KU}. That is, given any sequence $ \xi_k \to +\infty $ and any solution $u$ of \eqref{wpme}, we set
	\begin{equation}\label{rescaled sol}
		u_k(x,t) := \xi^\alpha_k \, u\!\left(\xi_k x , \xi_k^{\frac{\alpha}{\lambda_\alpha}} \, t\right) .
	\end{equation}
	Note that this is precisely the scaling \eqref{scale invar} under which the self-similar solution $ \mathcal{U}_\alpha $ is invariant, and it is straightforward to verify that $ u_k $ is a solution of
	\begin{equation}\label{wpme-resc}
	\begin{cases}
	\rho_k \, \partial_t u_k = \Delta \!\left(u_k^m\right) & \text{in } \mathbb{R}^N \times (0 , + \infty)\, , \\
	u_k = u_{0k} & \text{on } \mathbb{R}^N\times \{0\} \, ,
	\end{cases}
	\end{equation}
	namely problem \eqref{wpme} with respect to the rescaled weight \eqref{rescaled-weight} and \eqref{rescaled-datum} as its initial datum. It is readily seen that $ u_k $ falls in the same uniqueness class as $u$, so it coincides with the constructed solution taking $u_{0k}$ as its initial datum. For instance, under the asymptotic assumptions \eqref{weight-cond3} and~\eqref{datum-cond3}, we have that \eqref{wpme-resc} ``converges'' as $ k \to \infty $ to the limit problem \eqref{singularWPME}, {i.e.}~the one solved by $ \mathcal{U}_\alpha $. This is the key preliminary observation to determine asymptotics as $ t \to +\infty $, which however requires several technical tools in order to be made rigorous.
	
	\subsection{Main results and paper organization}
	
We are now in position to state our claimed asymptotic results for solutions taking initial data that behave like a non-integrable power $ |x|^{-\alpha} $ at spatial infinity.
	
	\begin{thm}[Global weighted $ L^p $ convergence]\label{main thm}
		Let  $N\geq3$, $m>1$ and $\rho$ be a measurable function satisfying \eqref{weight-cond}--\eqref{weight-cond2} with respect to some $\gamma\in[0,2)$ and $\underline{C},\overline{C},c>0$. Let $u_0\in L^1_{\mathrm{loc}}\!\left(\mathbb{R}^N,\rho\right)$ satisfy \eqref{conv-resc-datum} and \eqref{second-req-u0} with respect to some $ \alpha \in (0,N-\gamma) $, $ b>0 $ and an admissible weight $ \Phi $. Then the solution $ u $ of \eqref{wpme} has the following convergence property:
		\begin{equation}\label{main thm eq}
\lim_{t \to + \infty} \left\| t^{\lambda_\alpha} u\!\left( t^{\frac{\lambda_\alpha}{\alpha}} x , t \right) - g_\alpha(|x|) \right\|_{L^p\left( \mathbb{R}^N , \Phi \left| x \right|^{-\gamma} \right)} = 0
		\end{equation}
 	for every $ p \in [1,\infty) $,	where $ \lambda_\alpha $ is defined in \eqref{lambdaa} and $ g_\alpha $ is the solution of the ODE problem \eqref{ode1}.
	\end{thm}
	
\begin{cor}\label{main thm cor}
	Let the assumptions of Theorem \ref{main thm} hold, and let the same notations be adopted. Then
			\begin{equation}\label{main thm eq_bis}
	\lim_{t \to + \infty} t^{\lambda_\alpha \left( p - \frac{N-\gamma}{\alpha} \right) }  \int_{\mathbb{R}^N} \left| u(y,t) - \mathcal{U}_\alpha(y,t) \right|^p \, \Phi\!\left( \frac{y}{t^{\frac{\lambda_\alpha}{\alpha}}} \right)  \left| y \right|^{-\gamma} dy  = 0
	\end{equation}
	for every $ p \in [1,\infty) $, where $ \mathcal{U}_\alpha $ is the self-similar solution of \eqref{singularWPME} defined in \eqref{Barenblatt}.
	\end{cor}

\begin{oss}\rm \label{gimpl}
	It is plain that, if \eqref{main thm eq_bis} holds for an admissible weight $ \Phi $, then it holds for any other weight $ \Psi \le \Phi $, even if the latter is not admissible. In particular, since $  \Phi $ is locally positive and bounded, one can take  $ \Psi \equiv \chi_{B_1} $, which entails
	\begin{equation}\label{main thm eee}
		\lim_{t \to + \infty} t^{\lambda_\alpha \left( p - \frac{N-\gamma}{\alpha} \right) }  \int_{ B_{  t^{ {\lambda_\alpha}/{\alpha} } }  } \left| u(y,t) - \mathcal{U}_\alpha(y,t) \right|^p  \left| y \right|^{-\gamma} dy  = 0 \, .
	\end{equation}
	A straightforward computation yields
	$$
	\int_{ B_{ t^{ {\lambda_\alpha}/{\alpha} }} } \left| \, \mathcal{U}_\alpha(y,t) \right|^p  \left| y \right|^{-\gamma} dy  = t^{-\lambda_\alpha \left( p - \frac{N-\gamma}{\alpha} \right) }
	$$
	up to constants, so \eqref{main thm eee} implies that $ u(t)-\mathcal{U}_\alpha(t) $ is of smaller order with respect to $ \mathcal{U}_\alpha(t) $, if measured in $L^p$ on suitable growing balls as $ t \to +\infty $.
	\end{oss}

Now we state the main result on uniform convergence, whose proof, carried out in Section \ref{unique}, is inspired by \cite[Section 5]{RV2}. Such a proof requires, in particular, a smallness criterion on the tail of solutions that is not guaranteed by our \emph{a priori} estimates in Section \ref{exist}, which calls for a careful adaptation of the $ L^\infty $-smallness implication already established for globally integrable solutions (for more details see the end of the proof of Theorem \ref{thm-unif-conv}). To this aim, we find \eqref{second-req-u0-re} to be sufficient, a mild convergence assumption replacing \eqref{conv-resc-datum} and \eqref{second-req-u0} (but in general neither stronger nor weaker),  which still allows $u_0$ to be \emph{sign-changing} and unbounded. We also require a natural additional regularity assumption \eqref{rho-extra} on the weight.
	
\begin{thm}[Global uniform convergence]\label{thm-unif-conv}
Let  $N\geq3$, $m>1$ and $\rho$ be a measurable function satisfying \eqref{weight-cond}--\eqref{weight-cond2} with respect to some $\gamma\in[0,2)$ and $\underline{C},\overline{C},c>0$. Let $u_0\in L^1_{\mathrm{loc}}\!\left(\mathbb{R}^N,\rho\right)$ satisfy \eqref{second-req-u0-re}
with respect to some $ \alpha \in (0,N-\gamma) $ and $ b>0 $. Suppose moreover that
	\begin{equation}\label{rho-extra}
	\rho \in C^{1}\!\left( \mathbb{R}^N  \setminus \{ 0 \} \right) , \qquad \left| \nabla \rho (x) \right| \le \widehat{C} \left| x \right|^{-\gamma-1} \quad \forall x \in \mathbb{R}^N \setminus \{ 0 \} \, ,
	\end{equation}
	for some constant $ \widehat{C}>0 $. Then
			\begin{equation}\label{main thm unif}
\lim_{t \to + \infty} \left\| t^{\lambda_\alpha} u\!\left( t^{\frac{\lambda_\alpha}{\alpha}} x , t \right) - g_\alpha(|x|) \right\|_{L^\infty\left( \mathbb{R}^N \right)} = 0 \,,
		\end{equation}
		or equivalently
					\begin{equation}\label{main thm unif t}
\lim_{t \to + \infty} t^{\lambda_\alpha} \left\| u(t) - \mathcal{U}_\alpha(t) \right\|_{L^\infty\left( \mathbb{R}^N  \right)} = 0 \, .
		\end{equation}
	\end{thm}
	
There is a quite natural class of initial data that meet assumptions \eqref{conv-resc-datum}, \eqref{second-req-u0} and \eqref{second-req-u0-re}, to which the above asymptotic results thus apply.
	
\begin{pro}\label{datum}
Let  $N\geq3$ and $\rho$ be a measurable function satisfying \eqref{weight-cond} with respect to some $\gamma\in[0,2)$ and $\underline{C},\overline{C} > 0 $. Let $ u_0 \in L^1_{\mathrm{loc}}\!\left( \mathbb{R}^N , \rho \right) $ fulfill
\begin{equation}\label{EE1}
	\underset{|x| \to +\infty}{\operatorname{ess}\lim} \, |x|^{\alpha} \, u_0(x) = b \, ,
\end{equation}
for some $ \alpha \in (0, N-\gamma) $ and $ b>0 $. Then the conditions \eqref{conv-resc-datum} (for every admissible weight $ \Phi $), \eqref{second-req-u0} and \eqref{second-req-u0-re} all hold.
\end{pro}
	
The rest of the paper is organized as follows. In Section \ref{exist} we establish \emph{a priori} $ L^\infty $ and $ L^1 $ estimates, which are of key importance because they turn out to be uniform with respect to the scalings introduced above. In Section \ref{unique} we first prove Theorem~\ref{main thm} (along with Corollary~\ref{main thm cor}), by taking advantage of such estimates and further approximation arguments. {Then we prove our uniform convergence result Theorem \ref{thm-unif-conv}.} Section \ref{ODE} is devoted to a thorough study of the well-posedness of the ODE problem \eqref{ode1} {and the proof of Theorem \ref{selfsim-sol}. Finally, in Appendix \ref{technical} we prove a simple energy estimate for constructed solutions, and we bring together some postponed proofs of a few technical lemmas, including Proposition \ref{datum}.}
	
\begin{rem}[On low dimensions]\rm\label{N=2}
Similarly to \cite[Remark 4.4]{MP}, we observe that all of our main results also hold in dimensions $ N=2 $ and $N=1$, up to some technical modifications that we briefly explain. First of all, in the proof of Proposition \ref{p1} below one needs to replace the weighted Sobolev inequality with Caffarelli-Kohn-Nirenberg inequalities (i.e.~weighted Gagliardo-Nirenberg inequalities) in order to get suitable analogues of the crucial estimates \eqref{m21} and \eqref{m21-transl}. For more details on such an approach, we refer to \cite[pages 67--69]{BCP}. Note that such inequalities hold for the singular weight $ \rho(x) \equiv |x|^{-\gamma} $ or for a bounded weight complying with \eqref{weight-cond}, but in general \eqref{weight-cond} by itself is not enough. Moreover, in dimension $ N=1 $ it is necessary to require that $ \gamma<1 $, otherwise the weight $ |x|^{-\gamma} $ is not even locally integrable. The ODE analysis in Section \ref{ODE} requires a few technical adjustments for $ N=1 $, since the  ``artificial'' dimension $ \tilde{N} $ is actually less than $ 1 $ (contrarily to \eqref{ode-parameters}) but still positive and greater than $ \tilde \alpha $.

Finally, some comments on the norm introduced in \eqref{norm-trasl} are in order. In the case $ N=2 $ the multiplying factor $ R^{-\gamma(N-2)/2} $ disappears; nevertheless, the norm $ \| \cdot \|_{0,\rho} $ is still weaker than the usual $ L^1\!\left(\mathbb{R}^N , \rho \right) $ norm. On the other hand, in the case $ N=1 $ we see that this factor becomes a positive power, hence there is no longer an ordering between the two norms: one can easily construct functions in $ L^1\!\left(\mathbb{R}^N , \rho \right) $ with an infinite $ \| \cdot \|_{0,\rho} $ norm and vice versa (e.g.~constants).

\end{rem}

\section{A priori estimates} \label{exist}
From now on, we take for granted that $ N \ge 3 $, $ m>1 $ and $ \rho $ is a measurable function satisfying \eqref{weight-cond} for some $\gamma\in[0,2)$ and $\underline{C},\overline{C}>0$. When mentioning the ``solution'' of \eqref{wpme} we mean the one constructed in \cite{MP}, as explained in Subsection \ref{setup}. Here we do not deal explicitly with rescaled data and solutions as in \eqref{rescaled-datum} and \eqref{rescaled sol}, but the bounds we obtain will be crucially applied to them in Section \ref{unique}. 	
		
In order to lighten the reading, in most of the integrals below we will remove integration variables, at least when no ambiguity occurs. Moreover, except when strictly necessary, we will omit the use of ``almost everywhere''. Let us also recall that the special exponents $\lambda$ and $\theta$ are defined in \eqref{def-lambda-theta}.

	\subsection{Smoothing and stability estimates} \label{smooth}
	We begin by proving a global $L^\infty$ smoothing effect for solutions of \eqref{wpme} in terms of the norm $\norm{\cdot}_{0,\rho}$ introduced in \eqref{norm-trasl}. The proof relies on a preliminary elliptic estimate, which is a careful modification of the one presented in \cite[Proposition 3.2]{MP} for ``large'' norms.
	\begin{pro}\label{p1}
Let $u \in L^\infty\!\left(\mathbb{R}^N\right)$, with $ u\geq 0$. Suppose moreover that $ \nabla\!\left( u^m \right) \in L^2_{\mathrm{loc}}\!\left( \mathbb{R}^N \right) $, $ \Delta \!\left( u^m \right) \in L^1_{\mathrm{loc}}\!\left(\mathbb{R}^N\right) $ and
	\begin{equation*}\label{m10}
		\Delta \! \left( u^m \right) \geq - \Lambda \, \rho \,  u  \qquad \text{in } \mathbb{R}^N
	\end{equation*}
	for some constant $\Lambda> 0$. Then there exists a positive constant $C_0$, depending only on $N,m,\gamma,\underline{C},\overline{C}$, such that
	\begin{equation}\label{m11}
		\left\| u \right\|_{L^\infty \left( \mathbb{R}^N \right)} \leq C_0 \left( \Lambda^{\lambda} \left\| u \right\|_{0,\rho}^{\theta \lambda} + \left\| u \right\|_{0,\rho}\right).
	\end{equation}
	\end{pro}
	\begin{proof}
First of all, we let $ R_\gamma > 1 $ be a suitable fixed number such that
\begin{equation}\label{est-R-gamma}
R-R^{\frac \gamma 2} \ge \frac{R}{2}  \qquad \forall R \ge R_\gamma \, .
\end{equation}
For instance, one can pick $ R_\gamma=2^{2/(2-\gamma)} $. By repeating the first part of the proof of \cite[Proposition 3.2]{MP} with $ R \equiv R_\gamma $ there, which crucially takes advantage of a weighted Sobolev inequality in $ \mathbb{R}^N $ associated with $ \rho $, we end up with the following inequality:
	\begin{equation}\label{m21}
		\left[ \int_{\mathbb{R}^N} \left( \phi_{R_\gamma} u  \right)^{s(q+1) + \beta } \rho \, dx \right]^{\frac 1 s} \leq  C \, q \left(\Lambda + \left\| u \right\|_{L^\infty\left(\mathbb{R}^N\right)}^{m-1} \right) \int_{\mathbb{R}^N} \left(\phi_{R_\gamma} u \right)^{q+1} \rho \, dx \, ,
	\end{equation}
	where
	\begin{equation*}\label{m21-bis}
		s:= \frac{N-\gamma}{N-2} > 1 \, , \qquad \beta :=  s(m-1)>0
	\end{equation*}
and $ q $ is any exponent larger than $ m $, the generic constant $C>0$ depending only on $N,m,\gamma,\underline{C},\overline{C}$. Here $ \phi_{R_\gamma} $ is a typical cutoff function that vanishes outside $ B_{2R_\gamma} $ and is identically equal to $1$ in $ B_{R_\gamma} $. Then we can set up an analogous iterative scheme of Moser type, by applying \eqref{m21} to the sequence of exponents $ q + 1 \equiv p_{j} $ defined via the recurrence relation
$$
p_{j+1} = s p_j + \beta \qquad \Longrightarrow \qquad p_j = \left( p_0 + \frac{\beta}{s-1} \right) s^j - \frac{\beta}{s-1} \, ,
$$
for a fixed $ p_0 \ge m+1 $. This readily entails
	\begin{equation}\label{m21-ter}
\alpha_{j+1} \le p_j^{\frac{s}{p_{j+1}}} \left[ C \left( \Lambda + \left\| u \right\|_{L^\infty\left(\mathbb{R}^N\right)}^{m-1} \right) \right]^{\frac{s}{p_{j+1}}} \alpha_j^{s\frac{p_j}{p_{j+1}}} \, ,
	\end{equation}
where
$$
\alpha_j :=  \left[ \int_{\mathbb{R}^N} \left( \phi_{R_\gamma} u  \right)^{p_j}  \rho \, dx \right]^{\frac {1}{p_j}}  .
$$
Upon iterating \eqref{m21-ter} $ j $ times, it is not difficult to infer the estimate
	\begin{equation}\label{m21-quater}
\alpha_{j+1}	\le s^{\frac{\sum_{i=1}^{j} \left(j-i+1 \right) s^i}{p_{j+1}}} \left[ C \left( p_0 + \frac{\beta}{s-1} \right) \left( \Lambda + \left\| u \right\|_{L^\infty\left(\mathbb{R}^N\right)}^{m-1} \right) \right]^{\frac{\sum_{i=1}^{j+1} s^i}{p_{j+1}}} \alpha_0^{p_0 \frac{s^{j+1}}{p_{j+1}}} \, ,
	\end{equation}
so by letting $ j \to \infty $ in \eqref{m21-quater} we deduce
	\begin{equation}\label{M1}
\left\|  \phi_{R_\gamma} u  \right\|_{L^\infty\left(\mathbb{R}^N\right)}
 \le  C \left( \Lambda + \left\| u \right\|_{L^\infty\left(\mathbb{R}^N\right)}^{m-1} \right)^{ \frac{N-\gamma}{(N-\gamma)(m-1)+p_0(2-\gamma)} } \left[ \int_{\mathbb{R}^N} \left( \phi_{R_\gamma} u  \right)^{p_0}  \rho \, dx\right]^{ \frac{2-\gamma}{(N-\gamma)(m-1)+p_0(2-\gamma)} } ,
	\end{equation}
	where $ C>0 $ is a constant as above that depends in addition on $p_0$. In particular, due to the support properties of $ \phi_{R_\gamma} $ and the trivial inequality
	\begin{equation}\label{M1-interp}
	\left( \phi_{R_\gamma} u  \right)^{p_0} \le \left\| u \right\|_{L^\infty\left( \mathbb{R}^N \right)}^{p_0-1} u \, \chi_{B_{2R_\gamma}} \,  ,
	\end{equation}
	from \eqref{M1} it follows that
	\begin{equation}\label{M2}
	\begin{aligned}
 \left\|  u  \right\|_{L^\infty\left( B_{R_\gamma} \right)}
 \le & \, C  \left( \Lambda + \left\| u \right\|_{L^\infty\left(\mathbb{R}^N\right)}^{m-1} \right)^{ \frac{N-\gamma}{(N-\gamma)(m-1)+p_0(2-\gamma)} }  \left[ \int_{B_{2R_\gamma} } u  \,  \rho \, dx\right]^{ \frac{2-\gamma}{(N-\gamma)(m-1)+p_0(2-\gamma)} }  \\
 & \times \left\| u \right\|_{L^\infty\left( \mathbb{R}^N \right)}^{\frac{(p_0-1)(2-\gamma)}{(N-\gamma)(m-1)+p_0(2-\gamma)}} .
\end{aligned}
	\end{equation}
	On the other hand, it is plain that $ \| u \|_{0,\rho} $ controls the weighted $ L^1 $ norm of $ u $ in $ B_{2R_\gamma} $, hence \eqref{M2} yields in turn
	\begin{equation}\label{M3}
		\begin{aligned}
\left\|  u  \right\|_{L^\infty\left( B_{R_\gamma} \right)} \le  C \left( \Lambda + \left\| u \right\|_{L^\infty\left(\mathbb{R}^N\right)}^{m-1} \right)^{ \frac{N-\gamma}{(N-\gamma)(m-1)+p_0(2-\gamma)} } \left\| u \right\|_{0 , \rho}^{ \frac{2-\gamma}{(N-\gamma)(m-1)+p_0(2-\gamma)} } \left\| u \right\|_{L^\infty\left( \mathbb{R}^N \right)}^{\frac{(p_0-1)(2-\gamma)}{(N-\gamma)(m-1)+p_0(2-\gamma)}}
 \end{aligned}
	\end{equation}	
	up to a different constant depending on the same quantities as above.
		
	Our goal is now to prove \eqref{M3} with $ B_{R_\gamma} $ replaced by $ B_{R_\gamma}^c $. To this end, given an arbitrary $ R \ge R_\gamma $ and any $ z_R \in \partial B_R $, we set
		\begin{equation}\label{m-def-psiR}
		\psi_R(x) := \phi\!\left( 2\left| \frac{x-z_R}{R^{\frac{\gamma}{2}}} \right| \right) \qquad \forall x \in \mathbb{R}^N \, ,
		\end{equation}
		where $ \phi $ is again a standard cutoff function satisfying
		\begin{equation*}\label{m52}
			0 \leq \phi \leq 1 \quad \text{in } [0, +\infty) \, , \qquad \phi \equiv 1 \quad \text{in } [0,1] \, , \qquad \phi \equiv 0 \quad \text{in } [2,+\infty) \, .
		\end{equation*}
		We require moreover that
		\begin{equation}\label{m17}
		 \phi(r)^{m-2} \left[ \left|\phi''(r)\right| + (N-1) 	 \, \frac{\left| \phi'(r)\right|}{r}  \right]  + \phi(r)^{m-3} \left|\phi'(r)\right|^2  \leq K_m \, \chi_{[1,2]}(r) \qquad \forall r > 0
		\end{equation}
	for some constant $ K_m>0 $, which can be achieved by choosing carefully $ \phi $ since $ m>1 $. As a consequence, we can infer from \eqref{m17} and the definition of $ \psi_R $ that
		\begin{equation}\label{m18-transl}
			\psi_R(x)^{m-1} \left|\Delta \phi_R(x) \right| + \psi_R(x)^{m-2} \left|\nabla \phi_R(x)\right|^2 \leq \frac{K_m}{R^\gamma} \, \psi_R(x) \, \chi_{ \left[ \frac 1 2 R^{\frac \gamma 2}  , \,  R^{\frac \gamma 2} \right]}(|x-z_R|) \qquad \forall x \in \mathbb{R}^N \, .
		\end{equation}
		By repeating again the same computations as in the first part of the proof of \cite[Proposition 3.2]{MP} with $ \phi_R \equiv \psi_R $ there, and taking advantage of \eqref{m18-transl}, we obtain the following estimate:
\begin{equation}\label{m19-transl}
			\begin{aligned}
			\int_{\mathbb{R}^N} \left|\nabla\left(\psi_R u \right)^{\frac{q+m}2}\right|^2 dx
				\leq  C \, q\left[\Lambda \int_{\mathbb{R}^N} \left(\psi_R u \right)^{q+1} \rho \, dx + \frac 1{R^\gamma} \int_{B_{R^{\gamma/2}}(z_R) \setminus B_{\frac 1 2 R^{\gamma/2}}(z_R)} \left(\psi_R u \right)^{q+1} u^{m-1} \, dx \right] ,
			\end{aligned}
		\end{equation}
		where the exponent $ q $ and the generic constant $C$ are exactly as before. Since $ R \ge R_\gamma $ we observe that
		$$
		\rho(x) \ge \frac{\underline{C}}{\left( 1+|x| \right)^\gamma} \ge \frac{\underline{C}}{\left( 1 + R+R^{\frac{\gamma}{2}} \right)^\gamma} \ge \frac{1}{R^\gamma} \, \frac{\underline{C}}{\left( R_\gamma^{-1}+\frac 3 2 \right)^{\gamma}} \qquad \forall x \in B_{R^{{\gamma}/{2}}}(z_R)  \, ,
		$$
so \eqref{m19-transl} yields
		\begin{equation}\label{m20-transl}
			\int_{\mathbb{R}^N} \left|\nabla\left(\psi_R u \right)^{\frac{q+m}2}\right|^2 dx
			\leq {C} \, q  \left( \Lambda + \left\| u \right\|^{m-1}_{L^\infty\left(\mathbb{R}^N\right)} \right) \int_{\mathbb{R}^N} \left(\psi_R u \right)^{q+1} \rho \, dx \, .
		\end{equation}
		Now, we exploit the fact that the support of $ \psi_R $ lies in $ B_{R^{\gamma/2}}(z_R) $, a region of (unweighted) volume proportional to $ R^{N\gamma/2}$ in which, thanks to \eqref{weight-cond} and \eqref{est-R-gamma}, the density $ \rho $ is equivalent to $ R^{-\gamma} $. In particular, we can use the following localized version of the standard Sobolev inequality:
		\begin{equation*}\label{e12-transl}
			\begin{aligned}
				\left( \int_{\mathbb{R}^N} \left| f \right|^{\frac{2(N-\gamma)}{N-2}} \rho \, dx \right)^{\frac{N-2}{N-\gamma}} \leq & \, {C} R^{-\gamma\frac{N-2}{N-\gamma}} \, \left( \int_{\mathbb{R}^N} \left| f \right|^{\frac{2(N-\gamma)}{N-2}} dx \right)^{\frac{N-2}{N-\gamma}} \\
				\leq & \, {C} R^{-\gamma\frac{N-2}{N-\gamma}} \, R^{\frac{N\gamma}{2}\left( \frac{N-2}{N-\gamma} - \frac{N-2}{N} \right)} \left( \int_{\mathbb{R}^N} \left| f \right|^{\frac{2N}{N-2}} dx \right)^{\frac{N-2}{N}}
				\\
				\leq & \, {C} R^{-\frac{\gamma(N-2)(2-\gamma)}{2(N-\gamma)}} \int_{\mathbb{R}^N} \left| \nabla f \right|^{2} dx \qquad \forall f\in C^1_c\!\left( B_{R^{\frac{\gamma}{2}}}(z_R) \right) ,
			\end{aligned}
		\end{equation*}
which applied to \eqref{m20-transl} with $ f \equiv (\psi_R u)^{(q+m)/2} $ entails
		\begin{equation}\label{m21-transl}
			\left(\int_{\mathbb{R}^N} \left( \psi_R u  \right)^{s(q+1) + \beta } \rho \, dx \right)^{\frac 1 s} \leq  {C} \, q \left[ \left(\Lambda + \| u \|_{L^\infty(\mathbb{R}^N)}^{m-1} \right) R^{-\frac{\gamma(N-2)(2-\gamma)}{2(N-\gamma)}} \right] \int_{\mathbb{R}^N} \left(\psi_R u \right)^{q+1} \rho \, dx \, .
\end{equation}
At this point, starting from \eqref{m21-transl} one can carry out exactly the same iterations as in \eqref{m21}--\eqref{M1}, which lead to
	\begin{equation*}\label{M1-bis}
	\begin{aligned}
 \left\|  \psi_{R} u  \right\|_{L^\infty\left(\mathbb{R}^N\right)}
 \le  \, C \left( \Lambda + \left\| u \right\|_{L^\infty\left(\mathbb{R}^N\right)}^{m-1} \right)^{ \frac{N-\gamma}{(N-\gamma)(m-1)+p_0(2-\gamma)} } \frac{\left[ \int_{\mathbb{R}^N} \left( \psi_{R} u  \right)^{p_0}  \rho \, dx\right]^{ \frac{2-\gamma}{(N-\gamma)(m-1)+p_0(2-\gamma)} } } { R^{\frac{\gamma(N-2)(2-\gamma)}{2[(N-\gamma)(m-1)+p_0(2-\gamma)]}} } \, ,
\end{aligned}
	\end{equation*}
whence, by exploiting an analogue of \eqref{M1-interp},
	\begin{equation}\label{M1-ter}
	\begin{aligned}
\left\|  u  \right\|_{L^\infty\left( B_{\frac 1 2 R^{\gamma / 2}}(z_R) \right)} \le & \, C \left( \Lambda + \left\| u \right\|_{L^\infty\left(\mathbb{R}^N\right)}^{m-1} \right)^{ \frac{N-\gamma}{(N-\gamma)(m-1)+p_0(2-\gamma)} }   \\
 & \times \left[ \frac{ \int_{ B_{R^{\gamma / 2}}(z_R) } u  \,  \rho \, dx}{{ R^{\frac{\gamma(N-2)}{2}} }  }  \right]^{ \frac{2-\gamma}{(N-\gamma)(m-1)+p_0(2-\gamma)}} \left\| u \right\|_{L^\infty\left( \mathbb{R}^N \right)}^{\frac{(p_0-1)(2-\gamma)}{(N-\gamma)(m-1)+p_0(2-\gamma)}} .
\end{aligned}
	\end{equation}
By taking the supremum of both sides of \eqref{M1-ter} over all $ R \ge R_\gamma $ and $z_R \in \partial B_R$, we end up with
		\begin{equation*}\label{M3-complement}
\left\|  u  \right\|_{L^\infty\left( B_{R_\gamma}^c \right)}
 \le  C \left( \Lambda + \left\| u \right\|_{L^\infty\left(\mathbb{R}^N\right)}^{m-1} \right)^{ \frac{N-\gamma}{(N-\gamma)(m-1)+p_0(2-\gamma)} } \left\| u \right\|_{0 , \rho}^{ \frac{2-\gamma}{(N-\gamma)(m-1)+p_0(2-\gamma)} } \left\| u \right\|_{L^\infty\left( \mathbb{R}^N \right)}^{\frac{(p_0-1)(2-\gamma)}{(N-\gamma)(m-1)+p_0(2-\gamma)}} ,
	\end{equation*}	
	which, in combination with \eqref{M3}, ensures that
	\begin{equation*}\label{M3-complement-bis}
\left\|  u  \right\|_{L^\infty\left( \mathbb{R}^N \right)}
 \le  C \left( \Lambda + \left\| u \right\|_{L^\infty\left(\mathbb{R}^N\right)}^{m-1} \right)^{ \frac{N-\gamma}{(N-\gamma)(m-1)+p_0(2-\gamma)} } \left\| u \right\|_{0 , \rho}^{ \frac{2-\gamma}{(N-\gamma)(m-1)+p_0(2-\gamma)} } \left\| u \right\|_{L^\infty\left( \mathbb{R}^N \right)}^{\frac{(p_0-1)(2-\gamma)}{(N-\gamma)(m-1)+p_0(2-\gamma)}} ,
 	\end{equation*}	
namely
	\begin{equation*}\label{M4-1}
\left\|  u  \right\|_{L^\infty\left( \mathbb{R}^N \right)} \le C \left( \Lambda + \left\| u \right\|_{L^\infty\left(\mathbb{R}^N\right)}^{m-1} \right)^{ \frac{N-\gamma}{(N-\gamma)(m-1)+2-\gamma} } \left\| u \right\|_{0 , \rho}^{ \frac{2-\gamma}{(N-\gamma)(m-1)+2-\gamma} } .
	\end{equation*}
Estimate \eqref{m11} thus follows as a direct application of \cite[Lemma 1.3]{BCP}.
	\end{proof}
	
By taking advantage of Proposition \ref{p1}, we can now establish the analogue of \cite[Proposition 1.3]{BCP}, that is a stability estimate for the $ \| \cdot \|_{0,\rho} $ norm along with a global smoothing effect.
	
	\begin{pro}\label{p-transl}
		Let $ u_0 \in L^1_{\mathrm{loc}}\!\left( \mathbb{R}^N , \rho \right)  $ be such that $ \left\| u_0 	\right\|_{0,\rho} < + \infty $, and let $ u $ be the corresponding solution of \eqref{wpme}. Then there exist positive constants $ C_1$ and $ C_2 $, depending only on $ N,m,\gamma,\underline{C},\overline{C}$, such that
		\begin{equation}\label{l1 stability}
			\left\| u(t) \right\|_{0,\rho} \leq C_1 \left\| u_0 \right\|_{0,\rho} \qquad \forall t>0 \, ,
		\end{equation}
		\begin{equation}\label{smooth-parab-transl}
			\left\| u(t) \right\|_{L^\infty\left(\mathbb{R}^N\right)} \leq C_2 \left(  t^{-\lambda} \left\| u_0 \right\|_{0,\rho}^{\theta\lambda} + \left\| u_0 \right\|_{0,\rho} \right) \qquad \forall t>0 \, .
		\end{equation}
	\end{pro}
	\begin{proof}
		We combine ideas from the proofs of  \cite[Proposition 1.3]{BCP} and \cite[Proposition 3.4]{MP}. In particular, with no loss of generality we can assume that $ u \ge 0 $ and $ u_0 \in L^1\!\left( \mathbb{R}^N , \rho \right)  \cap L^\infty\!\left( \mathbb{R}^N \right) $, so all the quantities involved below are well defined and finite.
		
		  Upon adopting the same notations as in the proof of Proposition \ref{p1}, we obtain
		\begin{equation}\label{m32-t}
			\begin{aligned}
				\frac{d}{d t} \int_{\mathbb{R}^N} u(t) \, \phi_{R_\gamma} \, \rho \,  dx  =  & \, \int_{\mathbb{R}^N} \left[\Delta u(t)^m \right]  \phi_{R_\gamma} \, dx
				=  \, \int_{\mathbb{R}^N} u(t)^m \left(\Delta \phi_{R_\gamma} \right)  dx \\
				\leq & \, C  \int_{B_{2 R_\gamma} \setminus B_{R_\gamma} } u(t)^m \, dx \leq  {C} \left\| u(t) \right\|^{m-1}_{L^\infty\left(\mathbb{R}^N\right)}  \int_{B_{2 R_\gamma}} u(t) \, \rho \, dx \,  ,
			\end{aligned}
		\end{equation}
		where $C>0$ is a generic constant that depends at most on $N,m,\gamma,\underline{C},\overline{C}$, that we do not relabel from line to line. Note that the computation is justified since $ u $ can be shown to be a strong solution, i.e.~its time derivative is actually a function (see \cite[Proposition 3.3]{MP}). By integrating \eqref{m32-t} in $ (0,t) $, for any $t>0$, we infer
		\begin{equation*}\label{m33-t}
			\int_{\mathbb{R}^N} u(t) \, \phi_{R_\gamma} \, \rho \, dx  \leq \int_{\mathbb{R}^N} u_0 \, \phi_{R_\gamma} \, \rho \, dx +  {C} \int_0^t \left\| u(s) \right\|_{L^\infty\left(\mathbb{R}^N\right)}^{m-1} \int_{B_{2 R_\gamma}} u(s) \, \rho \, dx ds \, ,
		\end{equation*}
		which, upon setting $ g(t):= \left\| u(t) \right\|_{0,\rho} $, implies
				\begin{equation}\label{m34-t}
		\int_{B_{R_\gamma}} u(t) \,  \rho \, dx  \leq {C} \left( g(0) + \int_0^t \left\| u(s) \right\|_{L^\infty\left(\mathbb{R}^N\right)}^{m-1}  g(s) \, ds \right) .
		\end{equation}
		Now, for arbitrary $ R \ge R_\gamma $ and $ z_R \in \partial B_R $, we perform a similar computation to \eqref{m32-t} with $ \phi_{R_\gamma} $ replaced by $ \psi_{R} $ (recall \eqref{m-def-psiR}), obtaining
				\begin{equation}\label{m32-t-bis}
		\begin{aligned}
		\frac{d}{d t} \int_{\mathbb{R}^N} u(t) \, \psi_R \, \rho \,  dx  = & \, \int_{\mathbb{R}^N} \left[\Delta u(t)^m\right]  \psi_R \, dx
		=  \, \int_{\mathbb{R}^N} u(t)^m \left(\Delta \psi_R\right)  dx \\
		\leq & \, \frac{C}{R^\gamma}  \int_{ B_{R^{\gamma/2}}(z_R)  } u(t)^m \, dx
		\leq  {C} \left\| u(t) \right\|^{m-1}_{L^\infty\left(\mathbb{R}^N\right)}  \int_{ B_{R^{\gamma/2}}(z_R) } u(t) \, \rho \, dx \,  .
		\end{aligned}
		\end{equation}
		If we integrate \eqref{m32-t-bis} and multiply both sides by $ R^{-\gamma(N-2)/{2} } $, we end up with
				\begin{equation}\label{m33-t-10}
				\begin{aligned}
		R^{-\frac{\gamma(N-2)}{2} } \int_{ B_{\frac 1 2 R^{\gamma / 2}} (z_R) } u(t) \, \rho \, dx  \leq & \, R^{-\frac{\gamma(N-2)}{2} }  \int_{ B_{R^{\gamma / 2}(z_R) } } u_0 \, \rho \, dx \\
		   & +  C \int_0^t \left\| u(s) \right\|_{L^\infty\left(\mathbb{R}^N\right)}^{m-1} \left( R^{-\frac{\gamma(N-2)}{2} } \int_{B_{R^{\gamma/2}}(z_R)} u(s) \, \rho \, dx \right) ds \, .
		\end{aligned}
		\end{equation}
       By combining \eqref{m34-t}, \eqref{m33-t-10}, and taking the supremum over $ R \ge R_\gamma $, we deduce that
				\begin{equation*}\label{m34-u}
\int_{B_{R_\gamma}} u(t) \,  \rho \, dx + \sup_{\substack{ R \ge R_\gamma \\[0.4mm] z_R \in \partial B_{R} }} \frac{ \int_{ B_{\frac 1 2 R^{\gamma / 2}} (z_R)} u(t) \, \rho \, dx}{R^{\frac{\gamma(N-2)}{2} } } \leq {C} \left( g(0) + \int_0^t \left\| u(s) \right\|_{L^\infty\left(\mathbb{R}^N\right)}^{m-1}  g(s) \, ds \right) .
\end{equation*}
On the other hand, it is readily seen that the quantity on the left-hand side is equivalent to $ \left\| u(t) \right\|_{0,\rho} $, therefore
				\begin{equation}\label{m34-v}
g(t) \leq {C} \left( g(0) + \int_0^t \left\| u(s) \right\|_{L^\infty\left(\mathbb{R}^N\right)}^{m-1}  g(s) \, ds \right) \qquad \forall t> 0 \, .
\end{equation}
Since $u$ satisfies the B\'enilan-Crandall inequality (see again \cite[Proposition 3.3]{MP}), that is
$$
\rho \, u_t \ge - \frac{\rho \, u}{(m-1)t} \qquad \text{in } \mathbb{R}^N \times (0,+\infty) \, ,
$$
we are in position to apply Proposition \ref{p1} to $ u(t) $, which yields (up to a different $C_0$)
				\begin{equation}\label{m34-z}
		\left\| u(t) \right\|_{L^\infty \left( \mathbb{R}^N \right)} \leq C_0 \left( t^{-\lambda} \left\| u(t) \right\|_{0,\rho}^{\theta \lambda} + \left\| u(t) \right\|_{0,\rho}\right)  \qquad \forall t > 0 \, .
\end{equation}
By substituting such a bound into \eqref{m34-v}, we end up with
				\begin{equation}\label{n11}
g(t) \leq {C} \left[ g(0) + \int_0^t \left( s^{-\lambda(m-1)} g(s)^{\theta \lambda (m-1)+1}  + g(s)^m \right) ds \right] \qquad \forall t>0 \, .
\end{equation}
At this point, it is enough to observe that the integral inequality \eqref{n11} is exactly of the same type as \cite[formula (3.37)]{MP}, hence ODE comparison methods yield
				\begin{equation}\label{n12}
g(t) \le C \, g(0) \qquad \forall t \in \left( 0 , \frac{c}{g(0)^{m-1}}\right) ,
\end{equation}
where $ c>0 $ is another generic constant depending at most on $N,m,\gamma,\underline{C},\overline{C}$. As a result, by plugging \eqref{n12} into \eqref{m34-z}, we obtain
		\begin{equation}\label{smooth-parab-transl-almost}
\left\| u(t) \right\|_{L^\infty\left(\mathbb{R}^N\right)} \leq C \left(  t^{-\lambda} \left\| u_0 \right\|_{0,\rho}^{\theta\lambda} + \left\| u_0 \right\|_{0,\rho} \right) \quad \forall t \in \left( 0 , t_0 \right) , \qquad t_0 :=  \frac{c}{\left\| u_0 \right\|_{0,\rho}^{m-1}}  \, .
\end{equation}
The validity of \eqref{smooth-parab-transl} is then a simple consequence of \eqref{smooth-parab-transl-almost} along with the fact that the $ L^\infty $ norm of $ u(t) $ is not increasing in time. Indeed,
				\begin{equation}\label{n13}
\begin{gathered}
\left\| u(t) \right\|_{L^\infty\left(\mathbb{R}^N\right)}  \le \left\| u(t_0) \right\|_{L^\infty\left(\mathbb{R}^N\right)} \le C \left(  t_0^{-\lambda} \left\| u_0 \right\|_{0,\rho}^{\theta\lambda} + \left\| u_0 \right\|_{0,\rho} \right) = C\left( c^{-\lambda}+1 \right) \left\| u_0 \right\|_{0,\rho}  \qquad
 \forall t \ge t_0 \, ,
 \end{gathered}
\end{equation}
since $ \lambda(m-1)+\theta \lambda = 1 $. Finally, upon noticing that the $ L^\infty $ norm controls $ \| \cdot \|_{0,\rho} $, from \eqref{n12} and~\eqref{n13} it is immediate to deduce \eqref{l1 stability}.
	\end{proof}
	
	\subsection{Two nonstandard Cauchy estimates} \label{subsec: unif estimates}
	Our goal here is to show that the $ L^1 $ norms arising from admissible weights (recall in particular \eqref{A1}--\eqref{A2}), as well as the norm $ \| \cdot \|_{0,\rho} $, are well suited to obtain Cauchy estimates for solutions of \eqref{wpme}.
	
	\begin{pro}\label{cauchy-adm}
		Let $ \Phi : \mathbb{R}^N \to \mathbb{R}^+ $ be any weight complying with \eqref{A1}--\eqref{A2}, and let $ u_0 , v_0 \in L^1\!\left( \mathbb{R}^N , \Phi \rho \right)  $  be such that $ \left\| u_0 	\right\|_{0,\rho} +\left\| v_0 	\right\|_{0,\rho} < + \infty $. Let $ u $ and $v $ be the corresponding solutions of \eqref{wpme} starting from $ u_0 $ and $ v_0 $, respectively.  Then there exists a positive constant $C_3$, depending only on $ N,m,\gamma,\underline{C},\overline{C}$ and the constant $K$ in \eqref{A2}, such that
		\begin{equation}\label{stability-C-adm}
		\begin{aligned}
		\left\| u(t) - v(t) \right\|_{L^1\left( \mathbb{R}^N , \Phi \rho \right)}
	\le \,  e^{C_3 \left[ t^{\theta\lambda} \left( \left\| u_0 \right\|_{0,\rho} \vee \left\| v_0 \right\|_{0,\rho} \right)^{\theta \lambda (m-1)} + t \left( \left\| u_0 \right\|_{0,\rho} \vee \left\| v_0 \right\|_{0,\rho} \right)^{m-1} \right] } \left\| u_0 - v_0 \right\|_{L^1\left( \mathbb{R}^N , \Phi \rho \right)}
		\end{aligned}
		\end{equation}
		for all $t>0$.
		\end{pro}
	\begin{proof}
		Similarly to the proof of \cite[Proposition 3.5]{MP}, first of all we formally take the time derivative of the following quantity:
		\begin{equation}\label{Phiphi-1}
		\begin{aligned}
		\frac{d}{dt} \int_{\mathbb{R}^N} \left| u(t) - v(t) \right| \Phi \, \phi_R \, \rho \, dx = & \, \int_{\mathbb{R}^N} \sign \! \left( u(t) - v(t) \right) \Delta \! \left[ u(t)^m - v(t)^m  \right] \Phi \, \phi_R \, dx \\
		\le & \, \int_{\mathbb{R}^N} \left| u(t)^m - v(t)^m  \right| \Delta\! \left( \Phi \, \phi_R \right) dx \, ,
		\end{aligned}
		\end{equation}
where $ \phi_R $ ($ R \ge 1 $) is a standard cutoff function that vanishes outside $ B_{2R} $ and is everywhere equal to $1$ in $ B_R $. As observed in the aforementioned proof, the computation takes advantage of Kato's inequality and is rigorously justified provided $ u_0,v_0 \in L^1\!\left( \mathbb{R}^N , \rho \right)  \cap L^\infty\!\left( \mathbb{R}^N \right) $, since the corresponding solutions are strong. In particular, we have that $ u,v \in L^\infty\!\left( (0,+\infty) ; L^1\!\left( \mathbb{R}^N , \rho \right) \right) \cap L^\infty\!\left( \mathbb{R}^N  \times (0,+\infty)\right) $ (see \cite[Proposition 3.3]{MP}), a fact that we take for granted from now on. This is not restrictive via a standard approximation procedure of the actual initial data with data in $L^1\!\left( \mathbb{R}^N , \rho \right)  \cap L^\infty\!\left( \mathbb{R}^N \right) $. We then estimate the rightmost term in \eqref{Phiphi-1}, using \eqref{A2}, as follows:
\begin{equation}\label{Phiphi-2}
\begin{aligned}
  \int_{\mathbb{R}^N} \left| u(t)^m - v(t)^m  \right| \left|  \Delta\! \left( \Phi \, \phi_R \right) \right| dx
 =  & \, \int_{\mathbb{R}^N} \left| u(t)^m - v(t)^m  \right| \left| \phi_R \,  \Delta \Phi + \Phi \, \Delta \phi_R + 2 \nabla \Phi \cdot \nabla \phi_R \right| dx \\
 \le & \, K \int_{\mathbb{R}^N} \left| u(t)^m - v(t)^m  \right| \frac{\Phi}{\left( 1 + |x| \right)^\gamma} \, dx  \\
 & + C  \int_{ B_{2R} \setminus B_R } \left| u(t)^m - v(t)^m  \right| \left[ \frac{\Phi}{R^2} + \frac{2K}{R\left(1 + |x| \right)^{\gamma-1}} \right] dx \, ,
 \end{aligned}
\end{equation}
where $C>0$ is a constant related to $ \phi_R $ only but independent of $R$. By virtue of \eqref{weight-cond}, Lagrange's theorem and the smoothing estimate \eqref{smooth-parab-transl}, we infer
		\begin{equation*}\label{Phiphi-3}
		\begin{aligned}
		& \, \int_{\mathbb{R}^N} \left| u(t)^m - v(t)^m  \right| \frac{\Phi}{\left( 1 + |x| \right)^\gamma} \, dx \\
		\le & \, \underline{C}^{-1} \, m \left( \left\| u(t) \right\|_{L^\infty\left( \mathbb{R}^N \right)} \vee \left\| v(t) \right\|_{L^\infty\left( \mathbb{R}^N \right)}  \right)^{m-1} \int_{\mathbb{R}^N} \left| u(t)- v(t)  \right| \Phi \, \rho \, dx \\
		\le & \, \underline{C}^{-1} \, {m \, C_2^{m-1}} \left[  t^{-\lambda} \left( \left\| u_0 \right\|_{0,\rho} \vee \left\| v_0 \right\|_{0,\rho} \right)^{\theta\lambda} + \left( \left\| u_0 \right\|_{0,\rho} \vee \left\| v_0 \right\|_{0,\rho} \right) \right]^{m-1} \int_{\mathbb{R}^N} \left| u(t)- v(t)  \right| \Phi \, \rho \, dx \, .
		\end{aligned}
		\end{equation*}
As concerns the last integral in \eqref{Phiphi-2}, upon recalling that $ \Phi \in L^\infty\!\left(\mathbb{R}^N\right) $  and that $ R^{-\gamma} \le \rho $ (up to constants) in the annulus $ B_{2R} \setminus B_R $, we can deduce that 	
\begin{equation}\label{Phiphi-4}
\begin{aligned}
& \, \int_{ B_{2R} \setminus B_R } \left| u(t)^m - v(t)^m  \right| \left[ \frac{\Phi}{R^2} + \frac{2K}{R\left(1 + |x| \right)^{\gamma-1}} \right] dx \\
 \le & \,  C \left( \left\| u_0 \right\|_{L^\infty\left( \mathbb{R}^N \right)} \vee \left\| v_0 \right\|_{L^\infty\left( \mathbb{R}^N \right)}  \right)^{m-1} \int_{ B_{2R} \setminus B_R } \left| u(t) - v(t)  \right| \rho \, dx \, ,
\end{aligned}
\end{equation}
where $ C>0 $ is another suitable constant independent of $R$ and $t$. Note that we implicitly exploited the time decrease of $ L^\infty $ norms. As a result, by integrating \eqref{Phiphi-1} in $ (0,t) $, for an arbitrary $ t>0 $, and taking advantage of \eqref{Phiphi-2}--\eqref{Phiphi-4}, we end up with
\begin{equation}\label{Phiphi-5}
\begin{aligned}
& \, \int_{\mathbb{R}^N} \left| u(t) - v(t) \right| \Phi \, \phi_R \, \rho \, dx \\
 \le & \, \int_{\mathbb{R}^N} \left| u_0 - v_0 \right| \Phi \, \phi_R \, \rho \, dx \\
  & + C \int_0^t \left[  s^{-\lambda} \left( \left\| u_0 \right\|_{0,\rho} \vee \left\| v_0 \right\|_{0,\rho} \right)^{\theta\lambda} + \left( \left\| u_0 \right\|_{0,\rho} \vee \left\| v_0 \right\|_{0,\rho} \right) \right]^{m-1} \int_{\mathbb{R}^N} \left| u(s)- v(s)  \right| \Phi \, \rho \, dx ds \\
 & +  C \left( \left\| u_0 \right\|_{L^\infty\left( \mathbb{R}^N \right)} \vee \left\| v_0 \right\|_{L^\infty\left( \mathbb{R}^N \right)}  \right)^{m-1} \int_0^t \int_{ B_{2R} \setminus B_R } \left| u(s) - v(s)  \right| \rho \, dx ds \, ,
\end{aligned}
\end{equation}
where now $C>0$ stands for a generic constant as in the statement. Since both $u$ and $v$ belong to $ L^\infty\!\left( (0,+\infty) ; L^1\!\left( \mathbb{R}^N , \rho \right) \right)$, we can safely let $ R \to +\infty $ in \eqref{Phiphi-5}, obtaining
\begin{equation*}\label{Phiphi-6}
\begin{aligned}
& \, \int_{\mathbb{R}^N} \left| u(t) - v(t) \right| \Phi \, \rho \, dx \\
 \le & \, \int_{\mathbb{R}^N} \left| u_0 - v_0 \right| \Phi \,  \rho \, dx \\
  & + C \int_0^t \left[  s^{-\lambda} \left( \left\| u_0 \right\|_{0,\rho} \vee \left\| v_0 \right\|_{0,\rho} \right)^{\theta\lambda} + \left( \left\| u_0 \right\|_{0,\rho} \vee \left\| v_0 \right\|_{0,\rho} \right) \right]^{m-1} \int_{\mathbb{R}^N} \left| u(s)- v(s)  \right| \Phi \, \rho \, dx ds
\end{aligned}
\end{equation*}
for all $t>0$. The claimed estimate \eqref{stability-C-adm} thus follows from a direct application of Gr\"onwall's inequality.
		\end{proof}

\begin{oss}\rm
By combining the $ L^1\!\left( \mathbb{R}^N , \rho \right) $ continuity of solutions to \eqref{wpme} starting from $ L^1\!\left( \mathbb{R}^N , \rho \right) $ data, see \cite[Proposition 3.3]{MP} and references therein, and the just proved stability estimate \eqref{stability-C-adm}, it is not difficult to show that initial data complying with the assumptions of Proposition \ref{cauchy-adm} give rise to solutions that are continuous curves in $ L^1\!\left( \mathbb{R}^N , \Phi \rho \right) $.
\end{oss}

The following proposition, whose proof is a simple modification of the arguments of Propositions~\ref{p-transl} and \ref{cauchy-adm}, will be crucially used in Section \ref{unique} to control the tail behavior of the (rescaled) solutions uniformly in time.

	\begin{pro}\label{cauchy-adm-0norm}
Let $ u_0 , v_0 \in L^1_{\mathrm{loc}}\!\left( \mathbb{R}^N , \rho \right)  $  be such that $ \left\| u_0 	\right\|_{0,\rho} +\left\| v_0 	\right\|_{0,\rho} < + \infty $. Let $ u $ and $v $ be the corresponding solutions of \eqref{wpme} starting from $ u_0 $ and $ v_0 $, respectively.  Then there exist positive constants $C_4,C_5$, depending only on $ N,m,\gamma,\underline{C},\overline{C}$, such that
		\begin{equation}\label{stability-C-adm-bis}
		\begin{aligned}
		\left\| u(t) - v(t) \right\|_{0,\rho}
	\le C_4 \, e^{C_5 \left[ t^{\theta\lambda} \left( \left\| u_0 \right\|_{0,\rho} \vee \left\| v_0 \right\|_{0,\rho} \right)^{\theta \lambda (m-1)} + t \left( \left\| u_0 \right\|_{0,\rho} \vee \left\| v_0 \right\|_{0,\rho} \right)^{m-1} \right] } \left\| u_0 - v_0 \right\|_{ 0,\rho}
		\end{aligned}
		\end{equation}
		for all $t>0$.
		\end{pro}
	\begin{proof}
	As usual, we will prove \eqref{stability-C-adm-bis} assuming that $u$ and $v$ are $ L^1(\mathbb{R}^N , \rho) \cap L^\infty(\mathbb{R}^N) $ solutions, since it is easily checked that the estimate is stable with respect to the truncation approximation (see e.g.~the proof of Lemma \ref{constr-weak}). For arbitrary $ R \ge R_\gamma $ (where $R_\gamma$ is defined in \eqref{est-R-gamma}) and $ z_R \in \partial B_R $, we perform a similar computation to \eqref{Phiphi-1} with $ \Phi \phi_R $ replaced by $ \psi_{R} $ as in \eqref{m-def-psiR}. This leads to
\begin{equation*}\label{m32-tr}
		\begin{aligned}
		\frac{d}{d t} \int_{\mathbb{R}^N} \left| u(t) - v(t) \right| \psi_R \, \rho \,  dx
		\le
		 & \, \int_{\mathbb{R}^N} \left| u(t)^m - v(t)^m \right| \left(\Delta \psi_R\right)  dx \\
		\leq & \, \frac{C}{R^\gamma}  \int_{ B_{R^{\gamma/2}}(z_R)  } \left| u(t)^m - v(t)^m \right| dx \\
		\leq & \,  {C} \left( \left\| u(t) \right\|_{L^\infty\left(\mathbb{R}^N\right)} \vee \left\| v(t) \right\|_{L^\infty\left(\mathbb{R}^N\right)} \right)^{m-1}  \int_{ B_{R^{\gamma/2}}(z_R) } \left| u(t) - v(t) \right| \rho \, dx \,  .
		\end{aligned}
		\end{equation*}	
Integrating in time as in \eqref{m33-t-10}, we end up with
				\begin{equation*}\label{m33-t-100}
				\begin{aligned}
& \, \int_{ B_{\frac{1}{2} {R^{\gamma / 2}} } (z_R) } \left| u(t)-v(t) \right| \rho \, dx  \\
\leq & \,  \int_{ B_{R^{\gamma / 2}(z_R) } } \left| u_0-v_0 \right| \rho \, dx \\
		   & +  C \int_0^t \left( \left\| u(s) \right\|_{L^\infty\left(\mathbb{R}^N\right)} \vee \left\| v(s) \right\|_{L^\infty\left(\mathbb{R}^N\right)} \right)^{m-1}  \left( \int_{B_{R^{\gamma/2}}(z_R)} \left| u(s)-v(s) \right| \rho \, dx \right) ds \, .
		\end{aligned}
		\end{equation*}
By repeating the same computations with $ \psi_R $ replaced by $ \phi_{R_\gamma} $ (recall that $ R_\gamma $ is defined in \eqref{est-R-gamma}),  multiplying by $ R^{-\gamma(N-2)/2} $, taking suprema over $R \ge  R_\gamma$ and summing the two estimates, we can infer the analogue of \eqref{m34-v}, that is
				\begin{equation*}\label{m340-v}
g(t) \leq {C} \left[ g(0) + \int_0^t \left( \left\| u(s) \right\|_{L^\infty\left(\mathbb{R}^N\right)} \vee \left\| v(s) \right\|_{L^\infty\left(\mathbb{R}^N\right)} \right)^{m-1}  g(s) \, ds \right] \qquad \forall t> 0 \, ,
\end{equation*}		
		with $g(t):=\left\| u(t)-v(t) \right\|_{0,\rho} $. Estimate \eqref{stability-C-adm-bis} is then a simple consequence of the smoothing effect \eqref{p-transl} applied to both $u$ and $v$ and the integral version of Gr\"onwall's inequality (note that $ g(t) $ is continuous thanks to the $ L^1(\mathbb{R}^N,\rho) $ continuity of approximate solutions).
	\end{proof}	
	
	\section{Asymptotic convergence}\label{unique}
	
	Our first goal in this section is to prove Theorem \ref{main thm}. Before this, however, we need to establish a crucial convergence property of the solutions of
		\begin{equation}\label{wpme-resc-u0}
	\begin{cases}
	\rho_k \, \partial_t v_k = \Delta \!\left(v_k^m\right) & \text{in } \mathbb{R}^N \times (0 , + \infty)\, , \\
	v_k = u_0 & \text{on } \mathbb{R}^N\times \{0\} \, ,
	\end{cases}
	\end{equation}
	 towards the solution of
	 		\begin{equation}\label{wpme-resc-u1}
	\begin{cases}
	c \, |x|^{-\gamma} \, u_t = \Delta \!\left(u^m\right) & \text{in } \mathbb{R}^N \times (0 , + \infty)\, , \\
	u = u_0 & \text{on } \mathbb{R}^N\times \{0\} \, ,
	\end{cases}
	\end{equation}
at least for a ``good'' class of nonnegative initial data. We recall that here, and in the sequel, it is understood that $ \rho_k $ is the rescaled density \eqref{rescaled-weight} satisfying \eqref{weight-cond2}, for a fixed but arbitrary sequence $ \xi_k \to + \infty $. Furthermore, solutions to \eqref{wpme-resc-u0} or \eqref{wpme-resc-u1} are meant in the weak energy sense, according to \cite[Definition 3.1]{MP}.
	
	\begin{lem}\label{h1 compact}
		Let $ u_0 \in L^1\!\left( \mathbb{R}^N , |x|^{-\gamma} \right) \cap L^\infty\!\left( \mathbb{R}^N \right) $, with $ u_0 \ge 0 $. Let $ v_k $ and $ u $ be the corresponding solutions of \eqref{wpme-resc-u0} and \eqref{wpme-resc-u1}, respectively. Then
		\begin{equation}\label{L1-conv-cont}
	\lim_{k \to \infty} v_k	= u \quad \text{in } C\!\left( [0,T] ; L^p\!\left( \mathbb{R}^N , |x|^{-\gamma} \right) \right)  \qquad \forall T>0 \, , \ \forall p \in [1,\infty)  \, .
		\end{equation}
	\end{lem}
The very technical proof of this lemma is postponed until Appendix \ref{technical}. We are now in position to prove our main result regarding global weighted $L^p$ convergence for $p\in[1,\infty)$, namely Theorem \ref{main thm}.

\begin{proof}[Proof of Theorem \ref{main thm} and Corollary \ref{main thm cor}]
For a fixed $ n \in \mathbb{N} $, let us call $ u_{k,n} $ the solution of problem \eqref{wpme-resc-u0} taking the truncated initial datum
	$$
  \eta_n   := \left[ \left( b \left| x \right|^{-\alpha} \right) \wedge n \right] \chi_{B_n} \, ,
	$$
	which, of course, is nonnegative and belongs to $ L^1\!\left( \mathbb{R}^N , |x|^{-\gamma} \right) \cap L^\infty\!\left( \mathbb{R}^N \right) $, and $ v_n $ the solution of \eqref{wpme-resc-u1} taking the same datum. Also, we continue to let $ u_k $ denote the solution of the rescaled problem \eqref{wpme-resc} and $ \mathcal{U}_\alpha $ the self-similar solution of \eqref{singularWPME}.
	
	Our goal is to prove that $ u_k \to \mathcal{U}_\alpha $ as $ k \to \infty $ for a fixed time; we will explain at the end of the proof how this is equivalent to the thesis. To begin with, we write
		\begin{equation}\label{main-ineq-1-try-1}
	\left\| u_k(t_0)- \mathcal{U}_\alpha(t_0) \right\|_{L^1\left( \mathbb{R}^N , \Phi \left| x \right|^{-\gamma} \right)} = \left\| u_k(t_0)- \mathcal{U}_\alpha(t_0) \right\|_{L^1\left( B_\varepsilon , \Phi \left| x \right|^{-\gamma} \right)} + \left\| u_k(t_0)- \mathcal{U}_\alpha(t_0) \right\|_{L^1\left(B_\varepsilon^c , \Phi \left| x \right|^{-\gamma} \right)} .
		\end{equation}
Let us treat the two terms on the right-hand side separately. Thanks to \eqref{weight-cond-scaled} and Proposition \ref{p-transl}, for every $ t_0>0 $ and $ \varepsilon>0 $ we have
			\begin{equation}\label{main-ineq-1-try-2}	
			\begin{aligned}
& \left\| u_k(t_0)- \mathcal{U}_\alpha(t_0) \right\|_{L^1\left( B_\varepsilon , \Phi \left| x \right|^{-\gamma} \right)} \\
\le  &\left( \left\| u_k(t_0) \right\|_{L^\infty\left( \mathbb{R}^N \right)} \vee \left\| \, \mathcal{U}_\alpha(t_0)  \right\|_{L^\infty\left( \mathbb{R}^N \right)} \right) \left\| \Phi \right\|_{L^\infty\left( \mathbb{R}^N \right)} \tfrac{\omega_{N-1}}{N-\gamma} \, \varepsilon^{N-\gamma} \\
 \le & \left\{ \left[ C_2 \left(  t_0^{-\lambda} \left\| u_{0k} \right\|_{0,\rho_k}^{\theta\lambda} + \left\| u_{0k}  \right\|_{0,\rho_k} \right) \right] \vee \left\| \, \mathcal{U}_\alpha(t_0)  \right\|_{L^\infty\left( \mathbb{R}^N \right)} \right\} \left\| \Phi \right\|_{L^\infty\left( \mathbb{R}^N \right)} \tfrac{\omega_{N-1}}{N-\gamma} \, \varepsilon^{N-\gamma}
\end{aligned}
		\end{equation}	
	and
				\begin{equation}\label{main-ineq-1-try-3}	
			\begin{aligned}
			& \left\| u_k(t_0)- \mathcal{U}_\alpha(t_0) \right\|_{L^1\left(B_\varepsilon^c , \Phi \left| x \right|^{-\gamma} \right)} \\
			 \le & \left\| u_k(t_0)- u_{k,n}(t_0) \right\|_{L^1\left(B_\varepsilon^c , \Phi \left| x \right|^{-\gamma} \right)} + \left\| u_{k,n}(t_0) - v_n(t_0) \right\|_{L^1\left(B_\varepsilon^c , \Phi \left| x \right|^{-\gamma} \right)} \\
	& \, +  \left\| v_n(t_0)- \mathcal{U}_\alpha(t_0) \right\|_{L^1\left(B_\varepsilon^c , \Phi \left| x \right|^{-\gamma} \right)}  \\
\le 	& \, \frac{\left( 1+\varepsilon \right)^\gamma}{\underline{C} \, \varepsilon^\gamma} \left\|  u_k(t_0)- u_{k,n}(t_0)  \right\|_{L^1\left(B_\varepsilon^c , \Phi \rho_k \right)} + \left\| \Phi \right\|_{L^\infty\left( \mathbb{R}^N \right)}  \left\| u_{k,n}(t_0) - v_n(t_0) \right\|_{L^1\left(B_\varepsilon^c , \left| x \right|^{-\gamma} \right)} \\
	& \, +  \left\| v_n(t_0)- \mathcal{U}_\alpha(t_0) \right\|_{L^1\left(B_\varepsilon^c , \Phi \left| x \right|^{-\gamma} \right)} ,
			\end{aligned}
		\end{equation}	
where $ \omega_{N-1} $ stands for the total hypersurface of the $(N-1)$-dimensional unitary sphere. By virtue of Proposition \ref{cauchy-adm} and $ 0 \le \eta_n \le b|x|^{-\alpha}  $, we infer
		\begin{equation}\label{main-ineq-2}
		\begin{aligned}
	 & \left\|  u_k(t_0)- u_{k,n}(t_0)  \right\|_{L^1\left(B_\varepsilon^c , \Phi \rho_k \right)} \\
	  \le & \, e^{C_3 \left[ t_0^{\theta\lambda} \left( \left\| u_{0k} \right\|_{0,\rho_k} \vee \left\| \eta_n \right\|_{0,\rho_k} \right)^{\theta \lambda (m-1)} + t_0 \left( \left\| u_{0k} \right\|_{0,\rho_k} \vee \left\| \eta_n \right\|_{0,\rho_k} \right)^{m-1} \right] } \left\| u_{0k} - \eta_n \right\|_{L^1\left( \mathbb{R}^N , \Phi \rho_k \right)}
	 \end{aligned}
		\end{equation}
and
		\begin{equation}\label{main-ineq-3}
		\begin{aligned}
	\left\| v_n(t_0)- \mathcal{U}_\alpha(t_0) \right\|_{L^1\left(B_\varepsilon^c , \Phi \left| x \right|^{-\gamma} \right)} \le  \, e^{C_3 \left( t_0^{\theta\lambda} \left\| b \left| x \right|^{-\alpha} \right\|_{0,c \left| x \right|^{-\gamma}}^{\theta \lambda (m-1)} + t_0 \left\| b \left| x \right|^{-\alpha} \right\|_{0,c \left| x \right|^{-\gamma}}^{m-1} \right) } \left\| \eta_n - b \left| x \right|^{-\alpha} \right\|_{L^1\left( \mathbb{R}^N , \Phi |x|^{-\gamma} \right)} ,
	 \end{aligned}
		\end{equation}
where in the latter estimate the constant $C_3$ depends on $c$ in place of $ \underline{C},\overline{C} $. From \eqref{weight-cond-scaled} and the definition of $ \eta_n $, it is clear that
		$$
	\sup_{k,n \in \mathbb{N}} \left\| \eta_n \right\|_{0,\rho_k}	< + \infty \qquad \text{and} \qquad \sup_{n \in \mathbb{N}} \left\| \eta_n \right\|_{0,c \left| x \right|^{-\gamma}} 	< + \infty \, ,
		$$
		whereas $ \sup_{k \in \mathbb{N}} \left\| u_{0k} \right\|_{0,\rho_k} < +\infty $ by assumption. As a result, in view of \eqref{main-ineq-1-try-1}--\eqref{main-ineq-3} we can assert that there exists a constant $ \mathsf{C}>0 $, independent of $ k,n,\varepsilon $, such that
\begin{equation*}\label{main-ineq-4}
\begin{aligned}
 & \left\| u_k(t_0)- \mathcal{U}_\alpha(t_0) \right\|_{L^1\left( \mathbb{R}^N , \Phi \left| x \right|^{-\gamma} \right)} \\
 \le & \, \mathsf{C} \left[ \varepsilon^{N-\gamma} + \frac{\left( 1+\varepsilon \right)^\gamma}{\varepsilon^\gamma} \left\| u_{0k} - \eta_n \right\|_{L^1\left( \mathbb{R}^N , \Phi \rho_k \right)} + \left\| u_{k,n}(t_0) - v_n(t_0) \right\|_{L^1\left(\mathbb{R}^N , \left| x \right|^{-\gamma} \right)} \right. \\
& \left. \quad \ \  + \left\| \eta_n - b \left| x \right|^{-\alpha} \right\|_{L^1\left( \mathbb{R}^N , \Phi |x|^{-\gamma} \right)}^{\phantom{A^{A^A}}} \right] .
\end{aligned}
\end{equation*}
By letting first $ k \to \infty $, and taking advantage of \eqref{conv-resc-datum} and Lemma \ref{h1 compact}, we end up with
\begin{equation*}\label{main-ineq-5}
\begin{aligned}
 \limsup_{k \to \infty} \left\| u_k(t_0)- \mathcal{U}_\alpha(t_0) \right\|_{L^1\left( \mathbb{R}^N , \Phi \left| x \right|^{-\gamma} \right)} \le  \mathsf{C} \left[ \varepsilon^{N-\gamma} + \left( \frac{\left( 1+\varepsilon \right)^\gamma}{\varepsilon^\gamma} \, c + 1 \right) \left\| \eta_n - b \left| x \right|^{-\alpha} \right\|_{L^1\left( \mathbb{R}^N , \Phi |x|^{-\gamma} \right)} \right]  ,
\end{aligned}
\end{equation*}
whence
\begin{equation}\label{main-ineq-6}
\lim_{k \to \infty} \left\| u_k(t_0)- \mathcal{U}_\alpha(t_0) \right\|_{L^1\left( \mathbb{R}^N , \Phi \left| x \right|^{-\gamma} \right)}  = 0
\end{equation}
upon letting $ n \to \infty $ (note that $ \eta_n \to b \left| x \right|^{-\alpha} $ monotonically and \eqref{A3} holds) and finally $ \varepsilon \to 0 $. Because the sequence $ \left\{ u_k(t_0)- \mathcal{U}_\alpha(t_0) \right\}_k $ is bounded in $ L^\infty\!\left( \mathbb{R}^N \right) $, from \eqref{main-ineq-6} we immediately deduce
\begin{equation}\label{main-ineq-7}
\lim_{k \to \infty} \left\| u_k(t_0)- \mathcal{U}_\alpha(t_0) \right\|_{L^p\left( \mathbb{R}^N , \Phi \left| x \right|^{-\gamma} \right)}  = 0 \qquad \forall p \in [1,\infty) \, .
\end{equation}
Now we simply observe that with the choice $ t_0=1 $ we have $ \mathcal{U}_\alpha(x,1) = g_\alpha(|x|) $ and
$$
u_k(x,1) =  \xi^\alpha_k \, u\!\left(\xi_k x , \xi_k^{\alpha(m-1)+2-\gamma}  \right) ,
$$
therefore \eqref{main thm eq} follows upon choosing $ t_0 =1 $ and $ \xi_k = \zeta_k^{\lambda_\alpha / \alpha} $ in \eqref{main-ineq-7}, for an arbitrary sequence $ \zeta_k \to +\infty $. As for \eqref{main thm eq_bis}, it is a direct consequence of \eqref{main thm eq} via the change of variables $ y = t^{{\lambda_\alpha}/{\alpha}} x $, recalling that $ \mathcal{U}_\alpha $ satisfies the self-similar identity \eqref{scale invar}.
	\end{proof}

{Now we prepare to prove our uniform convergence result, Theorem \ref{thm-unif-conv}. To begin with, we prove convergence of $u_k$ to $\mathcal{U}_\alpha$ at a fixed time with respect to the norm $ \| \cdot \|_{0,|x|^{-\gamma}} $. We call this Claim 1, in which the key condition \eqref{second-req-u0-re} is crucially used. We state it first because, in particular, it allows us to identify the limit of $ \{ u_k \} $ as  $ \mathcal{U}_\alpha $ in subsequent compactness arguments. The following part of the proof mainly exploits the strategy of \cite[Section 5]{RV2}. A further step, which we call Claim 2, is to show that our rescaled solutions $\{u_k\}$ uniformly satisfy H\"{o}lder estimates in compact sets \emph{away from the origin}, so they uniformly converge in such sets. This is where we use the extra condition \eqref{rho-extra} on the weight. Next, in Claim 3 and Claim 4, we employ subtle local barrier arguments to extend such uniform convergence up to the origin. To handle tail convergence, our proof diverges strongly from that of \cite{RV2}. To conclude, we once again carefully apply Claim 1 and some \emph{a priori} estimates proven in Section \ref{exist} to show that the family $\{u_k\}$ is uniformly small far away from the origin.}

{Before we begin, let us state the following key lemma, which will be used in Claims 3 and 4, and whose proof is postponed until Appendix \ref{technical}.}

\begin{lem}\label{lowbar-lemma-lemma}
Given $   \mathcal{C}>0 $,  $ \ell  <   \mathcal{C}  $ (resp.~$ \ell \ge  \mathcal{C}  $), consider the solution $v$ to the nonhomogeneous Cauchy-Dirichlet problem
\begin{equation}\label{lowbar-lemma}
\begin{cases}
	\overline{C}\left|x\right|^{-\gamma} v_t = \Delta\!\left( v^m \right) & \text{in } B_1 \times(0,+\infty) \, , \\
	v = \mathcal{C} & \text{on } \partial B_1 \times(0,+\infty) \, ,\\
	v = \ell & \text{on } B_1 \times \{ 0 \} \, , \\
\end{cases}
\end{equation}
where $ \overline{C} $ is the upper constant appearing in \eqref{weight-cond}. Then $ v_t \ge 0 $ (resp.~$ v_t \le 0  $) and
\begin{equation}\label{conv-unif-v}
\lim_{t \to +\infty}\left\| v(t) - \mathcal{C} \right\|_{L^\infty\left( B_1 \right)} = 0 \, .
\end{equation}
\end{lem}

\begin{proof}[Proof of Theorem \ref{thm-unif-conv}]

We will split the proof into several claims. Throughout, we tacitly exploit the fact that \eqref{second-req-u0-re} trivially implies \eqref{second-req-u0}, therefore $   \left\| u_{0k}  \right\|_{0,\rho_k} $ is uniformly bounded with respect to $k$.


\medskip

\noindent{\textbf{Claim 1:}} \emph{For every $t_0>0$, we have}
\begin{equation}\label{conv-0-norm}
	\lim_{k \to \infty} \left\| u_k(t_0) - \mathcal{U}_\alpha(t_0) \right\|_{0,|x|^{-\gamma}} = 0 \, .
\end{equation}
We use the same notation as in the proof of Theorem \ref{main thm}. By proceeding similarly to \eqref{main-ineq-1-try-2}--\eqref{main-ineq-1-try-3}, for every $\varepsilon>0$ we obtain:
\begin{equation}\label{main-ineq-21}
	\begin{aligned}
		\left\| u_k(t_0)- \mathcal{U}_\alpha(t_0) \right\|_{0,|x|^{-\gamma}}
		\le & \left\| \left( u_k(t_0)- \mathcal{U}_\alpha(t_0) \right) \chi_{B_\varepsilon} \right\|_{0,|x|^{-\gamma}} + \left\| \left( u_k(t_0)- \mathcal{U}_\alpha(t_0) \right) \chi_{B^c_\varepsilon} \right\|_{0,|x|^{-\gamma}} \\
		\le &  \left( \left\| u_k(t_0) \right\|_{L^\infty\left( \mathbb{R}^N \right)} \vee \left\| \, \mathcal{U}_\alpha(t_0)  \right\|_{L^\infty\left( \mathbb{R}^N \right)} \right) \tfrac{\omega_{N-1}}{N-\gamma} \, \varepsilon^{N-\gamma} \\
		& \, + \left\| \left( u_k(t_0)- u_{k,n}(t_0) \right) \chi_{B_\varepsilon^c} \right\|_{0,|x|^{-\gamma}}\\
		& \,  + \left\| \left( u_{k,n}(t_0) - v_n(t_0) \right) \chi_{B_\varepsilon^c} \right\|_{0,|x|^{-\gamma}}  +  \left\| v_n(t_0)- \mathcal{U}_\alpha(t_0) \right\|_{0,|x|^{-\gamma}} \\
		\le & \left\{ \left[ C_2 \left(  t_0^{-\lambda} \left\| u_{0k} \right\|_{0,\rho_k}^{\theta\lambda} + \left\| u_{0k}  \right\|_{0,\rho_k} \right) \right] \vee \left\| \, \mathcal{U}_\alpha(t_0)  \right\|_{L^\infty\left( \mathbb{R}^N \right)}  \right\} \tfrac{\omega_{N-1}}{N-\gamma} \, \varepsilon^{N-\gamma} \\	
		& \, + \frac{\left( 1+\varepsilon \right)^\gamma}{\underline{C} \, \varepsilon^\gamma} \left\| \left( u_k(t_0)- u_{k,n}(t_0) \right) \chi_{B_\varepsilon^c} \right\|_{ 0 , \rho_k } \\
		& \, + \left\| \left( u_{k,n}(t_0) - v_n(t_0) \right) \chi_{B_\varepsilon^c} \right\|_{0,|x|^{-\gamma}}  +  \left\| v_n(t_0)- \mathcal{U}_\alpha(t_0) \right\|_{0,|x|^{-\gamma}} .
	\end{aligned}
\end{equation}
Thanks to Proposition \ref{cauchy-adm-0norm}, we infer
\begin{equation}\label{main-ineq-22}
	\begin{aligned}
		\left\|  u_k(t_0)- u_{k,n}(t_0)  \right\|_{0 , \rho_k }
		\le & \, C_4 \, e^{C_5 \left[ t_0^{\theta\lambda} \left( \left\| u_{0k} \right\|_{0,\rho_k} \vee \left\| \eta_n \right\|_{0,\rho_k} \right)^{\theta \lambda (m-1)} + t_0 \left( \left\| u_{0k} \right\|_{0,\rho_k} \vee \left\| \eta_n \right\|_{0,\rho_k} \right)^{m-1} \right] } \left\| u_{0k} - \eta_n \right\|_{0,\rho_k }
	\end{aligned}
\end{equation}
and
\begin{equation}\label{main-ineq-23}
	\begin{aligned}
 \left\| v_n(t_0)- \mathcal{U}_\alpha(t_0) \right\|_{0,|x|^{-\gamma}}
		\le  C_4 \,e^{C_5 \left[ t_0^{\theta\lambda} \left\| b \left| x \right|^{-\alpha} \right\|_{0,c \left| x \right|^{-\gamma}}^{\theta \lambda (m-1)} + t_0 \left\| b \left| x \right|^{-\alpha} \right\|_{0,c \left| x \right|^{-\gamma}}^{m-1} \right] } \left\| \eta_n - b \left| x \right|^{-\alpha} \right\|_{0,|x|^{-\gamma}} ,
	\end{aligned}
\end{equation}
where in the latter estimate the constants $C_4,C_5$ depend on $c$ in place of $ \underline{C},\overline{C} $. In view of \eqref{main-ineq-21}--\eqref{main-ineq-23}, by exploiting the same norm bounds on $ \eta_n $ and $ u_{0k} $ as in the proof of Theorem \ref{main thm}, we can find a constant $ \mathsf{C}>0 $, independent of $ k,n,\varepsilon $, such that
\begin{equation*}\label{main-ineq-24}
	\begin{aligned}
		& \left\| u_k(t_0)- \mathcal{U}_\alpha(t_0) \right\|_{0,|x|^{-\gamma}} \\
		\le & \, \mathsf{C} \left[ \varepsilon^{N-\gamma} + \frac{\left( 1+\varepsilon \right)^\gamma}{\varepsilon^\gamma} \left\| u_{0k} - \eta_n \right\|_{0,\rho_k} + \left\| u_{k,n}(t_0) - v_n(t_0) \right\|_{L^1\left(\mathbb{R}^N , \left| x \right|^{-\gamma} \right)}  + \left\| \eta_n - b \left| x \right|^{-\alpha} \right\|_{0,|x|^{-\gamma}}^{\phantom{A^{A^A}}} \right] ,
	\end{aligned}
\end{equation*}
where we used the straightforward inequality $ \| \cdot \|_{0,|x|^{-\gamma}} \le \| \cdot \|_{L^1\left(\mathbb{R}^N , \left| x \right|^{-\gamma} \right)}  $. By letting first $ k \to \infty $, and taking advantage of \eqref{second-req-u0-re} and Lemma \ref{h1 compact}, we end up with
\begin{equation*}\label{main-ineq-25}
	\begin{aligned}
		& \, \limsup_{k \to \infty} \left\| u_k(t_0)- \mathcal{U}_\alpha(t_0) \right\|_{0,|x|^{-\gamma}} \le \mathsf{C} \left[ \varepsilon^{N-\gamma} + \left( \frac{\left( 1+\varepsilon \right)^\gamma}{\varepsilon^\gamma} \, c + 1 \right) \left\| \eta_n - b \left| x \right|^{-\alpha} \right\|_{0,|x|^{-\gamma}} \right]  ,
	\end{aligned}
\end{equation*}
so that \eqref{conv-0-norm} follows upon letting first $ n \to \infty $ and then $ \varepsilon \to 0 $.

\medskip

\noindent{\textbf{Claim 2:}} \emph{For every $t_0>0$ and every $R>r>0$, we have}
\begin{equation}\label{Dib}
\lim_{k \to \infty} \left\| u_k(t_0) - \mathcal{U}_\alpha(t_0) \right\|_{L^\infty\left( B_{R} \setminus B_r \right)} = 0 \, .
\end{equation}
Such a claim is implied by the more general result that will be independently used in the sequel:
\begin{equation}\label{Dib2}
	u_k \underset{k \to \infty}{\longrightarrow} \mathcal{U}_\alpha \quad \textrm{uniformly on } Q \, ,
 \end{equation}
where $ Q = \Omega \times (t_1,t_2) $ for an arbitrary domain $ \Omega \Subset \left(\R^N\setminus\{0\}\right) $ and $ t_2>t_1>0 $. Indeed, as in \cite[page 500]{RV2}, one may rewrite
\eqref{wpme-resc} as the unweighted degenerate parabolic equation of porous-medium type
\[
\partial_t u_k =  \operatorname{div}\! \left(\rho_k^{-1}\,\nabla(u_k^m)\right)-\nabla\!\left(\rho_k^{-1}\right)\cdot\nabla\!\left(u_k^m\right) .
\]
 In order to apply standard local H\"older estimates for weak solutions according to~\cite[Theorem 1.2]{DiB} \emph{uniformly in $k$}, we require the following structural assumptions to hold:
\[
c_0\leq\rho_k^{-1}\leq c_1 \quad \text{in } Q \, , \qquad \norm{\nabla (\rho_k^{-1})}_{L^\infty(Q)}\leq c_2 \, ,  \qquad \norm{\nabla (u_k^m)}_{L^2(Q)}\leq c_3 \, ,
\]
for suitable constants $ c_i > 0$ that do not depend on $k$.
Recalling that $Q$ stays (compactly) away from the origin, the first two estimates easily follow from  \eqref{weight-cond} and \eqref{rho-extra}. The final estimate follows from Lemma \ref{constr-weak}, Proposition \ref{p-transl}, and \eqref{second-req-u0}. Then \eqref{Dib2} is a consequence of the aforementioned H\"{o}lder estimates, the Ascoli-Arzel\`{a} theorem, and the identification of the limit ensured by Claim 1.

\medskip

\noindent{\textbf{Claim 3:}} \emph{For every $ t_0>0 $ and $\delta>0$, there exist an $\varepsilon>0$ and a $\mathsf{K}>0$ such that}
\begin{equation}\label{loc conv 1}
 u_k(x,t_0)-\mathcal{U}_\alpha(x,t_0)  > -\delta \qquad \forall x \in B_\varepsilon \setminus \{0\} \, , \ \forall k>\mathsf{K} \, .
\end{equation}
To begin, let us fix $\delta>0$ and consider the solution $v_1$ to \eqref{lowbar-lemma} with
\begin{equation*}
	\mathcal{C}=\mathcal{U}_\alpha(0,t_0)-\tfrac 1 2 \delta >0 \, , \qquad \ell = - \sup_{k \in\mathbb{N}} \, \sup_{\frac {t_0}{2} \le  \tau \le t_0 } \norm{u_k(\tau)}_{L^\infty\left(\R^N\right)},
\end{equation*}
for a fixed $ \tfrac{t_0}{2} <\tau<t_0$ that will be chosen below. The finiteness of such a supremum is guaranteed by Proposition \ref{p-transl} and \eqref{second-req-u0}. Next, by Lemma \ref{lowbar-lemma-lemma}, there exists a positive time $T$ such that
\begin{equation}\label{almost limit}
	\mathcal{U}_\alpha(0,t_0)- \tfrac 3 4 \delta \leq v_1 \leq\mathcal{U}_\alpha(0,t_0)- \tfrac 1 2 \delta \qquad \text{in } B_1 \times [T,+\infty)    \, .
\end{equation}
To introduce the $\varepsilon$ present in \eqref{loc conv 1}, we consider the function $v_\varepsilon(x,t)=v_1\!\left(\varepsilon^{-1}x,\varepsilon^{\gamma-2}t \right)$ that solves the rescaled problem
\begin{equation}\label{lowbar-resc}
	\begin{cases}
		\overline{C}\left|x\right|^{-\gamma} v_t = \Delta\!\left( v^m \right) & \text{in } B_\varepsilon \times(0,+\infty) \, , \\
		v = \mathcal{C} & \text{on } \partial B_\varepsilon \times(0,+\infty) \, ,\\
		v = \ell & \text{on } B_\varepsilon \times \{ 0 \} \, . \\
	\end{cases}
\end{equation}
Then from \eqref{almost limit} it clearly follows that
\begin{equation}\label{almost limit resc}
	\mathcal{U}_\alpha(0,t_0)- \tfrac 3 4 \delta \leq v_\varepsilon \leq\mathcal{U}_\alpha(0,t_0)-\tfrac 1 2 \delta  \qquad \text{in } B_\varepsilon \times \left[T\varepsilon^{2-\gamma},+\infty\right)  .
\end{equation}
In particular, we shall choose $\varepsilon$ so small and $ \tau $ close enough to $ t_0 $ that both the following conditions hold:
\begin{align}[left=\empheqlbrace]
	& T\varepsilon^{2-\gamma}+\tau<t_0 \label{eps-small-2} \, ,\\
	&  \left| \, \mathcal{U}_\alpha(x,t) - \mathcal{U}_\alpha(0,t_0) \right| < \tfrac 1 4 \delta \qquad \forall(x,t)\in  B_\varepsilon\times(\tau , t_0)\label{lower-approx} \, ,
\end{align}
where the second inequality just follows by the continuity of $ \mathcal{U}_\alpha $, thanks to Theorem \ref{selfsim-sol}.

Due to the monotonicity property in Lemma \ref{lowbar-lemma-lemma} and \eqref{weight-cond-scaled}, $v_\varepsilon$ is a subsolution to \eqref{lowbar-resc} with $\overline{C}|x|^{-\gamma}$ replaced by $\rho_k$ (i.e.~the equation for $u_k$). Therefore we may apply local comparison (see e.g.~\cite{GMPo} and references therein) on the cylinder $ \overline{B}_\varepsilon\times [\tau,t_0]$, with the function $V_\varepsilon(x,t)=v_\varepsilon(x,t-\tau)$ serving as a uniform lower barrier to $u_k$. To this end, we must compare $u_k$ and $V_\varepsilon$ on the lateral boundary $\partial B_\varepsilon\times(\tau,t_0)$ and on the base $B_\varepsilon\times\{\tau\}$ of the cylinder. First, for boundary comparison, we can exploit~\eqref{Dib2} to conclude that there exists a $\mathsf{K}>0$ such that
\begin{equation}\label{Dib-cons}
	\left| u_k - \mathcal{U}_\alpha \right| \le  \tfrac 1 4 \delta \qquad\textrm{on } \partial B_\varepsilon\times(\tau,t_0) \, , \ \forall k>\mathsf{K} \, .
\end{equation}
 Combining this estimate with \eqref{lower-approx} and recalling the definition of $\mathcal{C}$ gives the conclusion that $u_k \geq V_\varepsilon=\mathcal{C}$ on the lateral boundary. As regards the initial datum, by the definition of $\ell $ we have that $u_k \geq V_\varepsilon = \ell $ on $B_\varepsilon \times \{ \tau \} $. Applying the above mentioned comparison yields
\begin{equation}\label{comparison-res}
	u_k \geq V_\varepsilon  \qquad\textrm{on } \overline{B}_\varepsilon\times[\tau,t_0] \, , \ \forall k>\mathsf{K} \, .
\end{equation}
Combining \eqref{comparison-res} and the leftmost inequality in \eqref{almost limit resc} (also recall \eqref{eps-small-2}), we conclude
\begin{equation}\label{local-lower-bd-cyl}
	u_k\geq \mathcal{U}_\alpha(0,t_0)- \tfrac 3 4 \delta \qquad\textrm{on } \overline{B}_\varepsilon\times  [T\varepsilon^{2-\gamma} + \tau  ,t_0] \, , \ \forall k>\mathsf{K} \, .
\end{equation}
The lower bound \eqref{loc conv 1} follows at once by applying \eqref{local-lower-bd-cyl} at $t=t_0$ and again \eqref{lower-approx}. Note that the origin must be removed from \eqref{loc conv 1} because, in view of Claim 2, we are able to ensure that each $u_k$ is continuous only away from the origin.

\medskip

\noindent{\textbf{Claim 4:}}
 \emph{For every $ t_0>0 $ and $\delta>0$, there exist an $\varepsilon>0$ and a $\mathsf{K}>0$ such that}
\begin{equation}\label{loc conv 2}
 u_k(x,t_0)-\mathcal{U}_\alpha(x,t_0)  < \delta \qquad \forall x \in B_\varepsilon \setminus \{0\} \, , \ \forall k>\mathsf{K} \, .
\end{equation}
The proof of \eqref{loc conv 2} is similar to to that of \eqref{loc conv 1}, so we will only stress the main differences. First of all, we modify the constants $ \mathcal{C},\ell $ as follows:
\begin{equation*}
	\mathcal{C}=\mathcal{U}_\alpha(0,t_0)+\tfrac 1 2 \delta \, , \qquad \ell =  \bigg( \sup_{k \in\mathbb{N}} \, \sup_{\frac {t_0}{2} \le  \tau \le t_0 } \norm{u_k(\tau)}_{L^\infty\left(\R^N\right)} \bigg) \vee \mathcal{C} \, ,
\end{equation*}
whereas $ \tfrac{t_0}{2} <\tau<t_0$ is still a suitable time to be chosen. The function $ v_\varepsilon $ is as before the solution to the rescaled problem \eqref{lowbar-resc} which, by Lemma \ref{lowbar-lemma-lemma}, is now monotone decreasing in time and satisfies
\begin{equation}\label{almost limit bis}
	\mathcal{U}_\alpha(0,t_0) +  \tfrac 1 2 \delta \leq v_\varepsilon \leq\mathcal{U}_\alpha(0,t_0) + \tfrac 3 4 \delta  \qquad \text{in } B_\varepsilon \times \left[T\varepsilon^{2-\gamma},+\infty\right)  .
\end{equation}
Now we choose $ \varepsilon $ and $\tau$ exactly as in \eqref{eps-small-2}--\eqref{lower-approx}. The opposite time monotonicity makes $ v_\varepsilon $ a supersolution to \eqref{lowbar-resc} with $ \overline{C}|x|^{-\gamma} $ replaced by $ \rho_k $, so that the function $ V_\varepsilon $ will actually serve as a uniform upper barrier to $ u_k $ on the same small cylinder. Lateral boundary comparison between $ u_k $ and $ V_\varepsilon $ still holds by virtue of \eqref{lower-approx} and \eqref{Dib-cons}, whereas the base comparison is just a consequence of the new definition of $ \ell $. Hence, by comparison we end up with
\begin{equation*}\label{comparison-res-bis}
	u_k \leq V_\varepsilon  \qquad\textrm{on } \overline{B}_\varepsilon\times[\tau,t_0] \, , \ \forall k>\mathsf{K} \, .
\end{equation*}
Therefore, in view of \eqref{almost limit bis} and \eqref{lower-approx},  the conclusion follows exactly as in Claim 3.

\medskip

\noindent {\textbf{End of proof.}} First of all, thanks to \eqref{Dib}, \eqref{loc conv 1} and \eqref{loc conv 2}, we can conclude that $ \{ u_k(t_0) \}_k $ converges to $ \mathcal{U}_\alpha(t_0) $ locally uniformly, that is
\begin{equation*}\label{Dib-unif}
\lim_{k \to \infty} \left\| u_k(t_0) - \mathcal{U}_\alpha(t_0)  \right\|_{L^\infty\left( B_{R} \right)} = 0 \qquad \forall R>0 \, .
\end{equation*}
In order to prove global uniform convergence, that is
\begin{equation}\label{Dib-unif-bis}
\lim_{k \to \infty} \left\| u_k(t_0) - \mathcal{U}_\alpha(t_0)  \right\|_{L^\infty\left( \mathbb{R}^N \right)} = 0 \, ,
\end{equation}
since $ x \mapsto \mathcal{U}_\alpha(x,t_0) $ vanishes uniformly as $ |x| \to +\infty $, it is enough to show that for every $ \delta>0 $ there exist $ R_\delta>0 $ and $ \mathsf{K}_\delta>0 $ such that
\begin{equation}\label{Dib-unif-outer}
 \left\| u_k(t_0) \right\|_{L^\infty\left( B_{R_\delta}^c \right)} < \delta \qquad \forall k>\mathsf{K}_\delta \, .
\end{equation}
Once \eqref{Dib-unif-bis} is established, formula \eqref{main thm unif} (and so \eqref{main thm unif t}) will follow exactly as at the end of the proof of Theorem \ref{main thm}. To this aim, let us go back to estimate \eqref{M1-ter} applied to $ u \equiv u_k(t_0) $, which reads
	\begin{equation}\label{M1-ter-proof}
	\begin{aligned}
\left\|  u_k(t_0) \right\|_{L^\infty\left( B_{\frac 1 2 R^{\gamma / 2}}(z_R) \right)} \le & \, C \left[ \tfrac{1}{(m-1)t_0} + \left\| u_k(t_0) \right\|_{L^\infty\left(\mathbb{R}^N\right)}^{m-1} \right]^{ \frac{N-\gamma}{(N-\gamma)(m-1)+p_0(2-\gamma)} }   \\
 & \times \left[ \frac{ \int_{ B_{R^{\gamma / 2}}(z_R) } u_k(t_0)  \,  \rho_k \, dx}{{ R^{\frac{\gamma(N-2)}{2}} }  }  \right]^{ \frac{2-\gamma}{(N-\gamma)(m-1)+p_0(2-\gamma)}} \\
& \times \left\| u_k(t_0) \right\|_{L^\infty\left( \mathbb{R}^N \right)}^{\frac{(p_0-1)(2-\gamma)}{(N-\gamma)(m-1)+p_0(2-\gamma)}} ,
\end{aligned}
	\end{equation}
where $ p_0 \ge m+1 $ is any chosen exponent and $C>0$ is a general constant that does not depend on $ k $ and $ R \ge  R_\gamma $, that may vary from line to line. By virtue of Proposition \ref{p-transl}, \eqref{weight-cond-scaled} and  \eqref{second-req-u0-re}, we can rewrite \eqref{M1-ter-proof} as
	\begin{equation}\label{M1-ter-proof-2}
\left\|  u_k(t_0) \right\|_{L^\infty\left( B_{\frac 1 2 R^{\gamma / 2}}(z_R) \right)} \le C \left[ \frac{ \int_{ B_{R^{\gamma / 2}}(z_R) } u_k(t_0)  \, |x|^{-\gamma} \, dx}{{ R^{\frac{\gamma(N-2)}{2}} }  }  \right]^{ \frac{2-\gamma}{(N-\gamma)(m-1)+p_0(2-\gamma)}} .
	\end{equation}
	Note that, rigorously, estimate \eqref{M1-ter-proof-2} holds for solutions that are globally integrable and bounded, but by a standard approximation argument it is not difficult to check that it is still satisfied by constructed solutions.
Thanks to the triangle inequality and the definition of $ \| \cdot \|_{0,|x|^{-\gamma}} $ norm, we may deduce from \eqref{M1-ter-proof-2} and \eqref{u-alpha-2}  the following bound:
	\begin{equation}\label{M1-ter-proof-3}
	\begin{aligned}
& \left\|  u_k(t_0) \right\|_{L^\infty\left( B_{\frac 1 2 R^{\gamma / 2}}(z_R) \right)} \\
\le & \, C \left[ \frac{ \int_{ B_{R^{\gamma / 2}}(z_R) } \mathcal{U}_\alpha(t_0)  \, |x|^{-\gamma} \, dx}{{ R^{\frac{\gamma(N-2)}{2}} }  } + \left\| u_k(t_0) - \mathcal{U}_\alpha(t_0) \right\|_{0,|x|^{-\gamma}}  \right]^{ \frac{2-\gamma}{(N-\gamma)(m-1)+p_0(2-\gamma)}} \\
\le & \,C \left[ \frac{ \int_{ B_{R^{\gamma / 2}}(z_R) } |x|^{-\alpha-\gamma} \, dx}{{ R^{\frac{\gamma(N-2)}{2}} }  } + \left\| u_k(t_0) - \mathcal{U}_\alpha(t_0) \right\|_{0,|x|^{-\gamma}}  \right]^{ \frac{2-\gamma}{(N-\gamma)(m-1)+p_0(2-\gamma)}} \\
\le & \,C \left[ R^{-\alpha} + \left\| u_k(t_0) - \mathcal{U}_\alpha(t_0) \right\|_{0,|x|^{-\gamma}}  \right]^{ \frac{2-\gamma}{(N-\gamma)(m-1)+p_0(2-\gamma)}} ,
\end{aligned}
	\end{equation}
where in the last passage we used the fact that $ |x| $ is comparable to $ R $ in $ B_{R^{\gamma/2}}(z_R) $ for all $R \ge R_\gamma$. Finally, due to \eqref{conv-0-norm}, we can choose $R_\delta$ and $ \mathsf{K}_\delta $ so large that
$$
C \left[ R_\delta^{-\alpha} + \left\| u_k(t_0) - \mathcal{U}_\alpha(t_0) \right\|_{0,|x|^{-\gamma}}  \right]^{ \frac{2-\gamma}{(N-\gamma)(m-1)+p_0(2-\gamma)}} < \delta \qquad \forall k > \mathsf{K}_\delta \, ,
$$
whence \eqref{Dib-unif-outer} follows upon taking the supremum of \eqref{M1-ter-proof-3} over $ R>R_\delta $ and $ z_R \in \partial B_R $.
\end{proof}
\normalcolor

\section{Detailed ODE analysis and self-similarity} \label{ODE}
	We split the proof of well-posedness of \eqref{ode1} over the course of several steps. First, we derive via a spatial dilation the equivalence of \eqref{ode1} and an unweighted ODE problem. Then, using classical methods, we prove existence and uniqueness of local solutions to the latter. By exploiting basic properties of the equation, we can extend local solutions to a unique global solution. Finally, we prove that such a solution has the desired asymptotic decay rate; this will allow us to construct the self-similar solution to \eqref{singularWPME} and prove its main properties claimed in Theorem \ref{selfsim-sol}. Some of the passages below are inspired by the works of Gilding and Peletier in \cite{GP1,GP2} and by Gilding in \cite{G3}, where they study the corresponding $1$-dimensional and unweighted problem.
	
In the sequel, we take for granted that $ \alpha \in (0,N-\gamma) $, and for the sake of better readability we discuss the case $ c=1 $ only, such a constant playing no significant role in the equation.
	
	\subsection{Reduction to an unweighted ODE and main result}
		First of all, we show that \eqref{ode1} can be easily transformed into an analogous problem without the presence of the weight $ r^{-\gamma} $, which will simplify our analysis.
	
\begin{lem}\label{equiv}
	Let $ b>0 $. A positive function $ g \in  C^2((0,+\infty)) \cap C([0,+\infty)) $ is a solution of the weighted ODE problem
		\begin{equation}\label{ode-app-1}
	\begin{cases}
	(g^m)^{''}(r)+\frac{N-1}{r} \, (g^m)^{'}(r)+ r^{-\gamma}\left[\frac{\lambda_\alpha}{\alpha} \, r \, g'(r)+\lambda_\alpha \, g(r)\right]=0 & \text{for } r>0 \, , \\
	g'(r) = o\!\left(r^{-\frac \gamma 2}\right) & \text{as } r \to 0^+ , \\
	\lim_{r \to +\infty}r^{\alpha} \, g(r)=b \, ,
	\end{cases}
	\end{equation}
	if and only if the function
	\begin{equation}\label{def-g-tilde}
	\tilde{g}(r) = g\!\left( r^{\frac{2}{2-\gamma}} \right) \qquad \forall r \ge 0
	\end{equation}
	is a positive $ C^2((0,+\infty)) \cap C^1([0,+\infty))  $ solution of the unweighted ODE problem
			\begin{equation}\label{ode-app-2}
	\begin{cases}
	(\tilde g^m)^{''}(r)+\frac{\tilde N-1}{r} \, (\tilde g^m)^{'}(r)+ \frac{\tilde \lambda}{\tilde \alpha} \, r \, \tilde g'(r)+\tilde \lambda \, \tilde g(r)=0 & \text{for } r>0 \, , \\
	\tilde g'(0) =0 \,  , \\
	\lim_{r \to +\infty}r^{\tilde \alpha} \, \tilde g(r)=b \, ,
	\end{cases}
	\end{equation}
	where
\begin{equation}\label{ode-parameters}
				\Tilde{N}=\frac{2(N-\gamma)}{2-\gamma}>1 \, , \qquad \Tilde{\alpha}=\frac{2\alpha}{2-\gamma} \, , \qquad \tilde{\lambda} = \frac{2 \tilde \alpha}{(2-\gamma)\alpha} \, \lambda_\alpha \, .
\end{equation}
\end{lem}
\begin{proof}
	From direct calculations, we have:
			\begin{equation*}
			\tilde g'(r)=\left(\frac{2}{2-\gamma}\right)r^{\frac{\gamma}{2-\gamma}}\, g'\!\left(r^{\frac{2}{2-\gamma}}\right), \qquad \left(\tilde g^m\right)\!'(r)=\left(\frac{2}{2-\gamma}\right)r^{\frac{\gamma}{2-\gamma}} \left( g^m\right)\!'\!\left(r^{\frac{2}{2-\gamma}}\right)
		\end{equation*}
		and
		\begin{equation*}
		\left(\tilde g^m\right)\!''(r) = \left(\frac{2}{2-\gamma}\right)^2 r^{\frac{2\gamma}{2-\gamma}} \left({g}^m\right)\!''\!\left(r^{\frac{2}{2-\gamma}}\right) +  \frac{2 \gamma}{(2-\gamma)^2}  \, r^{\frac{2\gamma-2}{2-\gamma}}\left( g^m\right)\!'\!\left(r^{\frac{2}{2-\gamma}}\right) .
		\end{equation*}
Combining these expressions and applying \eqref{ode-parameters}, after some algebraic manipulations one can check that if $ g $ solves the differential equation in \eqref{ode-app-1}, then $ \tilde{g} $ solves the one in \eqref{ode-app-2}. As for the boundary conditions, we have
$$
\lim_{r \to +\infty} r^{\tilde \alpha} \tilde g(r) = \lim_{r \to +\infty} r^{\tilde \alpha} g\!\left( r^{\frac{2}{2-\gamma}} \right) = \lim_{r \to +\infty} r^{\tilde \alpha \, \frac{2-\gamma}{2}}  g\!\left( r \right) =  \lim_{r \to +\infty} r^{\alpha}  g\!\left( r \right) = b
$$
and
$$
\lim_{r \to 0+} \tilde{g}'(r) = \left( \frac{2}{2-\gamma} \right) \lim_{r \to 0+} r^{\frac{\gamma}{2-\gamma}}\, g'\!\left(r^{\frac{2}{2-\gamma}}\right) = \left( \frac{2}{2-\gamma} \right) \lim_{r \to 0+} r^{\frac{\gamma}{2}}\, g'\!\left(r\right) = 0 \, .
$$
Hence, we have shown that a solution of \eqref{ode-app-1} becomes a solution of \eqref{ode-app-2} via the change of variables \eqref{def-g-tilde}. The opposite implication is completely analogous.		
\end{proof}

Before addressing the condition at infinity in \eqref{ode-app-1} or \eqref{ode-app-2}, we notice that looking for a positive solution that has such a prescribed behavior as $ r \to +\infty $ actually amounts to studying a \emph{local} Cauchy problem for a given initial datum.

	\begin{lem}\label{reduct-local}
There exists a positive function $ \tilde{g} \in  C^2((0,+\infty)) \cap C^1([0,+\infty))  $ that solves \eqref{ode-app-2} if and only if there exists a positive function $ G \in  C^2((0,+\infty)) \cap C^1([0,+\infty)) $ that solves
			\begin{equation}\label{ode-app-local}
	\begin{cases}
	(G^m)^{''}(r)+\frac{\tilde N-1}{r} \, ( G^m)^{'}(r)+ \frac{\tilde \lambda}{\tilde \alpha} \, r \, G'(r)+\tilde \lambda \, G(r)=0 & \text{for } r>0 \, , \\
	 G'(0) =0 \,  , \\
	G(0) = 1 \, , \\
	\exists \lim_{r \to +\infty}r^{\tilde \alpha} \, G(r) =: L \in (0,+\infty)  \, ,
	\end{cases}
	\end{equation}
	and in such a case the identity
		\begin{equation}\label{ode-app-id}
	\tilde{g}(r) = \sigma \, G\!\left( \sigma^{-\frac{m-1}{2}} r \right)
	\end{equation}
	holds for a suitable constant $ \sigma>0 $.
	\end{lem}
	\begin{proof}
	If $ h $ is any positive solution of the differential equation in \eqref{ode-app-local} with $ h'(0)=0 $, a straightforward computation shows that, for each $ \sigma>0 $, the rescaled function
	$$
	r \mapsto \sigma \, h\!\left( \sigma^{-\frac{m-1}{2}} r \right)
	$$
	is also a positive solution, with initial condition $ \sigma h(0) $. In particular, if there exists a positive solution $ G $ of \eqref{ode-app-local}, then the function
	$$
	\hat{g}(r) = \sigma \, G\!\left( \sigma^{-\frac{m-1}{2}} r \right)
	$$
	solves
				\begin{equation*}\label{ode-app-local-proof-1}
	\begin{cases}
	(\hat g^m)^{''}(r)+\frac{\tilde N-1}{r} \, ( \hat g^m)^{'}(r)+ \frac{\tilde \lambda}{\tilde \alpha} \, r \, \hat g'(r)+\tilde \lambda \, \hat g(r)=0 & \text{for } r>0 \, , \\
	 \hat g'(0) =0 \,  , \\
	\hat g(0) = \sigma \, , \\
	\lim_{r \to +\infty}r^{\tilde \alpha} \, \hat g(r) = L \, \sigma^{1+\frac{m-1}{2}\tilde{\alpha}} \, .
	\end{cases}
	\end{equation*}
	Therefore, with the choice
	$$
	\sigma = \left( \frac{b}{L} \right)^{\frac{2}{2+(m-1)\tilde{\alpha}}}
	$$
	it becomes a solution of \eqref{ode-app-2}. Similarly, if $ \tilde{g} $ is a positive solution of \eqref{ode-app-2}, in order to obtain a solution of \eqref{ode-app-local} it is enough to rescale it as above with
	$$
	\sigma = \frac{1}{\tilde{g}(0)} \, .
	$$
	Finally, identity \eqref{ode-app-id} must hold in view of the uniqueness for the Cauchy problem \eqref{ode-app-local} (regardless of the condition at infinity),
see also Subsections \ref{EX} and \ref{EG}.
	\end{proof}
	
In view of Lemmas \ref{equiv} and \ref{reduct-local}, in order to prove that \eqref{ode1} is well posed, what follows aims at establishing that the same is true for \eqref{ode-app-local}.
	
	\begin{thm}\label{main-prop-ode}
	There exists a unique positive solution $ G \in  C^2((0,+\infty)) \cap C^1([0,+\infty)) $ of the ODE problem \eqref{ode-app-local}.
	\end{thm}

\begin{rem}\rm\label{purepowerode}
One may wonder whether, dropping the conditions at $r=0$, an exact power $ g(r) = b \, r^{p} $ can be a solution of the differential equation in \eqref{ode-app-1} for some $b>0$ and $ p \in \mathbb{R} $. A simple computation shows that this is the case if and only if
$$
p=-\alpha \qquad \text{and} \qquad \alpha = \frac{N-2}{m} \, ,
$$
for arbitrary $b>0$. However, it is well known that $ -\Delta g^m(|x|) $ generates a Dirac delta at $ x=0 $. Hence, although
$$ (x,t) \mapsto  t^{-\lambda_\alpha} g\!\left(t^{-\frac{\lambda_\alpha}{\alpha}} |x|\right) $$
is still a (stationary) self-similar solution of \eqref{singularWPME}, the differential equation is solved only in $ \mathbb{R}^N \setminus \{ 0 \} $.
\end{rem}
	
	\subsection{Existence of local solutions}\label{EX}
	The content of this subsection is rather standard, since \eqref{ode-app-local} (ignoring the asymptotic condition) is in fact equivalent to a system of sublinear differential equations. For completeness, we provide the proof of local existence, based on a classical fixed-point argument.
	\begin{lem}\label{local existence}
		There exist $\varepsilon>0$ and a unique positive solution $ G \in C^1([0,\varepsilon])\cap C^2((0,\varepsilon))$ of the Cauchy problem
				\begin{equation}\label{ode1 system}
	\begin{cases}
	(G^m)^{''}(r)+\frac{\tilde N-1}{r} \, ( G^m)^{'}(r)+ \frac{\tilde \lambda}{\tilde \alpha} \, r \, G'(r)+\tilde \lambda \, G(r)=0 & \text{for } r \in (0,\varepsilon) \, , \\
	 G'(0) =0 \,  , \\
	G(0) = 1 \, .
	\end{cases}
\end{equation}
	\end{lem}
	\begin{proof}
		Under the definitions $w:=G^m$ and $v:=w'$, we see that \eqref{ode1 system} is equivalent to the first-order system
		\begin{equation}\label{ode2 system}
			\begin{cases}
				w' = v  & \text{in } (0,\varepsilon) \, , \\
				v'=-\frac{\Tilde{N}-1}{r} \, v-\frac{\Tilde{\lambda}}{\Tilde{\alpha}m} \,r\,w^{-\frac{m-1}{m}}v-\Tilde{\lambda} \, w^{\frac{1}{m}} & \text{in } (0,\varepsilon) \, , \\
				w(0)= 1 \, ,   \\
				v(0)=0 \, . \\
			\end{cases}
		\end{equation}
		We integrate these equations to derive the integral identities
		\begin{equation}\label{int w}
			w(r)=1+\int_0^r v \,ds
		\end{equation}
		and
		\begin{equation}\label{int v}
			v(r)=-\frac{\Tilde{\lambda}}{\Tilde{\alpha}m} \, \frac{1}{r^{\Tilde{N}-1}} \int_0^r w^{-\frac{m-1}{m}}v \, s^{\Tilde{N}} \, ds-\frac{\Tilde{\lambda}}{r^{\Tilde{N}-1}}\int_0^r w^{\frac{1}{m}} \, s^{\Tilde{N}-1} \, ds \, .
		\end{equation}
		We now prepare to apply Banach's fixed-point theorem to obtain existence and uniqueness of solutions to \eqref{ode2 system}. To this end, let us define the complete metric space
		\begin{equation}\label{defspace}
			X := \left\{ (u,v) \in C\!\left([0,\varepsilon];\R^2\right) : \  \kappa \leq w\leq K \, , \  |v| \leq H  \right\}
		\end{equation}
		under the metric induced by the norm
		$$
		\norm{(w,v)}_X=\norm{w}_{L^\infty((0,\varepsilon))}+\norm{v}_{L^\infty((0,\varepsilon))} ,
		$$
		where $\varepsilon,\kappa<K,H$ are positive constants that will be chosen later. Moreover, we let $ \mathcal{T} $ denote the map $ X \ni (w,v)\mapsto(\Tilde{w},\Tilde{v})$ modeled on \eqref{int w} and \eqref{int v}, that is,
		$$
		\begin{gathered}
			\Tilde{w}(r)= 1 + \int_0^r v \, ds \, , \\
			\Tilde{v}(r)=-\frac{\Tilde{\lambda}}{\Tilde{\alpha}m} \, \frac{1}{r^{\Tilde{N}-1}} \int_0^r w^{-\frac{m-1}{m}}v \, s^{\Tilde{N}} \, ds-\frac{\Tilde{\lambda}}{r^{\Tilde{N}-1}}\int_0^r w^{\frac{1}{m}} \, s^{\Tilde{N}-1} \, ds \, ,
		\end{gathered}
		$$
		for $0 < r \leq \varepsilon$, with obvious extensions at $ r=0 $. It is plain that $(\Tilde{w},\Tilde{v})\in C\!\left([0,\varepsilon];\R^2\right)$, and if we further impose the following conditions:
		\begin{equation}\label{extra cond w}
			 1 + H\varepsilon \leq K \, ,\qquad 1- H \varepsilon \geq \kappa \, ,
		\end{equation}
		and
		\begin{equation}\label{extra cond v}
			\frac{\Tilde{\lambda}}{\Tilde{\alpha}m\left(\Tilde{N}+1\right)}\, \kappa^{-\frac{m-1}{m}} H \varepsilon^2  + \frac{\Tilde{\lambda}}{\tilde{N}}\,K^{\frac{1}{m}}\,\varepsilon\leq H \, ,
		\end{equation}
one can check that $(\Tilde{w},\Tilde{v})\in X$. Hence, $ \mathcal{T} $ maps $X$ into itself. Next, we require it to be a contraction. Let us take two elements $(w_1,v_1),(w_2,v_2)\in X$ and estimate the difference of their images in the $X$ norm. First, for all $ r \in (0,\varepsilon] $ we have:
		\begin{equation}\label{w diff est}
			\begin{aligned}
				\left|\Tilde{w}_1(r)-\Tilde{w}_2(r)\right| =\left|1+\int_0^r v_1\,ds-1-\int_0^r v_2 \,ds\right| & \leq\int_0^r\left|v_1 -v_2 \right|ds \leq\varepsilon\norm{v_1-v_2}_{L^\infty((0,\varepsilon))} ,
			\end{aligned}
		\end{equation}
		so that $\norm{\Tilde{w}_1-\Tilde{w}_2}_{L^\infty((0,\varepsilon))}\leq\varepsilon\norm{v_1-v_2}_{L^\infty((0,\varepsilon))}$.
		On the other hand, 	
		\begin{equation}\label{v diff est}
		 \!	\left|\Tilde{v}_1(r)-\Tilde{v}_2(r)\right| \leq\underbrace{\frac{\Tilde{\lambda}}{\Tilde{\alpha}m} \, \frac{1}{r^{\Tilde{N}-1}}\int_0^r \left|w_1^{-\frac{m-1}{m}}v_1-w_2^{-\frac{m-1}{m}}v_2\right| s^{\Tilde{N}} ds}_{I} + \underbrace{\frac{\Tilde{\lambda}}{r^{\Tilde{N}-1}}\int_0^r\left|w_1^{\frac{1}{m}}-w_2^{\frac{1}{m}}\right|s^{\Tilde{N}-1} \, ds}_{II}   .
		\end{equation}
		Let us start from the simpler term, namely $II$. By Lagrange's theorem and \eqref{defspace},
		\begin{equation*}
			\left|w_1^{\frac{1}{m}}-w_2^{\frac{1}{m}}\right|\leq\frac{1}{m} \, \kappa^{-\frac{m-1}{m}} \left|w_1-w_2 \right|,
		\end{equation*}
		so that
		\begin{equation}\label{ii estimate}
			II\leq\frac{\Tilde{\lambda}}{m \Tilde{N}} \, \kappa^{-\frac{m-1}{m}}\varepsilon \norm{w_1-w_2}_{L^\infty((0,\varepsilon))} .
		\end{equation}
		Now, we estimate $I$. First, we apply the triangle inequality and again Lagrange's theorem in the following way:
		\begin{align*}
			\left|w_1^{-\frac{m-1}{m}}v_1-w_2^{-\frac{m-1}{m}}v_2\right|&=\left|w_1^{-\frac{m-1}{m}}v_1-w_1^{-\frac{m-1}{m}}v_2+w_1^{-\frac{m-1}{m}}v_2-w_2^{-\frac{m-1}{m}}v_2\right|\\
			&\leq \kappa^{-\frac{m-1}{m}}|v_1-v_2|+\tfrac{m-1}{m} \kappa^{-\frac{m-1}{m}-1} H \left|w_1-w_2\right| .
		\end{align*}
		Then, upon integration,
		\begin{equation}\label{i estimate}
		\begin{aligned}
			I\leq  \, \frac{\Tilde{\lambda}}{\Tilde{\alpha}m\left(\Tilde{N}+1\right)}\, \kappa^{-\frac{m-1}{m}}\varepsilon^2 \left\|v_1-v_2\right\|_{L^\infty((0,\varepsilon))}  +\frac{\Tilde{\lambda}(m-1)}{\Tilde{\alpha}m^2\left(\Tilde{N}+1\right)} \, \kappa^{-\frac{m-1}{m}-1} H \varepsilon^2\left\|w_1-w_2\right\|_{L^\infty((0,\varepsilon))}.
			\end{aligned}
		\end{equation}
		Combining \eqref{w diff est}--\eqref{i estimate}, we finally obtain
		\begin{equation}\label{contraction constant}
			\begin{aligned}
\norm{(\Tilde{w}_1 ,\Tilde{v}_1)-(\Tilde{w}_2,\Tilde{v}_2)}_X \leq & \left[ \left( 1+\frac{\Tilde{\lambda}}{m\Tilde{N}} \, \kappa^{-\frac{m-1}{m}} \right)\varepsilon+\frac{\Tilde{\lambda}}{\Tilde{\alpha}m\left(\Tilde{N}+1\right)} \, \kappa^{-\frac{m-1}{m}} \left( 1 + \tfrac{m-1}{m} \, \kappa^{-1} \, H \right) \varepsilon^2 \right] \\
& \, \times \norm{(w_1,v_1)-(w_2,v_2)}_X .
			\end{aligned}
		\end{equation}
	Therefore, for $ \mathcal{T} $ to be a contraction, we need the multiplying constant on the right-hand side of \eqref{contraction constant} to be strictly less than $1$, which, recalling \eqref{extra cond w}--\eqref{extra cond v}, amounts to the following nonlinear system of inequalities:
		\begin{equation}\label{ineq system}
			\begin{cases}
				1 + H\varepsilon \leq K  \, ,   \\
				1- H \varepsilon \geq \kappa \, ,  \\
				\frac{\Tilde{\lambda}}{\Tilde{\alpha}m\left(\Tilde{N}+1\right)}\, \kappa^{-\frac{m-1}{m}} H \varepsilon^2  + \frac{\Tilde{\lambda}}{\tilde{N}}\,K^{\frac{1}{m}}\,\varepsilon\leq H \, , \\
				\left( 1+\frac{\Tilde{\lambda}}{m\Tilde{N}} \, \kappa^{-\frac{m-1}{m}} \right)\varepsilon+\frac{\Tilde{\lambda}}{\Tilde{\alpha}m\left(\Tilde{N}+1\right)} \, \kappa^{-\frac{m-1}{m}} \left( 1 + \frac{m-1}{m} \, \kappa^{-1} \, H \right) \varepsilon^2 < 1 \, .
			\end{cases}
		\end{equation}
		It is clear that, for every fixed $ K>1 $, $ \kappa \in (0,1) $ and $ H>0 $, we can select an $\varepsilon>0$ small enough so that \eqref{ineq system} holds, even with strict inequalities. We may now apply Banach's fixed-point theorem, ensuring that $ \mathcal{T} $ has a unique fixed point. Because being a fixed point of $ \mathcal{T} $ is equivalent to being a solution of \eqref{ode1 system}, the thesis follows.
	\end{proof}
	
\begin{oss}\rm \label{rem-ext}
By means of a completely analogous proof, it is apparent that for every $ r_0>0 $, $ W_0 > 0 $ and $ V_0 \in \mathbb{R} $, there exists $ \varepsilon>0 $ (small enough) such that the Cauchy problem
$$
	\begin{cases}
	(G^m)^{''}(r)+\frac{\tilde N-1}{r} \, ( G^m)^{'}(r)+ \frac{\tilde \lambda}{\tilde \alpha} \, r \, G'(r)+\tilde \lambda \, G(r)=0 & \text{for } r \in (r_0-\varepsilon,r_0+\varepsilon) \, , \\
	 G'(r_0) =  V_0 \,  , \\
	G(r_0) = W_0 \, ,
	\end{cases}
$$
also admits a unique positive solution in $ C^2\!\left((r_0-\varepsilon,r_0+\varepsilon) \right) $.
\end{oss}
	
	\subsection{Extension to a global solution}\label{EG}
	According to the above procedure for existence of local solutions to \eqref{ode1 system}, by means of a standard ODE argument (recall Remark \ref{rem-ext}), there is a well-defined maximal radius $R\in (0,+\infty]$ up to which the positive solution $G$ can be uniquely extended, namely it is still a positive $ C^2((0,R)) \cap C^1([0,R)) $ function solving \eqref{ode1 system} with $\varepsilon = R$. We aim at showing that in fact $ R = +\infty $ and $ G(r) $ decreases to $ 0 $ as $ r \to +\infty $. To this purpose, we first derive an equality that is very important for the following results. Indeed, if we rewrite the differential equation in \eqref{ode1 system} as
	$$
\frac{1}{r^{\tilde{N}-1}} \left( r^{\tilde{N}-1} (G^m)^{'} \right)' + \frac{\tilde \lambda}{\tilde \alpha} \, \frac{1}{r^{\tilde{N}-1}} \left( r^{\tilde{N}} \, G \right)' =  \left(\frac{\Tilde{N}}{\Tilde{\alpha}}-1\right)\Tilde{\lambda}  \, G  \, ,
	$$
	multiply by $r^{\Tilde{N}-1}$ and integrate from $0$ to an arbitrary $ r  \in (0,R)$, we obtain:
	\begin{equation}\label{10p}
		r^{\Tilde{N}-1} \left(G^m\right)\!'(r)+\frac{\Tilde{\lambda}}{\Tilde{\alpha}} \, r^{\Tilde{N}} \, G(r)=\left(\frac{\Tilde{N}}{\Tilde{\alpha}}-1\right)\Tilde{\lambda} \int_0^r s^{\Tilde{N}-1} \, G(s) \, ds \, .
	\end{equation}
From \eqref{10p} and again standard ODE extension methods, it is plain that, if $ R<+\infty $, then the only possibilities for the behavior of $G$ are that either $ G(r)\to+\infty$ or $ G(r)\to0$ as $r\to R^-$ (\emph{a priori} at least along subsequences). 	

	\begin{lem}\label{lem-max exist}
The positive solution $G$ of 	\eqref{ode1 system} provided by Lemma \ref{local existence} can be maximally extended to the whole $ [0,+\infty) $. Moreover, the following further properties of $ G $ hold:
		\begin{enumerate}[(a)]
			\item $G'(r)<0$ for all $r>0$; \label{one}
			\item $ G(r) \to 0$ as $r\to+\infty$. \label{two}
		\end{enumerate}
	\end{lem}
	\begin{proof}
First of all, we notice that the differential equation is such that $G$ possesses no local minimum in $  (0,R)$ (if $ G'(\tilde r)=0 $ then $ \left( G^m \right)\!''(\tilde r) < 0 $). Moreover, from \eqref{10p} and the continuity of $G$ near the origin, it is straightforward to verify that
$$
\left(G^m\right)\!'(r) \sim  - \frac{\tilde{\lambda}}{\tilde{N}} \, r \qquad \text{as } r \to 0 \, ;
$$
in particular, we infer that $ G'<0 $ in a neighborhood of the origin. Therefore, we assert that $ G'(r)<0 $ for all $ r \in (0,R)$.

In order to prove that $ R = +\infty $, we can argue by contradiction  assuming that $ R<+\infty $ and thus $ \lim_{r \to R^-} G(r) = 0 $. By  taking the limit of \eqref{10p} as $ r \to R^- $ we obtain
$$
	\lim_{r \to R^-}	r^{\Tilde{N}-1} \left(G^m\right)\!'(r) = \left(\frac{\Tilde{N}}{\Tilde{\alpha}}-1\right)\Tilde{\lambda} \int_0^R s^{\Tilde{N}-1} \, G(s) \, ds > 0 \, ,
$$
		a contradiction since $G$ is always monotone decreasing. {Hence $ R=+\infty $ and \emph{(a)} holds.}
		
Finally, we show that $G(r) \to 0$ as $r\to+\infty$. Because $G$ is decreasing, there exists a finite limit $ \ell := \lim_{r \to +\infty} G(r)$, so assume by contradiction that $ \ell>0 $. Let us fix
\begin{equation}\label{eps-tilde}
0 < \varepsilon < \frac{\tilde{\alpha}\ell}{\tilde{N}-\tilde{\alpha}}
\end{equation}
and choose $ r_\varepsilon>0 $ large enough that $G(r) \leq \ell +\varepsilon$ for all $r>r_\varepsilon$. By applying \eqref{10p} and noticing that $ \ell < G \le 1 $, we infer:
$$
\begin{aligned}
r^{\Tilde{N}-1} \left(G^m\right)\!'(r)+\frac{\Tilde{\lambda}}{\Tilde{\alpha}} \, r^{\Tilde{N}} \, \ell  \le & \,  r^{\Tilde{N}-1} \left(G^m\right)\!'(r)+\frac{\Tilde{\lambda}}{\Tilde{\alpha}} \, r^{\Tilde{N}} \, G(r) \\ =
& \left(\frac{\Tilde{N}}{\Tilde{\alpha}}-1\right)\Tilde{\lambda} \int_0^r s^{\Tilde{N}-1} \, G(s) \, ds \\
\le & \left(\frac{\Tilde{N}}{\Tilde{\alpha}}-1\right)\Tilde{\lambda} \int_0^{r_\varepsilon} s^{\Tilde{N}-1}  \, ds + \left(\frac{\Tilde{N}}{\Tilde{\alpha}}-1\right)\Tilde{\lambda} \left( \ell + \varepsilon \right) \int_{r_\varepsilon}^r s^{\Tilde{N}-1}  \, ds \\
=: & \, C_{\varepsilon} + \frac{\tilde{N}-\tilde{\alpha}}{\tilde{N}\tilde{\alpha}} \, \Tilde{\lambda} \left( \ell + \varepsilon \right) r^{\tilde{N}} \, ,
\end{aligned}
$$
for all $ r>r_\varepsilon $ and an explicit positive constant $ C_\varepsilon $. Now, we rearrange terms and divide by $r^{\Tilde{N}-1}$, to obtain
		\begin{equation}\label{a20}
			 \left(G^m\right)\!'(r) \leq \frac{C_\varepsilon}{r^{\tilde{N}-1}} - \frac{\tilde{\lambda}\left[ \tilde{\alpha} \ell - (\tilde{N}-\tilde{\alpha})\varepsilon \right]}{\tilde{N}\tilde{\alpha}} \, r \qquad \forall r>r_\varepsilon \, .
		\end{equation}
In view of \eqref{eps-tilde}, it is clear from \eqref{a20} that $\left(G^m\right)\!'(r) \to-\infty$ as $ r \to +\infty $, which is a contradiction to $ G \ge 0 $.
	\end{proof}
	
	\subsection{Asymptotics and proofs of Theorem \ref{main-prop-ode} and Theorem \ref{selfsim-sol}}
In order to conclude the analysis of \eqref{ode-app-local}, it is still left to prove the precise asymptotic behavior of the solution at infinity. To begin our asymptotic analysis, it is convenient to define the function
	\begin{equation*}\label{deftheta}
		\Theta(r):=r^{\Tilde{\alpha}} \, G(r) \qquad \forall r \ge 0 \, ,
	\end{equation*}
	which immediately satisfies $\Theta(0)= 0$ and $\Theta(r)>0$ for all $r>0$.
	\begin{pro}\label{asymptotics prop}
		The globally positive solution $G$ provided by Lemmas \ref{local existence} and \ref{lem-max exist} satisfies
		\begin{equation*}
			\lim_{r\to\infty}r^{\Tilde{\alpha}} \, G(r) = \lim_{r\to\infty}\Theta(r) = L \in (0,+\infty) \, .
		\end{equation*}
		Moreover, the following dichotomy holds:
		\begin{itemize}
		
\item if $ \tilde{\alpha}m - \tilde{N} + 2 \le 0 $ then $ \Theta $ is globally increasing;

\item if $ \tilde{\alpha}m - \tilde{N} + 2 > 0 $ then $ \Theta $ admits a unique global maximum at $ r=r^\ast $, and it is increasing in $ (0,r^\ast) $ and decreasing in $ (r^\ast,+\infty) $.
		\end{itemize}
	\end{pro}	
	
	 From the differential equation in \eqref{ode-app-local}, it is an elementary calculation to verify that
	\begin{equation}\label{eq theta 1}
		\frac{\Tilde{\lambda}}{\Tilde{\alpha}} \, r^{\Tilde{N}-\Tilde{\alpha}}\, \Theta'=-\left(r^{\Tilde{N}-1} \left( G^m \right)\!'\right)'
	\end{equation}
	and
	\begin{equation}\label{eq theta 2}
		\left(\Theta^m\right)\!''= \left(2\Tilde{\alpha}m-\Tilde{N}+1\right)r^{-1}\left(\Theta^m\right)\!'-\Tilde{\alpha}m\left(\Tilde{\alpha}m-\Tilde{N}+2\right)r^{-2} \, \Theta^m - \frac{\Tilde{\lambda}}{\Tilde{\alpha}}\, r^{\Tilde{\alpha}(m-1)+1} \, \Theta' \, .
	\end{equation}
	
	\begin{proof}[Proof of Proposition \ref{asymptotics prop}]
		Let us use \eqref{eq theta 2} to study the possibility of local maxima and minima of $ \Theta $ in order to gain insight into the global behavior of $G$. First of all, we observe that, if $\Theta'(r^*)=0$ for some $r^*>0$, then
		\begin{equation}\label{theta-star}
			\left(\Theta^m\right)\!''(r^*)=-\Tilde{\alpha}m\left(\Tilde{\alpha}m-\Tilde{N}+2\right)r^{-2} \, \Theta(r^\ast)^m .
		\end{equation}
		Now, we split our analysis into three cases, depending on the sign of the coefficient on the right-hand side of \eqref{theta-star}.

\smallskip 		
		
\noindent \emph{Case 1:}
		If $\Tilde{\alpha}m-\Tilde{N} + 2<0$, then any critical point $ r^\ast >0 $ of $ \Theta $ is necessarily a point of strict local minimum, but since $\Theta(0)=0$ and $\Theta(r)>0$ this would imply that there exists at least another critical point in $(0,r^\ast) $ which is a local maximum, and this is impossible. Therefore $\Theta'(r) > 0$ for all $r>0$, thus it is a globally increasing function. Hence, in order to find a positive and finite limit of $ \Theta$ it is enough to show that it is bounded. Indeed, let us multiply \eqref{eq theta 1} by $r^{\Tilde{\alpha}-\Tilde{N}}$ and integrate between arbitrary $0<r_1<r_2$ to obtain
		\begin{equation}\label{a23}
			\begin{aligned}
				\frac{\Tilde{\lambda}}{\Tilde{\alpha}} \int_{r_1}^{r_2}\Theta' \, dr&=-\int_{r_1}^{r_2}r^{\Tilde{\alpha}-\Tilde{N}}\left(r^{\Tilde{N}-1}\left(G^m\right)'\right)' dr\\
				&=- r_2^{\Tilde{\alpha}-1} \left(G^m\right)\!'(r_2)+r_1^{\Tilde{\alpha}-1}\left(G^m\right)\!'(r_1) - \left(\Tilde{N}-\Tilde{\alpha}\right)\int_{r_1}^{r_2} r^{\Tilde{\alpha}-2}\left(G^m\right)' dr \, .
			\end{aligned}
		\end{equation}
		Using the identity
		$$
		r^{\Tilde{\alpha}-2} \left(G^m\right)\!'(r)=\frac{m}{m-1}\, r^{-2} \, \Theta(r)\left(G^{m-1}\right)\!'(r) \, ,
		$$
		we reach
		\begin{equation}\label{a24}
			\begin{aligned}
				 \frac{m-1}{m}\,\frac{\Tilde{\lambda}}{\Tilde{\alpha}}\int_{r_1}^{r_2} \Theta'\, dr
				& \underset{{\mathrm{(a)}}}{\le} -r_2^{-1}\, \Theta(r_2) \left(G^{m-1}\right)\!'(r_2) - \left(\Tilde{N}-\Tilde{\alpha}\right) \int_{r_1}^{r_2} r^{-2} \, \Theta(r) \left(G^{m-1}\right)\!'(r) \, dr \\
				&  \underset{{\mathrm{(b)}}}{\le} -r_2^{-1}\, \Theta(r_2) \left(G^{m-1}\right)\!'(r_2) - \left(\Tilde{N}-\Tilde{\alpha}\right) r_1^{-2} \, \Theta(r_2) \int_{r_1}^{r_2} \left(G^{m-1}\right)\!' \, dr \\
				& = -r_2^{-1}\, \Theta(r_2) \left(G^{m-1}\right)\!'(r_2) - \left(\Tilde{N}-\Tilde{\alpha}\right) r_1^{-2} \, \Theta(r_2) \left[ G(r_2)^{m-1} - G(r_1)^{m-1} \right] \\
				&\underset{{\mathrm{(c)}}}{\le} -r_2^{-1}\, \Theta(r_2) \left(G^{m-1}\right)\!'(r_2) + \left(\Tilde{N}-\Tilde{\alpha}\right) r_1^{-2} \, \Theta(r_2) \, ,
			\end{aligned}
		\end{equation}
		where in $\mathrm{(a)}$ we neglected the middle term of \eqref{a23} using $G'<0$, in $\mathrm{(b)}$ we used that $\Theta$ is increasing, and in $ \mathrm{(c)}$ we exploited the fact that $G \le 1$. Integrating and rearranging terms, we may rewrite \eqref{a24} as
		\begin{equation}\label{trick-sequence}
\left[ \frac{m-1}{m}\,\frac{\Tilde{\lambda}}{\Tilde{\alpha}} + r_2^{-1} \left(G^{m-1}\right)\!'(r_2) - \left(\Tilde{N}-\Tilde{\alpha}\right) r_1^{-2}  \right] \Theta(r_2)	\le \frac{m-1}{m} \, \frac{\Tilde{\lambda}}{\Tilde{\alpha}} \, \Theta(r_1) \, .
		\end{equation}
Since we know that $ G(r)^{m-1} \to 0 $ as $ r \to +\infty $, there exists a sequence $ r_k \to + \infty $ such that $ \left(G^{m-1}\right)\!'(r_k) \to 0 $. Therefore, if we fix $r_1$ large enough and apply \eqref{trick-sequence} with $r_2$ replaced by $ r_k $, we infer that
$$
\limsup_{k \to \infty} \Theta(r_k) < +\infty \, ,
$$
which actually implies that $ \Theta $ is bounded, due to its monotonicity.

\smallskip 		
		
\noindent \emph{Case 2:} If $\Tilde{\alpha}m-\Tilde{N} + 2>0$, then by \eqref{theta-star} we deduce that $\Theta$ may have at most one local maximum and it possesses no local minima. Let us show that a local maximum indeed exists, which is equivalent to ruling out the possibility that $ \Theta $ is globally increasing. If, by contradiction, $ \Theta'>0 $ everywhere, then we can show that $ \Theta $ has a positive and finite limit $L$ by the same reasoning as in Case 1; indeed, \eqref{a23}--\eqref{trick-sequence} do not depend on the sign of $ \Tilde{\alpha}m-\Tilde{N} + 2 $. Next, if we set
$$
r_0 := \left( \tfrac{2\tilde{\alpha}m-\tilde{N}+1}{\tilde{\lambda}} \,\tilde{\alpha} m L^{m-1} \right)^{\frac{1}{\tilde{\alpha}(m-1)+2}} \, ,
$$
it is readily seen that
$$
\left(2\Tilde{\alpha}m-\Tilde{N}+1\right)r^{-1}\left(\Theta^m\right)\!'(r) - \frac{\Tilde{\lambda}}{\Tilde{\alpha}}\, r^{\Tilde{\alpha}(m-1)+1} \, \Theta'(r) \le 0 \qquad \forall r \ge r_0 \, .
$$
Going back to \eqref{eq theta 2}, this implies
\begin{equation}\label{derivata-sec}
\left(\Theta^m\right)\!''(r) \le -\Tilde{\alpha}m\left(\Tilde{\alpha}m-\Tilde{N}+2\right)r^{-2} \, \Theta(r)^m \le -\frac{C}{r^2}  \qquad \forall r \ge r_0 \,
\end{equation}
for a suitable constant $C>0$. Such an inequality entails the concavity of $ \Theta^m $ in $ (r_0,+\infty) $, whence $ (\Theta^m)'(r) \to 0 $ as $ r \to +\infty $. So we can integrate \eqref{derivata-sec} in $ (r,+\infty) $ to obtain
$$
-\left(\Theta^m\right)\!'(r) \le - \frac{C}{r} \qquad \forall r \ge r_0 \, ;
$$
a further integration in $ (r_0,r) $ then yields
$$
C \log\!\left( \frac{r}{r_0} \right) + \Theta(r_0)^m \le \Theta(r)^m \qquad \forall r \ge r_0  \, ,
$$
which is clearly in contradiction to the boundedness of $ \Theta $. Therefore, $\Theta$ has a unique local maximum at some $r^*>0$. Hence, for $r > r^\ast$ it is monotone decreasing and obviously bounded below by $0$, so $\Theta(r) \to L \in [0,+\infty)$ as $r\to +\infty$. To conclude, we only need to check that $L>0$. We start by integrating \eqref{eq theta 1} from an arbitrary $r_1>r^*$ to $r_2>r_1$ to obtain
		\begin{equation}\label{a27}
			\begin{aligned}
				-r_2^{\Tilde{N}-1} \left(G^m\right)\!'(r_2)+r_1^{\Tilde{N}-1}\left(G^m\right)\!'(r_1) = \frac{\Tilde{\lambda}}{\Tilde{\alpha}}\int_{r_1}^{r_2}r^{\Tilde{N}-\Tilde{\alpha}} \, \Theta' \, dr & \leq \frac{\Tilde{\lambda}}{\Tilde{\alpha}} \left[r_1^{\Tilde{N}-\Tilde{\alpha}} \, \Theta(r_2)-r_1^{\Tilde{N}} \, G(r_1) \right] ,
			\end{aligned}
		\end{equation}
		where we have used the fact that $\Theta' \le 0$ within such range. Next, since $G'<0$, the sign of the first term on the left-hand side of \eqref{a27} allows us to ignore it. Let us suppose, in order to gain a contradiction, that $ L=0$. Then passing to the limit in \eqref{a27} as $r_2\to+\infty$, it follows that
		\begin{equation*}\label{a28}
			r_1^{\Tilde{N}-1} \left(G^m\right)\!'(r_1) \leq-\frac{\Tilde{\lambda}}{\Tilde{\alpha}} \, r_1^{\Tilde{N}} \, G(r_1) \, ,
		\end{equation*}
		for all $r_1>r^*$. Such an inequality can be easily rearranged to become
		\begin{equation*}
			\left(G^{m-1}\right)\!'(r)\leq-\frac{\Tilde{\lambda}}{\Tilde{\alpha}}\, \frac{m-1}{m} \, r \qquad \forall r>r^\ast \, ,
		\end{equation*}
which is clearly incompatible with $ \left(G^{m-1}\right)\!(r) \to 0 $ as $r\to+\infty$. Therefore, we can conclude that $L>0$ as desired.

\smallskip		
				
\noindent \emph{Case 3:} Finally, in the critical case $\Tilde{\alpha}m-\Tilde{N}+2=0$, equation \eqref{eq theta 2} reduces to
\begin{equation}\label{reduct-ode}
\begin{aligned}
		\left(\Theta^m\right)\!''= & \left(2\Tilde{\alpha}m-\Tilde{N}+1\right)r^{-1}\left(\Theta^m\right)\!' - \frac{\Tilde{\lambda}}{\Tilde{\alpha}}\, r^{\Tilde{\alpha}(m-1)+1} \, \Theta' \\
		= & \underbrace{\left[  \left(2\Tilde{\alpha}m-\Tilde{N}+1\right)r^{-1} - \frac{\Tilde{\lambda}}{\Tilde{\alpha}m}\, r^{\Tilde{\alpha}(m-1)+1} \, \Theta^{1-m} \right]}_{a(r)}  \left(\Theta^m\right)\!' \, .
		\end{aligned}
\end{equation}
Since $ \Theta(r)>0 $ for all $ r>0 $, the coefficient $a(r)$ is a continuous function on $ (0,+\infty) $. We want to show that $ \Theta' $ can never vanish. In order to derive a contradiction, let us assume that $\Theta'(r_0)=0$ for some $r_0>0$. Then, thanks to \eqref{reduct-ode}, we find that the $C^1$ function $f:=\left(\Theta^m\right)'$ solves the following linear, homogeneous, first-order Cauchy problem:
		\begin{equation*}\label{ode3}
			\begin{cases}
				f'(r)=a(r) \, f(r) & \text{for } r>0 \,, \\
				f(r_0)=0 \, . &
			\end{cases}
		\end{equation*}
Clearly, this implies that $ f(r)=0 $ for all $r>0$, namely $ \Theta $ is constant. Because $ \Theta(0)=0 $, the only possibility is that $ \Theta \equiv 0 $, which is a contradiction. We thus infer that $\Theta$ is monotone and hence increasing, that is the same situation as in case 1.
	\end{proof}
	
	\begin{proof}[Proof of Theorem \ref{main-prop-ode}]
	It is a direct consequence of the above results, in particular Lemma \ref{local existence}, Lemma~\ref{lem-max exist} and Proposition \ref{asymptotics prop}.
	\end{proof}
	
	We are finally in position to construct the self-similar solution $ \mathcal{U}_\alpha $.
	
	\begin{proof}[Proof of Theorem \ref{selfsim-sol}]
	First of all, we observe that we can rewrite the function $ \mathcal{U}_\alpha $, defined in \eqref{Barenblatt}, as
	$$
	\mathcal{U}_\alpha(x,t) = |x|^{-\alpha} \, \Theta\!\left( t^{-\frac{\lambda_\alpha}{\alpha}} |x|  \right) ,
	$$
where, with some notational abuse, we let $ \Theta(r) := r^{\alpha} g_\alpha(r) $. 	Hence, by an explicit computation we obtain
	\begin{equation}\label{ut}
	\begin{aligned}
	\partial_t \, \mathcal{U}_\alpha(x,t) = & - \tfrac{\lambda_\alpha}{\alpha} \,  t^{-\frac{\lambda_\alpha}{\alpha}-1} \, |x|^{1-\alpha}  \, \Theta'\!\left( t^{-\frac{\lambda_\alpha}{\alpha}} |x|  \right) \\
	= & -  t^{-\lambda_\alpha-1} \left[ \tfrac{\lambda_\alpha}{\alpha} \, t^{-\frac{\lambda_\alpha}{\alpha}} |x| \, g_\alpha'\!\left( t^{-\frac{\lambda_\alpha}{\alpha}} |x| \right) + \lambda_\alpha \, g_\alpha\!\left( t^{-\frac{\lambda_\alpha}{\alpha}} |x| \right) \right] .
	\end{aligned}
	\end{equation}	
	On the other hand, using the differential equation in \eqref{ode1} and the definition of $ \lambda_\alpha $, we infer that
	$$
		\begin{aligned}
	& - t^{-\lambda_\alpha-1} \left[ \tfrac{\lambda_\alpha}{\alpha} \, t^{-\frac{\lambda_\alpha}{\alpha}} |x| \, g_\alpha'\!\left( t^{-\frac{\lambda_\alpha}{\alpha}} |x| \right) + \lambda_\alpha \, g_\alpha\!\left( t^{-\frac{\lambda_\alpha}{\alpha}} |x| \right) \right]  \\
	= &  - \frac{|x|^\gamma}{c} \, t^{-\lambda_\alpha\left(1+\frac{\gamma}{\alpha}\right) -1} \left[ (g_\alpha^m)^{''}\!\left( t^{-\frac{\lambda_\alpha}{\alpha}} |x| \right)+\frac{N-1}{t^{-\frac{\lambda_\alpha}{\alpha}} |x|} \, (g_\alpha^m)^{'}\!\left( t^{-\frac{\lambda_\alpha}{\alpha}} |x| \right) \right] \\
	= &  - \frac{|x|^\gamma}{c} \left[ t^{-\lambda_\alpha\left( m + \frac{2}{\alpha} \right)}  (g_\alpha^m)^{''}\!\left( t^{-\frac{\lambda_\alpha}{\alpha}} |x| \right) + t^{-\lambda_\alpha\left( m + \frac{1}{\alpha} \right)} \, \frac{N-1}{|x|} \, (g_\alpha^m)^{'}\!\left( t^{-\frac{\lambda_\alpha}{\alpha}} |x| \right) \right] \\
	= & \, \frac{|x|^\gamma}{c} \, \Delta \!\left( \mathcal{U}_\alpha^m \right)\!(x,t) \, ,
		\end{aligned}
	$$
	thus $ \mathcal{U}_\alpha $ indeed satisfies the differential equation in \eqref{singularWPME}. Rigorously, these computations are justified only away from $ x=0 $. Nevertheless, recalling that $ g_\alpha'(r) = o\!\left(r^{-\gamma/2}\right) $ as $ r \to 0^+ $ and $ g_\alpha(0)>0 $, it follows that $ \left| (g^m)'(r) \right| r^{N-1}  = o\!\left( r^{N-1-\gamma/2} \right) = o(1) $ as $ r \to 0^+ $, which entails that $ \Delta g_\alpha^m(|x|) $ and thus $ \Delta \, \mathcal{U}_\alpha^m $ is actually a (locally integrable) function. It is in fact easier to check that also $ \partial_t \, \mathcal{U}_\alpha $ is a locally integrable function (w.r.t.~$ |x|^{-\gamma} dx $), hence we can assert that the differential equation in \eqref{singularWPME} is satisfied pointwise a.e.~in $ \mathbb{R}^N \times (0,+\infty) $, that is $ \mathcal{U}_\alpha $ is a {strong solution}. In order to prove \eqref{u-alpha-2}, it is enough to observe that there exist two constants $ c_1,c_2>0 $ as in the statement such that
	\begin{equation}\label{g-alpha-est}
\frac{c_1 \, b}{b^{(m-1)\lambda_\alpha}+r^\alpha} \le  g_\alpha(r) \le \frac{c_2 \, b}{b^{(m-1)\lambda_\alpha}+r^\alpha} \qquad \forall r \ge 0 \, ,
	\end{equation}
	from which \eqref{u-alpha-2} just follows by time scaling. The two-sided bound \eqref{g-alpha-est} is in turn a direct consequence of the fact that $g_\alpha(r)$ is continuous, positive and behaves like $ r^{-\alpha} $ at infinity, the dependence on $ b $ being explicit thanks to the spatial scaling \eqref{ode-app-id}.
	
	The continuity of $ t \mapsto \mathcal{U}_\alpha(t) $ in $ (0,+\infty) $, as a curve with values in  $ L^1_{\mathrm{loc}}\!\left( \mathbb{R}^N , |x|^{-\gamma} \right) $, is a straightforward consequence of the continuity of $ g_\alpha $. Moreover, the asymptotic condition in \eqref{ode1} ensures that
	$$
 \lim_{t \to 0^+}	\mathcal{U}_\alpha(x,t)  = b \, |x|^{-\alpha} \qquad \forall x \in \mathbb{R}^N \setminus \{ 0 \} \, ,
	$$
	which, along with \eqref{u-alpha-2}  and  $ \alpha \in (0,N-\gamma)$, guarantees that continuity in $ L^1_{\mathrm{loc}}\!\left( \mathbb{R}^N , |x|^{-\gamma} \right) $ also holds down to $t=0$. As a result of this and the behavior of $ g_\alpha' $ near the origin, we can assert that $ \mathcal{U}_\alpha $ is a solution of problem \eqref{singularWPME} even in the sense of Definition \ref{wpme sol}.  In particular, it falls in the uniqueness class of \cite[Theorem 2.3]{MP}. From identity \eqref{ut} it is readily seen that $ \partial_t \, \mathcal{U}_\alpha \in L^1_{\mathrm{loc}}\!\left( (0,+\infty) ; L^1_{\mathrm{loc}}\!\left( \mathbb{R}^N , |x|^{-\gamma} \right)  \right) $, thus in order to prove \eqref{u-alpha-1} it is enough to show that
	\begin{equation}\label{ut-L1}
	 \int_{0}^T \int_{B_R} \left|  \partial_t \, \mathcal{U}_\alpha(x,t) \right| |x|^{-\gamma} \, dx dt < +\infty  \qquad \forall T,R>0 \, .
	\end{equation}
	First, it is convenient to observe that, by virtue of Proposition \ref{asymptotics prop} (upon undoing \eqref{def-g-tilde} and \eqref{ode-parameters}), there are only two possibilities:
	\begin{enumerate}[(i)]
	\item if $ \alpha \in \left( 0 , \tfrac{N-2}{m} \right]  $ then $ \Theta'(r)>0 $ for all $r>0$; \label{i}
	\item if $ \alpha \in \left( \tfrac{N-2}{m} , N-\gamma \right)  $ then there exists $ r^\ast \in (0,+\infty) $ such that $ \Theta'(r)>0 $ for all $r \in (0,r^\ast)$ and  $ \Theta'(r)<0 $ for all $r \in (r^\ast,+\infty)$. \label{ii}
	\end{enumerate}
The dichotomy property of $ \partial_t \, \mathcal{U}_\alpha $ claimed in the statement is thus an immediate consequence of \eqref{ut} and \eqref{i}--\eqref{ii}. Using again \eqref{ut}, and the radial change of variables $ r = t^{-{\lambda_\alpha}/{\alpha}} |x| $, we can write:
	$$
	\begin{aligned}
  \frac{c}{\omega_{N-1}} \int_{0}^T \int_{B_R} \left|  \partial_t \, \mathcal{U}_\alpha(x,t) \right| |x|^{-\gamma} \, dx dt
  = \underbrace{ c \int_{0}^T t^{-1+\frac{\lambda_\alpha}{\alpha} \left( N-\alpha-\gamma \right) } \,  \int_{0}^{R \, t^{-\frac{\lambda_\alpha}{\alpha}} } \left|  \tfrac{\lambda_\alpha}{\alpha} \,  \Theta'\!\left( r \right)  r^{N-\alpha-\gamma} \right| dr dt }_{=:A} \, .
  \end{aligned}
	$$
Moreover, it is straightforward to verify (recall \eqref{eq theta 1}) that
\begin{equation}\label{idA}
c \, \tfrac{\lambda_\alpha}{\alpha} \,  \Theta'\!\left( r \right)  r^{N-\alpha-\gamma} = -\left( r^{N-1} \left( g_\alpha^m \right)' \right)'(r) \qquad \forall r > 0 \, .
\end{equation}
In case \eqref{i}, we can remove moduli and integrate exploiting \eqref{idA}, to obtain:
$$
\begin{aligned}
A = & \, - R^{N-1} \int_{0}^T t^{-1+\frac{\lambda_\alpha}{\alpha} \left( 1-\alpha-\gamma \right) } \left( g_\alpha^m \right)' \!\left(R \, t^{-\frac{\lambda_\alpha}{\alpha}} \right) dt \\
= & \, \tfrac{\alpha}{\lambda_\alpha} \, R^{N-\alpha-\gamma} \int_{R \, T^{-\frac{\lambda_\alpha}{\alpha}}}^{+\infty} s^{\alpha+\gamma-2} \left( g_\alpha^m \right)' \!\left(s \right) ds \\
= & \,  \tfrac{\alpha}{\lambda_\alpha} \, R^{N-\alpha-\gamma} \left[ \tfrac{R^{\alpha+\gamma-2}}{T^{\frac{\lambda_\alpha}{\alpha} \left( \alpha+\gamma-2 \right)}} \, g_\alpha^m\!\left( R \, T^{-\frac{\lambda_\alpha}{\alpha}} \right) - (\alpha+\gamma-2)  \int_{R \, T^{-\frac{\lambda_\alpha}{\alpha}}}^{+\infty} s^{\alpha+\gamma-3}  \, g_\alpha^m(s) \, ds \right] .
\end{aligned}
$$
Recalling that $ g_\alpha(s) $ behaves like $ s^{-\alpha} $ as $ s \to +\infty $, from the last identity it is readily seen that $ A<+\infty $, therefore \eqref{ut-L1} is indeed satisfied. In case \eqref{ii}, we know that $ \Theta $ is eventually monotone decreasing, hence we can still exploit \eqref{idA} via analogous computations, to show that $A$ is finite and therefore \eqref{ut-L1} is satisfied also in this range of parameters.
\end{proof}

\appendix

\section{Proofs of auxiliary results}\label{technical}

\begin{lem}\label{constr-weak}
	Let  $N\geq3$, $m>1$ and $\rho$ be a measurable function satisfying \eqref{weight-cond} with respect to some $\gamma\in[0,2)$ and $\underline{C},\overline{C}>0$. Let $u_0\in L^1_{\mathrm{loc}}\!\left(\mathbb{R}^N,\rho\right)$ fulfill $ \| u_0 \|_{0,\rho} < +\infty $. Then the solution $u$ constructed in \cite[Theorem 2.2]{MP} is a local weak energy solution of \eqref{wpme} in the sense of Definition~\ref{wpme sol}. Let $Q' :=B_R\times(t_0,t_1)$ and $Q:=B_{2R}\times\left(\tfrac{t_0}{2},2\,t_1\right)$ for $R>0$ and $0<t_0<t_1$. The energy estimate
	\begin{equation}\label{energy-est}
		\norm{\nabla u^m}_{L^2(Q')} \leq C\left(\norm{u}^{m+1}_{L^\infty(Q)} + \norm{u}^{2m}_{L^\infty(Q)} \right)
	\end{equation}
	holds for some $C>0$ depending only on $N,m,\gamma,\underline{C},\overline{C},R,t_0,t_1$.
\end{lem}

\begin{proof}
	We recall from \cite[Subsection 3.4]{MP} that $u$ can be uniquely obtained as the limit of solutions corresponding to the truncated initial data
	\[u_{0n}:=\tau_n(u_0)\,\chi_{B_n},\]
	where
	\begin{equation*}\label{truncation-function}
		\tau_n(s) :=
		\begin{cases}
			s & \text{if}\ -n < s < n \, ,\\
			n & \text{if}\ n\leq s \, , \\
			-n &  \text{if}\ s \leq-n \, ,
		\end{cases}
	\end{equation*}
	for all $n\in\mathbb{N}$. By \cite[Proposition 3.3]{MP}, there exists a global weak solution $u_n$ of \eqref{wpme} that takes $u_{0n}$ as its initial datum. The following energy estimate is derived as in \cite[Section 3.2.4]{Vaz} by multiplying~\eqref{wpme} by $\phi=u_n^m\,\eta^2$ and then integrating by parts repeatedly, where $\eta$ is a standard space-time cutoff function interpolating between $1$ on $Q'$ and $0$ outside $Q$:
	\begin{equation}\label{energy-est-app}
		\int_{0}^\infty\int_{\R^N} \left|\nabla u_n^m\right|^2\eta^2\,dxdt\leq C \int_{0}^\infty\int_{\R^N} \Big[ |u_n|^{m+1}\left(\partial_t \eta^2\right) \rho+|u_n|^{2m}|\nabla\eta|^2 \Big] \, dxdt \, .
	\end{equation}
	In fact, to rigorously obtain such an estimate, we must further approximate $u_{0n}$ and $\rho$ by smooth and non-degenerate objects and then pass to the appropriate limit in a standard way. Due to the smoothing estimate Proposition \ref{p-transl}, the right-hand side of \eqref{energy-est-app} is uniformly bounded with respect to $n$, so that $\left\{u_n^m\right\}$ is uniformly bounded in $L^2_{\textrm{loc}}\!\left((0,+\infty); H^1_{\textrm{loc}}\!\left(\R^N\right)\right)$. The reflexivity of such a space implies that the constructed solution $u$ is in fact locally weak and we can safely pass to the limit in~\eqref{energy-est-app}, whence \eqref{energy-est} easily follows.
\end{proof}

\begin{rem}\label{loc-weak}\rm
	Up to using a different smoothing estimate, the same argument can be used to show that \emph{all} constructed solutions in \cite[Theorem 2.2]{MP} (even with growing initial data) are local weak energy solutions of \eqref{wpme} in the sense of Definition \ref{wpme sol}, where in general $t$ might only range in $(0,T)$ for some finite $T>0$.
\end{rem}

	\begin{proof}[Proof of Lemma \ref{h1 compact}]
	First of all, we write a series of key estimates for $ v_k $, that are by now standard and can be rigorously derived {e.g.}~with the methods employed in \cite[Section 3]{GMPo}:
	\begin{equation}\label{decrease-norms}
		\left\| v_k(t) \right\|_{L^p\left( \mathbb{R}^N , \rho_k \right)} \le \left\| u_0 \right\|_{L^p\left( \mathbb{R}^N , \rho_k \right)} \qquad \forall t > 0 \, , \ \forall p \in [1,\infty] \, ,
	\end{equation}
	\begin{equation}\label{energy-1}
		\int_0^{+\infty} \int_{\mathbb{R}^N} \left| \nabla v_k^m \right|^2 dx dt \le  C_m \left\| u_0 \right\|_{L^{m+1}\left( \mathbb{R}^N , \rho_k \right)}^{m+1} ,
	\end{equation}
	\begin{equation}\label{energy-2}
		\int_t^{+\infty} \int_{\mathbb{R}^N} \left| \partial_t \!\left( v_k^{\frac{m+1}{2}} \right) \right|^2 \rho_k \,  dx ds  \le  \frac{C_m}{t} \left\| u_0 \right\|_{L^{m+1}\left( \mathbb{R}^N , \rho_k \right)}^{m+1} \qquad \forall t > 0 \, ,
	\end{equation}	
	where $C_m>0$ is a suitable constant depending only on $m$. Moreover, for positive solutions (such as $v_k$), the bound \eqref{decrease-norms} at $ p=1 $ becomes an identity, namely mass conservation (see \cite[Theorem 5.2]{RV1}):
	\begin{equation*}\label{mass-cons}
		\left\| v_k(t) \right\|_{L^1\left( \mathbb{R}^N , \rho_k \right)} = \left\| u_0 \right\|_{L^1\left( \mathbb{R}^N , \rho_k \right)} := M_k \qquad \forall t > 0 \, .
	\end{equation*}	
	We now split the proof in several steps.
	
	\medskip
	
	\noindent{\textbf{Claim 1:}} \emph{We have}
	\begin{equation}\label{C1-B}
		\lim_{k \to \infty} v_k^m = u^m \qquad \textit{in } L^2_{\mathrm{loc}}\!\left( (0,+\infty) ; L^2_{\mathrm{loc}}\!\left( \mathbb{R}^N \right) \right) .
	\end{equation}
	If we combine \eqref{decrease-norms}--\eqref{energy-2} and \eqref{weight-cond-scaled}, it is readily seen that
	$$
	\left\{ v^m_k \right\} \text{ is bounded in } L^2_{\mathrm{loc}}\!\left( [ 0,+\infty) ;H^1_{\mathrm{loc}}\!\left( \mathbb{R}^N \right) \right)
	$$
	and
	$$
	\left\{ \partial_t \! \left( v^m_k \right) \right\} \text{ is bounded in } L^2_{\mathrm{loc}}\!\left( ( 0,+\infty) ; L^2_{\mathrm{loc}}\!\left( \mathbb{R}^N \right) \right) .
	$$
	Indeed, the right-hand sides of \eqref{decrease-norms}--\eqref{energy-2} are all bounded in $ k $, as
	\begin{equation}\label{equi-bdd-norms}
		\left\| u_0 \right\|_{L^p\left( \mathbb{R}^N , \rho_k \right)} \le \overline{C}^{\frac 1 p} \left\| u_0 \right\|_{L^\infty\left( \mathbb{R}^N \right)}^{\frac{p-1}{p}} \left\| u_0 \right\|_{L^1\left( \mathbb{R}^N , |x|^{-\gamma} \right)}^{\frac 1 p} \qquad \forall p \in [1,\infty) \, ,
	\end{equation}
	and $ \rho_k $ is locally uniformly bounded away from zero. We are therefore in position to apply the standard Aubin-Lions compactness lemma, which guarantees the existence of some function $ w $ such that
	\begin{equation}\label{conv-uk-1}
		\lim_{k \to \infty} v_k^m = w \qquad \text{in } L^2_{\mathrm{loc}}\!\left( (0,+\infty) ; L^2_{\mathrm{loc}}\!\left( \mathbb{R}^N \right) \right) ,
	\end{equation}
	up to subsequences. On the other hand, since
	\begin{equation}\label{numer-eq}
		\left| A - B \right| \le \left| A^m - B^m \right|^{\frac{1}{m}} \qquad \forall A,B \in \mathbb{R}^+ \, ,
	\end{equation}
	it is plain that also
	\begin{equation}\label{conv-uk-2}
		\lim_{k \to \infty} v_k = w^{\frac 1 m} \qquad \text{in } L^2_{\mathrm{loc}}\!\left( (0,+\infty) ; L^2_{\mathrm{loc}}\!\left( \mathbb{R}^N \right) \right) .
	\end{equation}
	Let us write the weak energy formulation of \eqref{wpme-resc-u0} satisfied by $v_k$:
	\begin{equation}\label{q-k-1}
		\begin{aligned}
		 \int_0^{+\infty} \int_{\mathbb{R}^N} v_k \, \phi_t \, \rho_k \, dx dt  = - \int_{\mathbb{R}^N} u_0 \, \phi(0) \, \rho_k \, dx + \int_0^{+\infty} \int_{\mathbb{R}^N} \nabla v_k^m \cdot \nabla \phi \, dx dt \, ,
		\end{aligned}
	\end{equation}
	for all $\phi\in C^\infty_c\!\left(\mathbb{R}^N\times [0, +\infty)\right)$. By combining \eqref{weight-cond2}, \eqref{decrease-norms} (with $p=\infty$), \eqref{energy-1}, \eqref{conv-uk-1} and \eqref{conv-uk-2}, it is not difficult to check that one can pass to the limit safely in \eqref{q-k-1} as $ k \to \infty $ (still along a subsequence), obtaining
	\begin{equation*}\label{q-k-2}
		\begin{aligned}
		  \int_0^{+\infty} \int_{\mathbb{R}^N} u \, \phi_t \, c \, |x|^{-\gamma}  \, dx dt  = - \int_{\mathbb{R}^N} u_0 \, \phi(0) \, c \, |x|^{-\gamma} \, dx dt + \int_0^{+\infty} \int_{\mathbb{R}^N} \nabla u^m \cdot \nabla \phi \, dx dt \, ,
		\end{aligned}
	\end{equation*}
	where we set $ u := w^{\frac 1 m} $. Furthermore, by lower semi-continuity we infer
	\begin{equation*}\label{decrease-norms-limit}
		\left\| u(t) \right\|_{L^p\left( \mathbb{R}^N , |x|^{-\gamma} \right)} \le \left\| u_0 \right\|_{L^p\left( \mathbb{R}^N , |x|^{-\gamma} \right)} \qquad \forall t > 0 \, , \ \forall p \in [1,\infty] \, ,
	\end{equation*}
	\begin{equation*}\label{energy-1-limit}
		\int_0^{+\infty} \int_{\mathbb{R}^N} \left| \nabla u^m \right|^2 dx dt \le  C_m \, c \left\| u_0 \right\|_{L^{m+1}\left( \mathbb{R}^N , |x|^{-\gamma} \right)}^{m+1} .
	\end{equation*}
	As a result, $u$ is also a weak energy solution of \eqref{wpme-resc-u1}, which is uniquely identified thanks to \cite[Proposition 6]{GMPo}. Therefore, the limit being independent of the particular subsequence, the claim is established. Note that the notions of weak energy solution of \cite[Definition 3.1]{MP} and \cite[Definition~3.1]{GMPo} basically differ in that the latter does not require continuity in $ L^1\!\left( \mathbb{R}^N , |x|^{-\gamma} \right) $ and thus the initial datum appears directly in the weak formulation. This is convenient here because \emph{a priori} we have no information on such continuity property for the limit $ u $. However, \emph{a posteriori} the notions coincide for the class of data we consider.
	
	\medskip
	
	\noindent{\textbf{Claim 2:}}  \emph{We have}
	\begin{equation}\label{C2}
		\lim_{k \to \infty} v_k^m = u^m \quad \textit{in } C\!\left( [T_1,T_2]  ; L^2_{\mathrm{loc}}\!\left( \mathbb{R}^N \right) \right) , \qquad \forall T_2>T_1>0 \, .
	\end{equation}
	By virtue of \eqref{energy-2} and \eqref{equi-bdd-norms}, it turns out that the sequence $ \left\{ v_k^m \right\} $ is in fact equicontinuous in $ L^2_{\mathrm{loc}}\!\left(  \mathbb{R}^N\right) $, on any compact time interval $ [T_1,T_2] $:
	\begin{equation}\label{equi-cont-1}
		\begin{aligned}
			\left\| v_k(t)^m - v_k(s)^m \right\|_{L^2_{\mathrm{loc}}\left(  \mathbb{R}^N\right) }
			\le & \, \int_s^t  \left\| \partial_\tau \! \left( v_k^m \right) \! (\tau) \right\|_{L^2_{\mathrm{loc}}\left(  \mathbb{R}^N\right) } d\tau \\
			\le & \, \frac{2m}{m+1} \, \sqrt{t-s} \left\| u_0 \right\|_{L^\infty\left( \mathbb{R}^N \right)}^{\frac{m-1}{2}} \left\| \partial_\tau \! \left( v_k^{\frac{m+1}{2}} \right) \right\|_{L^2\left(  (T_1,T_2) ; L^2_{\mathrm{loc}}\left( \mathbb{R}^N \right) \right)}  \\
			\le & \, C \, \sqrt{t-s} \, ,
		\end{aligned}
	\end{equation}
	for every $ t>s $ lying in $ [T_1,T_2]  $. Here $C>0$ is a suitable constant that depends on $m , \gamma , \underline{C} , \overline{C} ,T_1 , u_0  $ and the precompact set of $ \mathbb{R}^N $ one fixes, but is independent of $k$. On the other hand, thanks to \eqref{C1-B} we can extract a subsequence of  $ \left\{ v_k^m \right\} $ (that we do not relabel) such that
	\begin{equation}\label{equi-cont-2}
		\lim_{k \to \infty} v_k(t)^m = u(t)^m  \qquad \text{in } L^2_{\mathrm{loc}}\!\left( \mathbb{R}^N \right) , \ \text{for a.e. } t \in (T_1,T_2) \, .
	\end{equation}
	Hence \eqref{C2} is a straightforward consequence of the Ascoli-Arzel\`a theorem with values in Banach spaces, due to \eqref{equi-cont-1}, \eqref{equi-cont-2} and the fact that the limit is independent of the particular subsequence.
	
	\medskip	
	
	\noindent{\textbf{Claim 3:}}  \emph{We have}
	\begin{equation}\label{C3}
		\lim_{k \to \infty} v_k = u \quad \textit{in } C\!\left( [T_1,T_2] ; L^p\!\left( \mathbb{R}^N , |x|^{-\gamma}  \right) \right) , \qquad \forall T_2>T_1>0 \, , \ \forall p \in [1,\infty) \, .
	\end{equation}	
	First of all, let us show that \eqref{C3} holds with $ L^p\!\left( \mathbb{R}^N , |x|^{-\gamma}  \right) $ replaced by $L^p_{\mathrm{loc}}\!\left( \mathbb{R}^N , |x|^{-\gamma}  \right)$. To this purpose, we can assume without loss of generality that $ p \ge 2m $. Given any $ R>T_2^{1/2} > \varepsilon>0 $, upon recalling \eqref{decrease-norms}, \eqref{numer-eq} and \eqref{equi-cont-1}, we have:
	\begin{equation*}\label{equi-cont-3}
		\begin{aligned}
			& \left\| v_k(t) - v_k(s) \right\|_{L^p\left( B_R , |x|^{-\gamma} \right) }  \\
			\le &  \left\| v_k(t)^m - v_k(s)^m \right\|_{L^{\frac p m}\left( B_R , |x|^{-\gamma} \right) }^{\frac 1 m}  \\
			\le & \left( \tfrac{\omega_{N-1}}{N-\gamma} \right)^{\frac 1 p} \left\| u_0 \right\|_{L^\infty\left( \mathbb{R}^N \right)} \varepsilon^{\frac{N-\gamma}{p}} + \varepsilon^{-\frac \gamma p} \left\| v_k(t)^m - v_k(s)^m \right\|_{L^{\frac p m}\left( B_R \setminus B_\varepsilon  \right) }^{\frac 1 m}  \\
			\le & \left( \tfrac{\omega_{N-1}}{N-\gamma} \right)^{\frac 1 p} \left\| u_0 \right\|_{L^\infty\left( \mathbb{R}^N \right)} \varepsilon^{\frac{N-\gamma}{p}} + \left\| u_0 \right\|_{L^\infty\left( \mathbb{R}^N \right)}^{\frac{p-2m}{p}} \varepsilon^{-\frac \gamma p} \left\| v_k(t)^m - v_k(s)^m \right\|_{L^{2}\left( B_R  \right) }^{\frac 2 p}  \\
			\le & \left( \tfrac{\omega_{N-1}}{N-\gamma} \right)^{\frac 1 p} \left\| u_0 \right\|_{L^\infty\left( \mathbb{R}^N \right)} \varepsilon^{\frac{N-\gamma}{p}} + \left\| u_0 \right\|_{L^\infty\left( \mathbb{R}^N \right)}^{\frac{p-2m}{p}} \varepsilon^{-\frac \gamma p} \, C_R^{\frac 2 p} \left( t-s \right)^{\frac 1 p} ,
		\end{aligned}
	\end{equation*}
	for every $ t>s $ lying in $ [T_1,T_2]  $, where $ C_R $ is the constant appearing in \eqref{equi-cont-1} in the specific case of $B_R$. Now, if we pick \emph{e.g.}~$ \varepsilon = (t-s)^{1/2} $ we end up with
	\begin{equation*}\label{equi-cont-4}
		\left\| v_k(t) - v_k(s) \right\|_{L^p\left( B_R , |x|^{-\gamma} \right) } \le \left( \tfrac{\omega_{N-1}}{N-\gamma} \right)^{\frac 1 p} \left\| u_0 \right\|_{L^\infty\left( \mathbb{R}^N \right)} \left( t-s \right)^{\frac{N-\gamma}{2 p}} + \left\| u_0 \right\|_{L^\infty\left( \mathbb{R}^N \right)}^{\frac{p-2m}{p}} C_R^{\frac 2 p} \left( t-s \right)^{\frac {2-\gamma}{2p}}  ,
	\end{equation*}
	that is the equicontinuity of $ \left\{ v_k \right\} $ in $L^p_{\mathrm{loc}}\!\left( \mathbb{R}^N , |x|^{-\gamma}  \right)$. Moreover, still by using the fact that the sequence is uniformly bounded by $ \left\| u_0 \right\|_{L^\infty\left( \mathbb{R}^N \right)} $ and the local integrability of $ |x|^{-\gamma} $, it is straightforward to deduce from \eqref{C2} that
	$$
	\lim_{k \to \infty} v_k(t) = u(t)  \qquad \text{in } L^p_{\mathrm{loc}}\!\left( \mathbb{R}^N , |x|^{-\gamma} \right) , \ \forall t \in [T_1,T_2] \, .
	$$
	Therefore, a further application of the Ascoli-Arzel\`a theorem ensures that
	\begin{equation}\label{C3-bis}
		\lim_{k \to \infty} v_k = u \quad \text{in } C\!\left( [T_1,T_2] ; L^p_{\mathrm{loc}}\!\left( \mathbb{R}^N , |x|^{-\gamma}  \right) \right)  \qquad \forall T_2>T_1>0 \, , \ \forall p \in [1,\infty) \, .
	\end{equation}
	We finally exploit mass conservation to pass from local to global convergence. It is plain that, by dominated convergence,
	\begin{equation}\label{mass-cons-limit-0}
		M : = \lim_{k \to \infty} M_k = c \left\| u_0 \right\|_{L^1\left( \mathbb{R}^N , |x|^{-\gamma} \right)} ,
	\end{equation}	
	and since mass conservation also holds for the limit problem \eqref{wpme-resc-u1}, we have
	\begin{equation}\label{mass-cons-limit}
		c \left\| u(t) \right\|_{L^1\left( \mathbb{R}^N , |x|^{-\gamma} \right)} = M \qquad \forall t > 0 \, .
	\end{equation}	
	By taking advantage of this property, we aim at showing that for every $ \varepsilon>0 $ there exist $ R_\varepsilon , k_\varepsilon >0 $ large enough such that
	\begin{equation}\label{mass-compact}
		\left\| v_k(t) \right\|_{L^1\left( B_{R_\varepsilon}^c , |x|^{-\gamma} \right)} \le \varepsilon \qquad \forall t \in [T_1,T_2] \, , \ \forall k>k_\varepsilon \, .
	\end{equation}	
	Let us argue by contradiction. If \eqref{mass-compact} failed, then there would exist $ \varepsilon_0>0 $, a sequence $ R_j \to +\infty $, a subsequence  $ k_j \to \infty $ and a corresponding sequence of times $ \left\{ t_j \right\} \subset [T_1,T_2] $ such that
	\begin{equation}\label{mass-compact-2}
		\left\| v_{k_j}\!\left(t_j\right) \right\|_{L^1\left( B_{R_j}^c , |x|^{-\gamma} \right)} > \varepsilon_0 \qquad \forall j \in \mathbb{N} \, .
	\end{equation}	
	Up to taking a further subsequence, we may assume that $ t_j \to t_\ast \in [T_1,T_2] $ as $ j \to \infty $. Hence, for every fixed $R>1$, by means of \eqref{weight-cond-scaled}, \eqref{C3-bis}, \eqref{mass-cons-limit-0} and \eqref{mass-compact-2} we would infer
	\begin{equation}\label{mass-compact-3}
		c \left\| u\!\left(t_\ast\right) \right\|_{L^1\left( B_{R} , |x|^{-\gamma} \right)}  = \lim_{j \to \infty}  \left\| v_{k_j}\!\left(t_j\right) \right\|_{L^1\left( B_{R} , \rho_{k_j} \right)} \le M -\frac{\varepsilon_0}{2^\gamma \underline{C}}  \, .
	\end{equation}	
	However, if we let $ R \to +\infty $ in \eqref{mass-compact-3}, we reach a contradiction with \eqref{mass-cons-limit}. As a result, we can obtain:
	$$
	\left\| v_k - u \right\|_{C\left( [T_1,T_2] ; L^p\left( \mathbb{R}^N , |x|^{-\gamma}  \right) \right)} \le \left\| u_0  \right\|_{L^\infty\left( \mathbb{R}^N \right)}^{\frac{p-1}{p}} \left(  \left\| v_k - u \right\|_{C\left( [T_1,T_2] ; L^1\left( B_{R_\varepsilon} , |x|^{-\gamma}  \right) \right)} + 2 \varepsilon \right)^{\frac 1 p}
	$$
	for all  $k > k_\varepsilon$, whence \eqref{C3} by first letting $ k \to \infty $ and then $ \varepsilon \to 0 $.
	
	\medskip
	
	\noindent {\textbf{End of proof.}} From \eqref{q-k-1}, after a standard approximation argument, one can deduce that
	\begin{equation*}\label{q-k-3}
			\int_{\mathbb{R}^N} v_k(x,t) \, \psi(x) \, \rho_k(x)  \, dx   = \int_{\mathbb{R}^N} u_0(x) \, \psi(x) \, \rho_k(x) \, dx + \int_0^{t} \int_{\mathbb{R}^N}  v_k(x,s)^m \, \Delta \psi(x) \, dx ds \, ,
	\end{equation*}
	for all $ t>0 $ and all $\psi\in C^\infty_c\!\left(\mathbb{R}^N \right)$. In particular, we obtain:
	\begin{equation}\label{q-k-4}
		\begin{aligned}
			& \, c \left| \int_{\mathbb{R}^N} v_k(x,t) \, \psi(x) \, |x|^{-\gamma}  \, dx - \int_{\mathbb{R}^N} u_0(x) \, \psi(x) \, |x|^{-\gamma} \, dx \right| \\
			\le & \,  t \left\| u_0 \right\|_{L^\infty\left( \mathbb{R}^N \right)}^m \left\| \Delta \psi \right\|_{L^1\left( \mathbb{R}^N \right)} + \left\| u_0 \right\|_{L^\infty\left( \mathbb{R}^N \right)}  \left\| \psi \right\|_{L^\infty\left( \mathbb{R}^N \right)} \left\| \rho_k - c\,|x|^{-\gamma} \right\|_{L^1\left(B_{R_0} \right)} ,
		\end{aligned}
	\end{equation}
	where $  R_0> 0$	is so large that $ \operatorname{supp} \psi \subset B_{R_0} $. Therefore, if $ \{ t_k \} \subset (0,+\infty) $ is an arbitrary sequence such that $ t_k \to 0 $, thanks to \eqref{weight-cond2} and \eqref{q-k-4} we can infer
	\begin{equation}\label{q-k-5}
		\lim_{k \to \infty} \int_{\mathbb{R}^N} v_k(x,t_k) \, \psi(x) \, |x|^{-\gamma}  \, dx	= \int_{\mathbb{R}^N} u_0(x) \, \psi(x) \, |x|^{-\gamma} \, dx \qquad \forall \psi\in C^\infty_c\!\left(\mathbb{R}^N \right) .
	\end{equation}
	On the other hand, from \eqref{weight-cond-scaled} and \eqref{decrease-norms} it is plain that
	\begin{equation}\label{q-k-6}
		\sup_{k \in \mathbb{N}} \left\| v_k(t_k) \right\|_{L^p\left( \mathbb{R}^N , |x|^{-\gamma} \right)} <+\infty \, .
	\end{equation}
	Hence, in view of \eqref{q-k-5}--\eqref{q-k-6}, it follows that
	\begin{equation}\label{weak-1}	
		v_k(t_k)	\underset{k \to \infty}{\longrightarrow} u_0 \qquad \text{weakly in } L^p\!\left( \mathbb{R}^N , |x|^{-\gamma} \right) \! , \ \text{for all } p \in (1,\infty) \, .
	\end{equation}
	In particular, weak convergence also holds locally in $L^1\!\left( \mathbb{R}^N , |x|^{-\gamma} \right)$. Similarly to \eqref{mass-compact}, we can now show that for every $ \varepsilon > 0 $ there exist $ R_{\varepsilon} , k_\varepsilon >0  $ such that
	\begin{equation}\label{mass-compact-final-1}
		\left\| v_k(t_k) \right\|_{L^1\left( B_{R_\varepsilon}^c , |x|^{-\gamma} \right)} \le \varepsilon \qquad \forall k>k_\varepsilon \, .
	\end{equation}	
	Indeed, if \eqref{mass-compact-final-1} failed, there would exist $ \varepsilon_0>0 $, a sequence $ R_j \to +\infty $ and a sequence $ k_j \to \infty $ such that
	\begin{equation*}\label{mass-compact-final-2}
		\left\| v_{k_j}\!\left(t_{k_j}\right) \right\|_{L^1\left( B_{R_j}^c , |x|^{-\gamma} \right)} > \varepsilon_0 \qquad \forall j \in \mathbb{N} \, .
	\end{equation*}
	Therefore, for every fixed $R>1$, we would infer
	\begin{equation*}\label{mass-compact-final-3}
		c \left\| u_0 \right\|_{L^1\left( B_{R} , |x|^{-\gamma} \right)} \le \liminf_{j\to \infty} \left\| v_{k_j}\!\left(t_{k_j}\right) \right\|_{L^1\left( B_{R} , \rho_{k_j} \right)}  \le M -\frac{\varepsilon_0}{2^\gamma \underline{C}}  \, ,
	\end{equation*}		
	where in the first passage we used the weak lower semi-continuity of the local $ L^1\!\left( \mathbb{R}^N , |x|^{-\gamma} \right) $ norm along with \eqref{weight-cond2}. However, upon letting $ R \to +\infty $, we reach a contradiction to the very definition of $M$. Still in view of \eqref{weight-cond2}, it is apparent that for every $ \varepsilon>0 $ it holds
	\begin{equation*}\label{mass-compact-final-4}
		\lim_{k \to \infty} \left|	\left\| v_{k}\!\left(t_{k}\right) \right\|_{L^p\left( B_{R_\varepsilon} , \rho_{k} \right)} - c \left\| v_{k}\!\left(t_{k}\right) \right\|_{L^p\left( B_{R_\varepsilon} ,|x|^{-\gamma} \right)}  \right| =  0 \qquad \forall p \in [1,\infty) \, ,
	\end{equation*}
	whence, upon again recalling \eqref{decrease-norms} and the weak lower semi-continuity of $ L^p\!\left( \mathbb{R}^N , |x|^{-\gamma} \right) $ norms,
	\begin{equation*}\label{mass-compact-final-5}
		\begin{aligned}
			c \left\| u_0 \right\|_{L^p\left( \mathbb{R}^N ,|x|^{-\gamma} \right)} \le & \, c \, \liminf_{k \to \infty} \left\| v_{k}\!\left(t_{k}\right) \right\|_{L^p\left( \mathbb{R}^N ,|x|^{-\gamma} \right)}  \\
			\le & \, c \, \limsup_{k \to \infty} \left\| v_{k}\!\left(t_{k}\right) \right\|_{L^p\left( \mathbb{R}^N ,|x|^{-\gamma} \right)} \\
			\le & \, c \, \limsup_{k \to \infty} \left( \left\| v_{k}\!\left(t_{k}\right) \right\|_{L^p\left( B_{R_\varepsilon} ,|x|^{-\gamma} \right)} + \left\| v_{k}\!\left(t_{k}\right) \right\|_{L^p\left( B_{R_\varepsilon}^c ,|x|^{-\gamma} \right)}  \right) \\
			\le & \, \limsup_{k \to \infty} \left\| v_{k}\!\left(t_{k}\right) \right\|_{L^p\left( B_{R_\varepsilon} , \rho_k  \right)} + c \, \limsup_{k \to \infty}  \left\| v_{k}\!\left(t_{k}\right) \right\|_{L^p\left( B_{R_\varepsilon}^c ,|x|^{-\gamma} \right)}  \\
			\le & \, c \left\| u_0 \right\|_{L^p\left( \mathbb{R}^N ,|x|^{-\gamma} \right)} + c \left\| u_0 \right\|_{L^\infty\left( \mathbb{R}^N \right) }^{\frac{p-1}{p}} \varepsilon^{\frac 1 p} .
		\end{aligned}
	\end{equation*}
	Since $\varepsilon>0 $ is arbitrary, it follows that
	$$
	\lim_{k \to \infty} \left\| v_{k}\!\left(t_{k}\right) \right\|_{L^p\left( \mathbb{R}^N ,|x|^{-\gamma} \right)} = \left\| u_0 \right\|_{L^p\left( \mathbb{R}^N ,|x|^{-\gamma} \right)} ,
	$$	
	which, together with \eqref{weak-1}, ensures that
	\begin{equation}\label{weak-2}	
		v_k(t_k)	\underset{k \to \infty}{\longrightarrow} u_0 \qquad \text{strongly in } L^p\!\left( \mathbb{R}^N , |x|^{-\gamma} \right)\! , \ \text{for all } p \in [ 1,\infty) \, .
	\end{equation}
	Note that, rigorously, this holds for $ p \in (1,\infty) $	 only, but the extension to the case $ p=1 $ is again a straightforward consequence of \eqref{mass-compact-final-1}. Because we have established \eqref{weak-2} for any vanishing sequence $ \{ t_k \} $, this means that for every $ \varepsilon >0 $ there exist $ \delta , k_\varepsilon>0 $ such that
	\begin{equation}\label{weak-3}	
		\left\| v_k(t) - u_0 \right\|	_{L^p\left( \mathbb{R}^N , |x|^{-\gamma} \right)} < \varepsilon \qquad \forall t \in [0,\delta] \, , \ \forall k>k_\varepsilon \, .
	\end{equation}	
	The claimed convergence \eqref{L1-conv-cont} is thus a consequence of \eqref{C3} and \eqref{weak-3}.
\end{proof}

\begin{proof}[Proof of Lemma \ref{lowbar-lemma-lemma}]
For simplicity, we will prove the case $ \ell<\mathcal{C} $ only. The thesis in the case $ \ell \ge \mathcal{C} $ can be reached with standard modifications that will be briefly described at the end of the proof.

The monotonicity property $ v_t \ge 0 $ is a consequence of the fact that, formally, the function $ V = v_t $ satisfies in turn a parabolic problem of the type
	\begin{equation}\label{lowbar-vt}
		\begin{cases}
			\overline{C}\left|x\right|^{-\gamma} V_t= \Delta\!\left( m \, |v|^{m-1} \, V \right) & \text{in } B_1 \times(0,+\infty) \, , \\
			V \ge 0 & \text{on } \partial B_1 \times(0,+\infty) \, ,\\
			V = 0 & \text{on } B_1 \times \{ 0 \} \, , \\
		\end{cases}
	\end{equation}
	for which a maximum principle holds. Note that this can be justified rigorously by approximating the weight, the boundary data and the nonlinearity $ v \mapsto v^m $ with regular and nondegenerate objects, via a standard procedure that we omit (see e.g.~\cite[Chapter 5]{Vaz}).
	
	In order to prove the uniform convergence \eqref{conv-unif-v}, it is convenient to rewrite \eqref{lowbar-lemma} as a \emph{homogeneous} Dirichlet problem upon setting $ w:= \mathcal{C}-v $ and observing that the latter satisfies
	\begin{equation}\label{lowbar-hom}
		\begin{cases}
			\overline{C}\left|x\right|^{-\gamma} w_t = \operatorname{div}\!\left( m \left| \mathcal{C}-w \right|^{m-1} \nabla w \right) & \text{in } B_1 \times(0,+\infty) \, , \\
			w = 0 & \text{on } \partial B_1 \times(0,+\infty) \, ,\\
			w = \mathcal{C}-\ell  & \text{on } B_1 \times \{ 0 \} \, . \\
		\end{cases}
	\end{equation}
	Hence \eqref{conv-unif-v} amounts to proving that
	\begin{equation}\label{conv-unif-w}
		\lim_{t \to +\infty}\left\| w(t)  \right\|_{L^\infty \left( B_1 \right)} = 0 \, ;
	\end{equation}
	to this aim, one can set up a Moser-type iteration, of which we will only highlight the key points as it is by now a rather standard tool. First of all, by comparison we have that $ 0 \le w \le \mathcal{C}-\ell  $. In particular, thanks to \cite[Lemma 5.8]{GM13}, we can assert that there exist constants $ \kappa , q >0 $, depending only on $ m $ and $ \mathcal{C} $, such that
	\begin{equation}\label{nonlin-below}
		\Psi_p(w) := \int_0^w \left| \mathcal{C}-y \right|^{\frac{m-1}{2}}  y^{\frac{p-2}{2}} \, dy  \ge \frac{\kappa}{p^{q}} \, w^{\frac{p}{2}}  \qquad \forall p \ge 2 \, .
	\end{equation}
	Therefore, upon (formally) multiplying both sides of the differential equation in \eqref{lowbar-hom} by $ p \, w^{p-1} $, and integrating by parts in $ B_1 \times (t_0,t_1) $ (let $ t_1>t_0 \ge 0 $), we obtain:
	\begin{equation}\label{first-step-smooth}
		\left\| w(t_1) \right\|_{L^p\left( B_1 , |x|^{-\gamma} \right)}^p - \left\| w(t_0) \right\|_{L^p\left( B_1 , |x|^{-\gamma} \right)}^p = - \frac{m \, p \,  (p-1)}{\overline{C}} \int_{t_0}^{t_1} \int_{B_1} \left| \nabla \Psi_p(w) \right|^2 dx \, dt \, ,
	\end{equation}
	whence, using \eqref{nonlin-below} and the weighted Sobolev inequality associated with $ |x|^{-\gamma} dx $  (see e.g.~\cite[formula (3.13)]{MP}), we end up with the estimate
	\begin{equation}\label{second-step-smooth}
		\begin{aligned}
			\frac{m \, p \,  (p-1) \, \kappa^2}{\overline{C} \, C_S \, p^{2q}} \int_{t_0}^{t_1} \left\| w(t) \right\|_{L^{\frac{2^\ast}{2} p} \left( B_1 , \, |x|^{-\gamma}  \right) }^p  \, dt \le & \,  \frac{m \, p \,  (p-1)}{\overline{C} \, C_S} \int_{t_0}^{t_1} \left\| \Psi_p(w(t)) \right\|_{L^{2^\ast} \! \left( B_1 , \, |x|^{-\gamma}  \right) }^2  \, dt  \\
			\le & \left\| w(t_0) \right\|_{L^p\left( B_1 , \, |x|^{-\gamma} \right)}^p ,
		\end{aligned}
	\end{equation}
	where $ 2^\ast := 2(N-\gamma)/(N-2) $ and $C_S>0$ is the Sobolev embedding constant. By exploiting the fact that $ L^p $ norms decrease in time (trivial consequence of \eqref{first-step-smooth}), and raising to the $ 1 / p  $ power, from~\eqref{second-step-smooth} we end up with
	\begin{equation}\label{third-step-smooth}
		\left\| w(t_1) \right\|_{L^{\frac{2^\ast}{2} p} \left( B_1 \, , |x|^{-\gamma}  \right) }  \le \left[ \frac{\overline{C} \, C_S \, p^{2q}}{m \, p \,  (p-1) \, \kappa^2  \left( t_1-t_0 \right)} \right]^{\frac 1 p} \left\| w(t_0) \right\|_{L^p\left( B_1 , |x|^{-\gamma} \right)} .
	\end{equation}
	Now, we set $ t_0 \equiv t_n := (1-2^{-n}) \, t  $ and $ t_1 \equiv t_{n+1} := \left(1-2^{-n-1}\right) t  $, for a fixed but arbitrary $t>0$, and $ p \equiv p_n := \left( {2^\ast}/{2} \right)^n 2$ (let $n \in \mathbb{N}$), so that \eqref{third-step-smooth} can be written recursively as
	\begin{equation}\label{fourth-step-smooth}
		\left\| w(t_{n+1}) \right\|_{L^{p_{n+1}} \left( B_1 \, , |x|^{-\gamma}  \right) }  \le D^{\frac {n+1} {p_n}} t^{- \frac{1}{p_n} } \left\| w(t_n) \right\|_{L^{p_n} \left( B_1 , |x|^{-\gamma} \right)} ,
	\end{equation}
	for a suitable constant $ D>0 $ depending only on $ N,m,\gamma,\underline{C},\mathcal{C} $, that may change from line to line. At this stage, we can iterate \eqref{fourth-step-smooth} in a classical way (see e.g.~the proof of \cite[Theorem 5.10]{GM13}), to finally obtain
	$$
	\left\| w(t) \right\|_{L^{\infty} \left( B_1  \right) }  \le D \, t^{- \frac{N-\gamma}{2(2-\gamma)} } \left\| w(0) \right\|_{L^2 \left( B_1 , |x|^{-\gamma} \right)} ,
	$$
	whence \eqref{conv-unif-w} follows. Note that the rigorous justification of the above computations can be shown via usual approximations (we refer again to \cite[Chapter 5]{Vaz} or \cite[Section 3]{GMPo}).
	
	In order to handle the case $ \ell \ge \mathcal{C} $, first of all we notice that the boundary condition in \eqref{lowbar-vt} is negative, therefore $ V \le 0 $ and so $ v_t \le 0 $ upon approximation. Then, instead of the bound $ 0 \le w \le \mathcal{C}-\ell $ we have that $ \mathcal{C}-\ell \le w \le 0  $, so \eqref{nonlin-below} is \emph{a fortiori} satisfied up to taking moduli. Finally, to obtain~\eqref{first-step-smooth}, we multiply the differential equation by $-p|w|^{p-1} $ in the place of $p w^{p-1} $, and from~\eqref{first-step-smooth} on, the rest of the proof is exactly the same.
\end{proof}

\begin{proof}[Proof of Proposition \ref{datum}]
Let us start from \eqref{conv-resc-datum}; for convenience we will prove its equivalent version \eqref{conv-resc-datum-bis}. Thanks to \eqref{EE1}, for every fixed $ \varepsilon>0 $ there exists $ R_\varepsilon>0 $ such that
	\begin{equation}\label{eq101}
		\left| u_0(x) - b \left| x \right|^{-\alpha} \right| \le \varepsilon \left| x \right|^{-\alpha} \qquad \text{for a.e. } x \in B_{R_\varepsilon}^c \, .
	\end{equation}
	Hence, we can bound the integral in \eqref{conv-resc-datum-bis} as follows:
	\begin{equation}\label{eq102}
		\begin{aligned}
			& \, \int_{\mathbb{R}^N} \left| u_0\!\left( y \right) - b \left| y \right|^{-\alpha} \right| \Phi\!\left( \tfrac{y}{\xi} \right) \rho(y) \, dy \\
			= & \, \int_{B_{R_\varepsilon}} \left| u_0\!\left( y \right) - b \left| y \right|^{-\alpha} \right| \Phi\!\left( \tfrac{y}{\xi} \right) \rho(y) \, dy + \int_{B_{R_\varepsilon}^c} \left| u_0\!\left( y \right) - b \left| y \right|^{-\alpha} \right| \Phi\!\left( \tfrac{y}{\xi} \right) \rho(y) \, dy \\
			\le & \left\| \Phi \right\|_{L^\infty\left( \mathbb{R}^N \right)} \int_{B_{R_\varepsilon}} \left| u_0\!\left( y \right) - b \left| y \right|^{-\alpha} \right|  \rho(y) \, dy + \varepsilon \, \overline{C} \, \int_{B_{R_\varepsilon}^c}  \left| y \right|^{-\alpha}  \Phi\!\left( \tfrac{y}{\xi} \right) \left| y \right|^{-\gamma} dy \, ,
		\end{aligned}
	\end{equation}	
	where in the last passage we also took advantage of the boundedness of $ \Phi $ and \eqref{weight-cond}. Scaling back to the $x$ variable, we obtain
	\begin{equation}\label{eq103}
		\int_{B_{R_\varepsilon}^c}  \left| y \right|^{-\alpha}  \Phi\!\left( \tfrac{y}{\xi} \right) \left| y \right|^{-\gamma} dy \le \int_{\mathbb{R}^N}  \left| y \right|^{-\alpha}  \Phi\!\left( \tfrac{y}{\xi} \right) \left| y \right|^{-\gamma} dy = \xi^{N-\alpha-\gamma} \int_{\mathbb{R}^N} \frac{\Phi(x)}{|x|^{\alpha + \gamma}} \, dx \, .
	\end{equation}
	Therefore, by combining \eqref{eq102} and \eqref{eq103}, we end up with
	$$
	\begin{aligned}
		& \, \xi^{\alpha+\gamma-N} \int_{\mathbb{R}^N} \left| u_0\!\left( y \right) - b \left| y \right|^{-\alpha} \right| \Phi\!\left( \tfrac{y}{\xi} \right) \rho(y) \, dy \\
		\le & \, \xi^{\alpha+\gamma-N} \left\| \Phi \right\|_{L^\infty\left( \mathbb{R}^N \right)} \int_{B_{R_\varepsilon}} \left| u_0\!\left( y \right) - b \left| y \right|^{-\alpha} \right|  \rho(y) \, dy + \varepsilon \, \overline{C} \int_{\mathbb{R}^N} \frac{\Phi(x)}{|x|^{\alpha + \gamma}} \, dx \, ,
	\end{aligned}
	$$
	which, due to $ \alpha < N-\gamma $, entails
	\begin{equation*}\label{eq104}
		\limsup_{\xi \to +\infty} \, \xi^{\alpha+\gamma-N} \int_{\mathbb{R}^N} \left| u_0\!\left( y \right) - b \left| y \right|^{-\alpha} \right| \Phi\!\left( \tfrac{y}{\xi} \right) \rho(y) \, dy \le \varepsilon \, \overline{C} \int_{\mathbb{R}^N} \frac{\Phi(x)}{|x|^{\alpha + \gamma}} \, dx \, .
	\end{equation*}
	Given the arbitrariness of $ \varepsilon $ and \eqref{A3}, it follows that \eqref{conv-resc-datum-bis} holds.
	
To continue the proof let us now focus on \eqref{second-req-u0-re}, which clearly implies \eqref{second-req-u0}. As in the proof of Proposition \ref{p1}, we let $ R_\gamma > 1 $ be as in \eqref{est-R-gamma}. It can be easily checked that the quantity
	\begin{equation}\label{varnorm}
		\int_{B_{R_\gamma}} \left| u_{0k} - b|x|^{-\alpha} \right| \,  \rho_k \, dx \, + \sup_{\substack{ R \ge R_\gamma \\[0.4mm] z_R \in \partial B_{R} }} \frac{ \int_{ B_{R^{\gamma / 2}} (z_R) } \left| u_{0k} - b|x|^{-\alpha} \right| \, \rho_k \, dx}{R^{\frac{\gamma(N-2)}{2} } }
	\end{equation}
is equivalent to the norm $\norm{u_{0k}-b|x|^{-\alpha}}_{0,\rho_k}$ (with constants independent of $k$), so it is enough to prove that both terms in \eqref{varnorm} converge to $0$ as $k\to\infty$. The starting point is the rescaled version of \eqref{eq101}
	\begin{equation}\label{eq101-resc}
		\left| u_{0k}(x) - b \left| x \right|^{-\alpha} \right| \le \varepsilon \left| x \right|^{-\alpha} \qquad \text{for a.e. } x \in B_{R_\varepsilon/\xi_k}^c \, .
	\end{equation}
Let us  take $k$ large enough so that $R_\varepsilon/\xi_k< R_\gamma/2$.
	Therefore, applying \eqref{eq101-resc}, \eqref{weight-cond-scaled} and  undoing the scaling, we obtain
	\begin{align*}
		\int_{B_{R_\gamma}} \left| u_{0k} - b|x|^{-\alpha} \right| \,  \rho_k \, dx & \leq \varepsilon\,\overline{C}\int_{B_{R_\gamma}\setminus B_{R_\varepsilon/\xi_k}} |x|^{-\alpha-\gamma}\, dx + \int_{B_{R_\varepsilon/\xi_k}} \left| u_{0k} - b|x|^{-\alpha} \right| \,  \rho_k \, dx \\
		& \le \varepsilon\, C \,  R_\gamma^{N-\alpha-\gamma} +  \xi_k^{\alpha+\gamma-N}\int_{B_{R_\varepsilon}} \left| u_{0} - b|x|^{-\alpha} \right| \,  \rho \, dx \, ,
	\end{align*}
	where $C$ is a positive constant independent of $k$ that may vary from line to line. Taking first $k\to\infty$ and then $\varepsilon\to0$, recalling that $ \alpha < N-\gamma $, proves that the first term of \eqref{varnorm} converges to $0$ as $k\to\infty$. For the second term of \eqref{varnorm}, still under $R_\varepsilon/\xi_k< R_{\gamma}/2$, we have
	\[
	B_{R^{\gamma / 2}} (z_R) \cap B_{R_\varepsilon/\xi_k} = \emptyset \qquad \forall R\geq R_\gamma \, ,
	\]
	so we always work in the set where \eqref{eq101-resc} holds. Therefore, once again applying \eqref{eq101-resc}, we conclude
	\begin{align*}
		\int_{ B_{R^{\gamma / 2}} (z_R) } \left| u_{0k} - b|x|^{-\alpha} \right| \, \rho_k \, dx & \leq \varepsilon\,\overline{C} \int_{ B_{R^{\gamma / 2}} (z_R) } |x|^{-\alpha-\gamma}\,dx \leq \varepsilon\,C\, R^{-\alpha+{\frac{\gamma(N-2)}{2} }}
	\end{align*}
for $k$ large enough (depending on $\varepsilon$). The fact that the second term of \eqref{varnorm} vanishes now follows simply by applying this estimate, taking $ k \to \infty $ and then $\varepsilon\to0$.
\end{proof}

\normalcolor

\noindent \textbf{Acknowledgments.} M.M.~and T.P.~were supported by project no.~2022SLTHCE ``Geometric-Analytic Methods for PDEs and Applications (GAMPA)'' -- funded by European Union -- Next Generation EU within the PRIN 2022 program (D.D. 104 - 02/02/2022 Ministero dell’Università e della Ricerca, Italy). They both thank the INdAM-GNAMPA group (Italy). F.Q.~was supported by grants PID2020-116949GB-I00, CEX2019-000904-S,  and RED2022-134784-T, all of them funded by MCIN/AEI/10.13039/501100011033, and by the Madrid Government (Comunidad de Madrid – Spain) under the multiannual Agreement with UAM in the line for the Excellence of the University Research Staff in the context of the V PRICIT (Regional Programme of Research and Technological Innovation).

\end{document}